\documentclass[11pt]{article}
\usepackage[utf8]{inputenc}

\usepackage{graphicx}
\graphicspath{ {./images/} }

\usepackage{standalone}

\usepackage{mathUtil}

\usepackage{fullpage}

\usepackage{comment}
\usepackage{indentfirst}

\usepackage{xcolor}

\usepackage[backend=biber,style=numeric,maxnames=50]{biblatex}
\usepackage{hyperref}
\hypersetup{
    colorlinks=true,
    linkcolor=blue,
    filecolor=magenta,
    urlcolor=cyan,
    pdftitle={Polyatomic Blok-Esakia},
    pdfpagemode=FullScreen,
}
\title{Polyatomic Logics and Generalised Blok-Esakia Theory}
\author{Rodrigo Nicolau Almeida}

\begin{document}

\maketitle

\begin{abstract}
    This paper presents a novel concept of a \textit{Polyatomic Logic} and initiates its systematic study. This approach, inspired by Inquisitive semantics, is obtained by taking a variant of a given logic, obtained by looking at the fragment covered by a \textit{selector term}. We introduce an algebraic semantics for these logics and prove algebraic completeness. These logics are then related to \textit{translations}, through the introduction of a number of classes of translations involving selector terms, which are noted to be ubiquitous in algebraic logic. In this setting, we also introduce a generalised Blok-Esakia theory which can be developed for special classes of translations. We conclude by showing some systematic connections between the theory of Polyatomic Logics and the general Blok-Esakia theory for a wide class of interesting translations.
\end{abstract}

\section{Introduction}

When working with multiple logical systems, a natural question one might ask is whether such systems are related by translations. Many examples are well-known in the literature:
\begin{itemize}
    \item \textit{Double Negation Translations}, introduced by Kolmogorov \cite{kolmogorovexcludedmidle}, Gödel \cite{davis_1990} and Gentzen \cite{Gentzen1936}, which we refer to as the Kolmogorov-Gödel-Gentzen translation (KGG) from classical logic $\mathsf{CPC}$ into intuitionistic logic, $\mathsf{IPC}$.
    \item The \textit{Gödel-McKinsey-Tarski translation} (GMT) \cite{godelintuition}, from $\mathsf{IPC}$ into $\mathsf{S4}$ modal logic.
    \item The \textit{Girard call-by-value translation} \cite{GIRARD19871}, from $\mathsf{IPC}$ to Intutitionistic Linear logic.
    \item The \textit{Goldblatt translation} \cite{Goldblatt1974}, from orthologic into $\mathsf{KTB}$ modal logic.
\end{itemize}

Such translations have found applications in logic, mathematics, philosophy and linguistics. This is often done by a semantic analysis of the translation, using some specific class of models of the logical systems. Taking the KGG translation as an example, we start with the algebraic models of intuitionistic logic -- Heyting algebras -- and the models of classical logic -- Boolean algebras. Given a Heyting algebra $\mathcal{H}$, we can consider the set of \textit{regular elements}, i.e.,  elements such that $a=\neg\neg a$,
\begin{equation*}
    \mathcal{H}^{\neg} = \{a\in H : a=\neg\neg a\}.
\end{equation*}
This set can be equipped with a Boolean algebra structure, by taking the meet, implication and bounds of $\mathcal{H}$, as well as the operation $x\veeio y \coloneqq \neg\neg (x\vee y)$. The key idea in applications of such a translation is then to relate these two structures -- the original Heyting algebra $\mathcal{H}$ and the Boolean algebra of regular elements $\mathcal{H}^{\neg}$ -- whether for mathematical (see e.g. \cite{Banaschewski1996}) or linguistic purposes (see e.g. \cite{Ciardelli2010}). For instance in the case of \textit{inquisitive logic}, studied by  
Ciardelli et al.,  \cite{Ciardelli2010}, the goal was to provide logical meaning to inquisitive statements. Following the algebraic approach taken in  \cite{Bezhanishvili2019_grilletti_holliday}, the models of such a logic can be taken to be Heyting algebras $\mathcal{H}$ equipped with a valuation $v$ taking values in $\mathcal{H}^{\neg}$. The idea is then that regular elements represent ``ultimately classical" facts, whilst the surrounding intuitionistic structure represents the inquisitive state of knowledge. This has grown into a rich linguistic, philosophical and logical literature, spawning numerous variations (see \cite{Ciardelli2018-ql} for an extended discussion). As such an example illustrates, the semantic meaning of translations can be exploited to find new logical tools for various purposes. This motivates a deeper understanding of translations as mathematical objects, as well of the logical frameworks which, like inquisitive logic, make productive use of such translations.

This paper, based on a recent masters thesis \cite{minhatesedemestrado}, introduces two tools aimed at this direction of inquiry -- the notion of a \textit{Polyatomic Logic}, and the concept of a \textit{Generalised Blok-Esakia Theory}. Each of these concepts generalises a known line of analysis in the literature, but they are presented here in greater generality. Additionally, we highlight a series of connections between the key ideas of the theory of Polyatomic Logic and those of general Blok-Esakia theory, which as far as we are aware are discussed here for the first time.

Polyatomic logics generalise the notion of a \textit{DNA-logic}, introduced in \cite{Bezhanishvili2019_grilletti_holliday,bezhanishvili_grilletti_quadrellaro_2021}, which in turn generalise the above mentioned inquisitive propositional logic \cite{Ciardelli2010}. The idea is to attach to each superintuitionistic logic (i.e., an s.i. axiomatic extension) $L$, a collection of formulas denoted $L^{\neg}$, defined as
\begin{equation*}
    \phi\in L^{\neg} \iff \phi[\neg \overline{p}/\overline{p}]\in L
\end{equation*}
where $\phi[\neg\overline{p}/\overline{p}]$ is the result of substituting each proposition letter $q$ for $\neg q$ on $\phi$. Such a collection of formulas is called the \textit{DNA-variant} of $L$. Such ``logics" have the feature that they are not closed under uniform substitution. Semantically, this amounts to considering only \textit{regularly generated Heyting algebras of $L$}, i.e., those $\mathcal{H}\vDash L$ such that $\mathcal{H}$ is isomorphic to $\langle \mathcal{H}^{\neg}\rangle$, the subalgebra of $\mathcal{H}$ generated by the regular elements. 

The analysis of such logics has contributed to the understanding  of propositional inquisitive logics \cite{bezhanishvili_grilletti_quadrellaro_2021}, such as a full description of the lattice of extensions of such logics. However, the key features enabling the above definitions are far from exclusive to the KGG situation (a fact already noted and exploited in some generalisations by Grilletti and Quadrellaro \cite{Grilletti2022-oh}, and extensively discussed by Nakov and Quadrellaro \cite{quadrellaronakov}). As an example, let us quickly look at the GMT translation: here, given an $\mathbf{S4}$ algebra $\mathcal{B}$, an element $a$ is called \textit{open} if $a=\Box a$. Then we consider:
\begin{equation*}
    \mathcal{B}_{\Box} \coloneqq \{ a\in B : a=\Box a\}.
\end{equation*}
Equipping this set with the induced operations from the distributive lattice as well as the implication $a\Rightarrow b \coloneqq \Box (\neg a\vee b)$, we obtain a Heyting algebra. In analogy with the situation above, given a normal modal logic $M$, extending $\mathsf{S4}$, we can consider the $\Box$-variant of $M$, $M^{\Box}$, defined as:
\begin{equation*}
    \phi\in M^{\Box} \iff \phi[\Box\overline{p}/\overline{p}]\in M,
\end{equation*}
and consider as algebraic models only the openly generated $S4$-algebras.

 Translations which act by affixing a given term to certain formulas are ubiquitous: from orthomodular lattices to residuated ortholattices \cite{Fussner2021}; reflexivisation translations \cite[Chapter 4]{Chagrov1997-cr} in modal logic; as well as numerous substructural logics into modal systems \cite{Hartonas2020-kn}. Hence, a natural question arises as to what conditions are needed to ensure that one can construct such a logic, and what properties from the theory outlined in \cite{bezhanishvili_grilletti_quadrellaro_2021} can be carried out in general. This is what is achieved by the concept of a Polyatomic Logic, which generalises the above situation to any logic possessing an idempotent term -- deemed a ``selector term" -- analogous to $\neg\neg$ or $\Box$.

 Our second line of analysis concerns the study of translations, from the point of view of the lattices of extensions of the logical systems at play -- something we dub \textit{Generalised Blok-Esakia theory}. This is modelled after the GMT translation, and the notion of a \textit{modal companion} of an s.i. logic $L$, i.e., a normal extension $M$ of $\mathsf{S4}$ such that for each $\phi$ in the language of intuitionistic logic:
 \begin{equation*}
     \phi\in L \iff GMT(\phi)\in M.
 \end{equation*}

The namesake \textit{Blok-Esakia isomorphism}, established independently by Wim Blok and Leo Esakia (\cite{Chagrov1997-cr,esakiapaperatconference,Blok1976VarietiesOI}, see also \cite{antoniocleani,stronkowskiblokesakia,Wolter2014}), establishes the fact that there is an isomorphism between the lattice of normal extensions of $\mathsf{S4.Grz}$ and the lattice $\Lambda(\mathsf{IPC})$ of s.i. logics, sending each s.i. logic to its greatest modal companion. This has been generalised in many directions, and with respect to various different systems (see \cite{Chagrov1997-cr,Wolter2014} for extended discussions). However, as it becomes clear in the analysis in Section \ref{Section: General Blok-Esakia Theory},  many of the properties of $\mathsf{IPC}$ and $\mathsf{S4.Grz}$ used to establish this isomorphism, depend solely on a categorical adjunction between the category of Heyting algebras and the category of Grz-algebras.

Starting from the work of Moraschini \cite{Moraschini2018}, which identifies translations between logical systems and adjunctions between algebraic categories, we study of various kinds of translations, obtained by requiring the adjunction to have more properties which make it resemble an adjoint equivalence, such as a fully faithful left adjoint, or a full right adjoint. We outline the natural generalisations of the concepts present in the classical Blok-Esakia isomorphism, and show that they hold in the presence of such strong properties on the adjunction. We also give various examples of translations which have some, all, or none of these properties.

Having introduced both of these tools as generalisations of known concepts, we then highlight the connection between Polyatomic Logics and Blok-Esakia theory. Key amongst these is the notion of 
$\mathsf{PAt}$-\textit{Blok Esakia isomorphism}, which holds when the lattice of logics of the system being translated is isomorphic to the lattice of Polyatomic Logics of the target system, with respect to the translation. We also show that there is a correspondence betwen concepts which occur naturally in Polyatomic logics -- such as $\mathsf{PAt}$-minimal and maximal logics -- and Blok-Esakia concepts -- such as the greatest and least modal companion. We hence show that the schematic fragment, generalised from DNA-logic, can be seen as the greatest modal companion. This reanalyses ``Blok's lemma" -- which establishes that every $\mathsf{S4.Grz}$ logic is sound and complete with respect to its openly generated algebras, and which establishes the Blok-Esakia isomorphism in full -- and highlights which properties are special to GMT. For instance, the schematic fragment of inquisitive propositional logic is the well-known Medvedev Logic, a logic which is not known to have a recursive axiomatisation.

The paper is organised as follows: in Section \ref{Section 2} we discuss the basic preliminaries necessary to proceed with the analysis. In Section \ref{Section: Basic Theory and Completeness of PA logic} we define Polyatomic logics, and develop its basic theory, proving algebraic completeness of these logics with respect to specific classes of quasivarieties. In Section \ref{Section: Translations and Polyatomic Logics} we recap the approach of Moraschini to translations as adjunctions, and introduce the concept of selective translation, as well as some strengthenings, such as strongly selective and sober translations; in Section \ref{Section: General Blok-Esakia Theory} we outline a general Blok-Esakia theory, and discuss some of the preservations properties which can be derived at a great level of generality; finally in Section \ref{Connecting Blok-Esakia and Polyatomic Logics} we explore the connections between this theory and Polyatomic Logics, as detailed above. We conclude in Section \ref{Section: Conclusion} by outlining some questions left open by the present work.

\section{Preliminaries}\label{Section 2}

In this section we fix notation, and discuss some essential concepts that will be needed throughout the paper. We assume throughout basic familiarity with universal algebra, see for instance \cite{BurrisSankappanavar} for the basic theory of varieties, and \cite{Gorbunov1998-sh} for the basic theory of quasivarieties. For more information on abstract algebraic logic see \cite{Font2016-dk}. We also assume familiarity with the basic theory of Heyting algebras and Boolean algebras (see for example \cite{Davey2002-lr}).

\subsection{Algebraic Semantics}

Throughout we work with sets $\fancyL$ of function symbols, with an arity operation, which we refer to as \textit{algebraic languages} or \textit{types of algebras}. Given an algebraic language $\fancyL$, and a set $X$, we write $Tm(\fancyL,X)$ for the set of terms built up with the variables in $X$. We sometimes refer to terms as ``formulas", especially in the context of algebras of logic. We write $\terms(\fancyL,X)$ for the corresponding absolutely free algebra. We write Greek letters, $\lambda,\gamma,\iota,...$ for terms of the language. When we want to emphasise the variables in the term we write $\lambda(x_{1},...,x_{n})$, or sometimes $\lambda(\overline{x})$. We denote by $Eq(\fancyL,X)$ the set of equations built from this language over $X$. These are formally pairs of elements from $Tm(\fancyL,X)$, and we usually write
\begin{equation*}
    \lambda\approx\gamma,
\end{equation*}
to denote some equation in the language. When the language $\fancyL$ is clear from context, we omit it from the above defintions.

\begin{definition}
Let $\fancyL$ be an algebraic language. A \textit{quasi-equation} is an expression $\Phi$ of the form:
\begin{equation*}
    \Phi\coloneqq \lambda_{1}\approx \gamma_{1} \ \&  \ ... \ \& \ \lambda_{n}\approx \gamma_{n} \rightarrow \lambda\approx \gamma
\end{equation*}
where $\lambda_{i},\gamma_{i},\lambda,\gamma$ are terms in the language.
\end{definition}

Throughout we use $\mathcal{A}$ to denote an algebra, and $A$ to mean the carrier of the algebra. Given an algebra $\mathcal{A}$ of type $\fancyL$ and a term $\lambda$, we write $\lambda^{\mathcal{A}}$ to denote the interpretation of $\lambda$ on $\mathcal{A}$. Often we will omit $\mathcal{A}$ when this is clear from context.

\begin{definition}
Let $\mathcal{A}$ be an algebra of type $\fancyL$, and let $X$ be a set, which we think of as a set of variables. We say that a map $v:X\to A$ is an \textit{assignment of variables} on $\mathcal{A}$. Given an assignment $v$, we extend this to a \textit{valuation} $\overline{v}:Tm(\fancyL,X)\to A$ recursively, by defining:
\begin{itemize}
    \item $\overline{v}(p)=v(p)$
    \item For each $n$-ary operation symbol $f(x_{1},...,x_{n})\in \fancyL$, and $\lambda_{1},...,\lambda_{n}\in Tm(\fancyL,X)$ such that $\overline{v}(\lambda_{i})$ has been defined, we let:
    \begin{equation*}
        \overline{v}(f(\lambda_{1},...,\lambda_{n}))=f^{\mathcal{A}}(\overline{v}(\lambda_{1}),...,\overline{v}(\lambda_{n}))
    \end{equation*}
\end{itemize}
We refer to a pair $(\mathcal{A},v)$ as a \textit{model}. Given an equation $\lambda\approx\gamma\in Eq(\fancyL,X)$, we write $\mathcal{A},v\vDash \lambda\approx\gamma$ to mean that $\overline{v}(\lambda)=\overline{v}(\gamma)$. We write $\mathcal{A}\vDash \lambda\approx\gamma$ if for each model $(\mathcal{A},v)$, we have that $\mathcal{A},v\vDash \lambda\approx\gamma$.

Similarly, given a quasi-equation $\lambda_{1}\approx \gamma_{1} \ \& \ ... \ \& \ \lambda_{n}\approx \gamma_{n} \rightarrow \lambda\approx \gamma$, we write $$\mathcal{A},v\vDash  \lambda_{1}\approx \gamma_{1} \ \&  \ ... \ \& \ \lambda_{n}\approx \gamma_{n} \rightarrow \lambda\approx \gamma$$ to mean that, if $\overline{v}(\lambda_{i})=\overline{v}(\gamma_{i})$ for each $1\leq i\leq n$, then $\overline{v}(\lambda)=\overline{v}(\gamma)$. We write $\mathcal{A}\vDash \lambda_{1}\approx \gamma_{1} \ \&  \ ... \ \& \ \lambda_{n}\approx \gamma_{n} \rightarrow \lambda\approx \gamma$ to mean that $\mathcal{A},v\vDash \lambda_{1}\approx \gamma_{1} \ \&  \ ... \ \& \ \lambda_{n}\approx \gamma_{n} \rightarrow \lambda\approx \gamma$ for each assignment of variables $v$.
\end{definition}

If $\mathcal{A}$ is a structure with a distinguished element $1$ -- as will be the case in many cases throughout this article -- we adopt a particular shorthand for models: given a valuation $v$ on $\mathcal{A}$, and $\lambda\in \fancyL$ we write $\mathcal{A},v\vDash \lambda$ to mean that $\mathcal{A},v\vDash \lambda\approx 1$.

\begin{definition}
Let $\mathbf{K}$ be a class of algebras. We say that $\bf{K}$ is an \textit{equational class} (resp. quasi-equational class) if there exists some set of equations $\Gamma$ (resp., quasi-equations) such that for each algebra $\mathcal{A}$, $\mathcal{A}\in \mathbf{K}$ if and only if $\mathcal{A}\vDash \Phi$ for each $\Phi\in \Gamma$.
\end{definition}

Throughout, if $\mathbf{K}$ is a class of similar algebras, we denote by $\fancyL_{\mathbf{K}}$ the algebraic language of $\mathbf{K}$; when the subscript is understood we drop it. 

We recall the following monotone and idempotent operators on classes of algebras:
\begin{enumerate}
    \item $\mathbb{I}(\mathbf{K})$ - isomorphic copies of algebras in $\mathbf{K}$;
    \item $\mathbb{H}(\mathbf{K})$ - homomorphic images of algebras in $\mathbf{K}$;
    \item $\mathbb{S}(\mathbf{K})$ - subalgebras of algebras in $\mathbf{K}$;
    \item $\mathbb{P}(\mathbf{K})$ - direct products of algebras in $\mathbf{K}$;
    \item $\mathbb{P}_{R}(\mathbf{K})$ - reduced products of algebras in $\mathbf{K}$.
\end{enumerate}

\begin{definition}
Let $\mathbf{K}$ be a class of algebras. Then we say that $\mathbf{K}$ is:
\begin{enumerate}
    \item A \textit{variety} if it is closed under subalgebras, homomorphic images and direct products;
    \item A \textit{quasi-variety} if it is closed under isomorphisms, subalgebras, and reduced products.
\end{enumerate}
\end{definition}

It is clear to see that quasivarieties and varieties form complete lattices under inclusion. Given a quasivariety $\bf{K}$, we write $\Xi(\bf{K})$ for the lattice of subquasivarieties of $\bf{K}$.

The following theorems were proved by Tarski and Maltsev (see respectively \cite[Chapter 2, Theorem 9.5]{BurrisSankappanavar} and \cite[Chapter 5, Theorem 2.25]{BurrisSankappanavar}):

\begin{theorem}
For every class of similar algebras $\mathbf{K}$, we have:
\begin{itemize}
    \item $\mathbf{K}$ is a variety if and only if $\mathbf{K}=\mathbb{HSP}(\mathbf{K}')$ for some class $\mathbf{K}'$ of similar algebras.
    \item $\mathbf{K}$ is a quasivariety if and only if $\mathbf{K}=\mathbb{ISP}_{R}(\mathbf{K}')$ for some class $\mathbf{K}'$ of similar algebras.
\end{itemize}
\end{theorem}

We also recall the following theorems, due respectively to Birkhoff and Mal'tsev, which we refer to throughout as ``algebraic completeness" results:

\begin{theorem}
Let $\mathbf{K}$ be a class of algebras. Then:
\begin{itemize}
    \item $\mathbf{K}$ is an equational class if and only if $\mathbf{K}$ is a variety;
    \item $\mathbf{K}$ is a quasi-equational class if and only if $\mathbf{K}$ is a quasivariety.
\end{itemize}
\end{theorem}
\begin{proof}
See \cite[Chapter 2,Theorem 11.9; Chapter 5,Theorem 2.25]{BurrisSankappanavar}.\end{proof}

We note that in this paper, quasivarieties will take a more prominent role than varieties, given our focus on logics, as discussed in the next section.

\subsection{Logics and Consequence Relations}\label{Subsection: Logics and Cons. Relations}

\begin{definition}
A \textit{consequence relation} on a set $C$ is a relation $\vdash \ \subseteq \mathbb{P}(C)\times C$ such that for each $W\cup V\cup \{a\}\subseteq C$:
\begin{itemize}
    \item (Reflexivity) If $a\in W$ then $(W,a)\in \  \vdash$
    \item (Cut) If $(W,v)\in \  \vdash$ for every $v\in V$, and $(V,a)\in \ \vdash$, then $(W,a)\in \ \vdash$
\end{itemize}
We furthermore say that $\vdash$ is \textit{finitary} if whenever $W\cup \{a\}\subseteq C$
\begin{equation*}
    \text{ if } (W,a)\in \  \vdash \text{ then there exists finite $V\subseteq C$, such that } (V,a)\in \ \vdash 
\end{equation*}
We write $W\vdash a$ to mean $(W,a)\in {\vdash}$; we write $\vdash a$ to mean $(\emptyset,a)\in {\vdash}$. Given $W,V\subseteq C$, we write $W\vdash V$ to mean that $W\vdash v$ for each $v\in V$.
\end{definition}

In this paper, all consequence relations are assumed to be finitary. This is not necessary for most results, but it makes the presentation simpler, and avoids complications with respect to infinitary axiomatisations.

As is well-known, consequence relations form a complete lattice, where the arbitrary meet is intersection, and the arbitrary join is obtained by closing the union under being a consequence relation. This entitles us to the following definition:

\begin{definition}
Let $C$ be a set, and $S\subseteq \mathbb{P}(C)\times C$. We denote by ${\vdash_{s}}\subseteq \mathbb{P}(X)\times X$ the smallest consequence relation on $X$ such that $S\subseteq {\vdash_{s}}$.
\end{definition}

A special kind of consequence relation we will be interested are \textit{logics}.

\begin{definition}
Given an algebraic language $\fancyL$ and a set $X$, a given consequence relation ${\vdash} \subseteq \mathbb{P}(Tm_{\fancyL}(X))\times Tm_{\fancyL}(X)$ is called a \textit{logic} if $\vdash$ is substitution invariant: for every homomorphism $\nu: \terms_{\fancyL}(X)\to \terms_{\fancyL}(X)$, and every $\Gamma\cup\{\phi\}\subseteq Tm(\fancyL,X)$:
\begin{equation*}
    \Gamma\vdash \phi \text{ implies } \nu[\Gamma]\vdash \nu(\phi).
\end{equation*}
\end{definition}

Similar to the above, we also have that given a logic $\vdash$, the collection $\Lambda(\vdash)$ of finitary extensions of $\vdash$ forms a complete lattice with meet as intersection. And similar to before, if $S\subseteq \mathbb{P}(Tm_{\fancyL}(X))\times \terms_{\fancyL}(X)$, we denote by $\mathsf{Log}(S)$ the smallest logic generated by $S$. It is then clear that given $(\vdash_{i})_{i\in I}$ a collection of finitary logics, then $\bigvee_{i\in I}\vdash_{i}=\mathsf{Log}(\bigcup_{i\in I}\vdash_{i})$. Moreover, we have the following:

\begin{lemma}\label{Explicit definition of supremum of logics}
If $(\vdash_{i})_{i\in I}$ is a collection of finitary logics, and $S=\bigcup_{i\in I}\vdash_{i}$, then:
\begin{equation*}
    \bigvee_{i\in I}\vdash_{i} \  = {\vdash_{S}}
\end{equation*}
\end{lemma}

Given a logic $\vdash$, and a collection $\Gamma$ of axioms, we  write
\begin{equation*}
    {\vdash} {\oplus} \Gamma,
\end{equation*}
to denote the smallest logic extending $\vdash$ which contains $\Gamma$.

We denote by $\mathsf{IPC},\mathsf{CPC},\mathsf{S4}$, the well-known logical systems. For each system $L$, we denote by $\Lambda(L)$ the lattice of extensions of this system. Exception is made for modal logical systems, where instead we denote this lattice by $\mathbf{NExt}(L)$, since we wish to emphasise the fact that  only \textit{normal} extensions of these systems are to be considered.

Unfortunately, in these contexts we are faced with an unfortunate ambiguity: the term ``logic" is often used to refer to what would by us be called \textit{axiomatic extensions} of these systems, i.e., extensions by axioms of the form $\vdash \phi$. This ambiguity is persistent and inevitable; we will therefore seek to always make clear which sense of the word is meant throughout the article, by adding ``in the sense of an axiomatic extension" to the term ``logic", where this is appropriate.

We now turn to a crucial example of a logic, coming from algebraic logic.

\begin{definition}
Let $\mathbf{K}$ be a class of algebras and $X$ a set of variables. We define the \textit{equational consequence relative to $\mathbf{K}$}, ${\vDash_{\bf{K}}}\subseteq \mathbb{P}(Eq(\fancyL_{\bf{K}},X))\times Eq(\fancyL_{\bf{K}},X)$ as follows if it satisfies the following properties: for $\Theta\subseteq Eq(\fancyL_{\bf{K}},X)$ and $\gamma\approx\delta\in Eq(X)$:
\begin{align*}
    \Theta\eqrelK \gamma\approx \delta \iff &\text{For every $\mathcal{A}$} \in \mathbf{K} \text{ and } v:X\to A \text{ a variable assignment }\\
    &\text{  if for all $\phi\approx \psi \in \Theta$  \ } \overline{v}(\phi)=\overline{v}(\psi), \text{ then } \overline{v}(\gamma)=\overline{v}(\delta)
\end{align*}
\end{definition}

The equational consequence relative to a class of algebras is very often used as the semantics for specific logics. This is boiled down to the property of having a completeness theorem:

\begin{definition}
Let $\fancyL$ be an algebraic language, $\vdash$ be a logic in this language, and $\mu(x)=\{\gamma_{i}(x)\approx \delta_{i}(x) : i\in I\}$ a set of equations in one variable. We say that a class $\mathbf{K}$ of algebras is a $\mu$-algebraic semantics for $\vdash$ if for all $\Gamma\cup\{\phi\}\subseteq Tm(\fancyL)$:
\begin{equation*}
    \Gamma\vdash \phi \iff \mu[\Gamma]\vDash_{\mathbf{K}} \mu(\phi)
\end{equation*}
where $\mu[\Gamma]=\{\gamma_{i}(\psi)\approx\delta_{i}(\psi) : i\in I, \psi\in \Gamma\}$. We say that $\mathbf{K}$ is an algebraic semantics for the logic $\vdash$, if it is a $\mu$-algebraic semantics for the logic $\vdash$, for some $\mu$.
\end{definition}

As is well-known by the work of Blok and Piggozzi, for some logics there is a canonical choice of algebraic semantics, namely the so-called ``equivalent algebraic semantics".

\begin{definition}\label{Algebraizable logic}
Let $\fancyL$ be an algebraic language, and $\vdash$ a logic for this language. We say that $\vdash$ is \textit{algebraizable} if there is a quasivariety $\mathbf{K}$, $\mu(x)$ a finite set of equations in one variable, and $\Delta(x,y)=\{\phi_{j}(x,y) : j\in J\}$ a finite set of formulas in two variables, such that for all formulas $\Gamma\cup\{\phi\}\subseteq Tm(\fancyL)$ and $\Theta\cup \{\gamma\approx\delta\}\subseteq Eq(\fancyL)$:
\begin{enumerate}
    \item $\Gamma\vdash \phi$ if and only if $\mu[\Gamma]\vDash_{\mathbf{K}}\mu(\phi)$
    \item $\Theta\vDash_{\mathbf{K}} \gamma\approx\delta$ if and only if $\Delta[\Theta]\vdash \Delta(\gamma,\delta)$
    \item $\phi\dashv \vdash \Delta[\mu(\phi)]$
    \item $\gamma\approx \delta \Dashv \vDash_{\mathbf{K}} \mu[\Delta(\gamma,\delta)]$
\end{enumerate}
where $\Delta(\varepsilon,\upsilon)=\{\phi_{j}(\varepsilon,\upsilon) : \phi_{j}\in \Delta(x,y)\}$ and $\Delta[\Theta]=\bigcup_{\varepsilon\approx\upsilon\in \Theta}\Delta(\varepsilon,\upsilon)$. In this case, $\mathbf{K}$ is said to be an \textit{equivalent algebraic semantics} for the logic $\vdash$.
\end{definition}

And indeed, we have:

\begin{theorem}\cite[Theorem 3.17]{Font2016-dk}
Each algebraizable logic has a unique equivalent algebraic semantics.
\end{theorem}

We remark that the map $\mu$ as given above has the following structural property: given any $\nu$ a substitution, and any set of formulas $\Gamma\subseteq Tm(X)$:
\begin{equation*}
    \mu[\nu[\Gamma]]=\{\lambda\approx\gamma \in \nu[\Gamma]\}=\{\nu(\lambda)\approx\nu(\gamma) : \lambda\approx\gamma\in \mu[\Gamma]\}=\nu[\mu[\Gamma]].
\end{equation*}

\subsection{Classical Translations}\label{Subsection: Classical Translation}

We recall here some classical translations which will serve us as motivation, and key examples, for the development of this theory. We also outline some basic facts about inquisitive logic, and their generalisation to DNA logics, which we will need moving forward. Throughout this section, fix a countably infinite set of propositions, $\prop$.

\begin{definition}\label{Double Negation Translation}
The \textit{Kolmogorov-Gödel-Gentzen Double Negation Translation} (KGG)\footnote{We note that this translation, as noted in the introduction, is originally due to Kolmogorov; however, his translation applies the $\neg\neg$-operation to every formula. The specific version outlined is due to Gödel, Gentzen and Glivenko.}, denoted $K$, maps $Tm_{\fancyL_{\mathsf{BA}}}(\mathsf{Prop})$ to $Tm_{\fancyL_{\mathsf{HA}}}(\mathsf{Prop})$ through the following assignment:
\begin{enumerate}
    \item For each $p\in \prop$, $K(p)=\neg\neg p$;
    \item $K(\top)=\top$ and $K(\bot)=\bot$;
    \item $K(\phi\ast \psi)=K(\phi)\ast K(\psi)$, $\ast \in \{\wedge,\rightarrow\}$;
    \item $K(\phi\vee \psi)=\neg\neg( K(\phi)\vee K(\psi))$.
\end{enumerate}
\end{definition}

As mentioned in the introduction, given a Heyting algebra $\mathcal{H}$ we write
\begin{equation*}
    \mathcal{H}^{\neg\neg}=\{a\in H : a=\neg\neg a\}
\end{equation*}
for the Boolean algebra of \textit{regular elements}.

\begin{definition}\label{Heyting algebras and S4}
The \textit{Gödel-McKinsey-Tarski} (GMT) translation maps the set $Tm_{\fancyL_{\mathbf{HA}}}(\prop)$ to $Tm_{\fancyL_{\mathbf{S4}}}(\prop)$ through the following assignment:
\begin{enumerate}
\item For each $p\in \prop$, $GMT(p)=\Box p$;
\item $GMT(\top)=\top$ and  $GMT(\bot)=\bot$;
    \item $GMT(\phi\ast \psi)=GMT(\phi)\ast GMT(\psi)$, $\ast \in \{\wedge,\vee\}$;
    \item $GMT(\phi\rightarrow \psi)=\Box (\neg GMT(\phi)\vee GMT(\psi))$.
\end{enumerate}
\end{definition}

Given an S4-algebra $\mathcal{B}=(B,\Box)$ we write
\begin{equation*}
    \mathcal{B}_{\Box}=\{a\in B : a=\Box a\}
\end{equation*}
for the set of \textit{open elements}. As shown by Godel, this forms a Heyting algebra with the induced operations, $\hat{0}=0$ and $\hat{1}=1$, $a\hat{\wedge} b = a \wedge b$,  $a\hat{\vee }b= a\vee b$ and $a\hat{\hookrightarrow} b = \Box (\neg a \vee b)$.

We recall that an \textit{ortholattice} is an algebra, in the language $\fancyL_{\mathsf{BA}}$, satisfying the same axioms as a Boolean algebra except possibly distributivity. To avoid confusion, though it is the same algebraic language, we write $\fancyL_{\mathsf{Ort}}$ for the type of these algebras. We also recall the reader that a $\mathsf{KTB}$-algebra is a normal modal algebra satisfying the axioms
\begin{equation*}
    \Box p \rightarrow p \text{ and } p\rightarrow \Box\dia p.
\end{equation*}

The following was introduced in \cite{Goldblatt1974} by Goldblatt:

\begin{definition}
The \textit{Goldblatt translation} maps the set $Tm_{\fancyL_{\mathsf{Ort}}}(\prop)$ to $Tm_{\fancyL_{\mathsf{KTB}}}(\prop)$ through the following assignment:
\begin{enumerate}
    \item For each $p\in \prop$, $G(p)=\Box\dia p$;
    \item $G(\top)=\top$ and $G(\bot)=\bot$;
    \item $G(\psi \wedge \phi)=G(\psi)\wedge G(\phi)$;
    \item $G(\phi^{\bot})=\Box\neg G(\phi)$.
\end{enumerate}
\end{definition}

As in previous cases, given a $\mathsf{KTB}$-algebra $\mathcal{B}$, we can consider the set
\begin{equation*}
    \mathcal{B}_{\Box\dia}=\{a\in B : a=\Box\dia a\},
\end{equation*}
called the set of $\bot$-\textit{regular elements}. This can be equipped with the following operations: $a\wedgeio b=a\wedge b$ and $a^{\bot}=\Box \neg a$. It was shown by Goldblatt that this forms an ortholattice.

The three translations just introduced were all proven to be sound and faithful with respect to appropriate derivation systems: $\mathsf{CPC}$ and $\mathsf{IPC}$ for $K$; $\mathsf{IPC}$ and $\mathsf{S4}$ for $GMT$; and \textit{orthologic} \cite{Goldblatt1974} and $\mathsf{KTB}$. This was established by the authors in \cite{kolmogorovexcludedmidle,Goldblatt1974,davis_1990}. We isolate here a particular fact which will interest us later, and which establishes, in all cases, one half of the correctness of the translation (for instance, it shows that if $K(\phi)\notin \mathsf{IPC}$, then $\phi\notin \mathsf{CPC}$, and so on for each system).

\begin{theorem}\label{Translation theorem for the double negation translation}
For each formula $\phi\in Tm(\fancyL_{\mathsf{BA}})$, and each Heyting algebra $\mathcal{H}$, we have that
\begin{equation*}
    \mathcal{H}^{\neg}\vDash \phi \iff \mathcal{H}\vDash K(\phi).
\end{equation*}
For each formula $\phi\in Tm(\fancyL_{\mathsf{HA}})$, and each $\mathsf{S4}$-algebra $\mathcal{B}$, we have that
\begin{equation*}
    \mathcal{B}_{\Box}\vDash \phi \iff \mathcal{B}\vDash GMT(\phi).
\end{equation*}
For each formula $\phi\in Tm(\fancyL_{\mathsf{Ort}})$, and each $\mathsf{KTB}$-algebra $\mathcal{B}$, we have that
\begin{equation*}
    \mathcal{B}_{\Box\dia}\vDash \phi \iff \mathcal{H}\vDash G(\phi).
\end{equation*}
\end{theorem}

\subsection{DNA-logics}\label{DNA-logics}

We recall here the basic theory of DNA-logic, as presented in \cite{bezhanishvili_grilletti_quadrellaro_2021}. The concepts outlined here will be mostly immediately relevant for the subsequent sections; further concepts will be introduced when needed.

\begin{definition}\label{Negative variant of a logic}
Let $L\in \Lambda(\mathsf{IPC})$ be an intermediate logic, in the sense of an axiomatic extension. The \textit{negative variant}\footnote{The original definition takes instead the substitution $\phi[\neg \overline{p}/\overline{p}]$; the equivalence with the current presentation is discussed in \cite[Theorem 12]{Grilletti2022}; we choose this presentation as it better suits our purposes.} of $L$ is defined as
\begin{equation*}
    L^{\neg\neg}=\{\phi : \phi[\neg\neg \overline{p}/\overline{p}]\in L\}
\end{equation*}
where $\phi[\neg\neg \overline{p}/\overline{p}]$ is the result of substituting $\neg\neg q$ for each proposition letter $q$ in the term $\phi$.
\end{definition}

DNA-logics are defined to be the DNA-variants of some intermediate logic in the above sense. More explicitly, they are collections $S$ of formulas such that:
\begin{itemize}
    \item $\mathsf{IPC}\subseteq S$;
    \item $\neg\neg p\rightarrow p\in S$;
    \item $S$ is closed under Modus Ponens.
\end{itemize}
Crucially, these ``logics" are not closed under uniform substitution. We will nevertheless refer to them as DNA-logics, and refer, when needed, to logics as outlined in \ref{Subsection: Logics and Cons. Relations} as \textit{standard logics}.

 The algebraic models of DNA-logics are pairs $(\mathcal{H},v^{\neg\neg})$ of a Heyting algebra $\mathcal{H}$ and a valuation $v^{\neg\neg}$, such that $v^{\neg\neg}$ is a valuation taking values in the regular elements of $\mathcal{H}$. Note that given any valuation $v$, one can consider its \textit{double negation variant}, $v^{\neg\neg}$ by defining
 \begin{equation*}
     v^{\neg\neg}(p)=\neg\neg v(p).
 \end{equation*}
 
Given any Heyting algebra $\mathcal{H}$, we denote by:
\begin{equation*}
    \langle \mathcal{H}^{\neg\neg} \rangle
\end{equation*}
The Heyting subalgebra of $\mathcal{H}$ generated by the regular elements. Those Heyting algebras $\mathcal{H}$ such that $\mathcal{H}=\langle \mathcal{H}^{\neg\neg} \rangle$ are called \textit{regularly generated}. In parallel with the usual Birkhoff theorem, and to correspond to DNA-logics, a notion of \textit{DNA-variety} was introduced in \cite{bezhanishvili_grilletti_quadrellaro_2021}:

\begin{definition}
Let $\mathcal{H}$ and $\mathcal{H}'$ be two Heyting algebras. We say that $\mathcal{H}'$ is a \textit{core superalgebra} of $\mathcal{H}$, in symbols, $\mathcal{H}\preceq_{\neg\neg}\mathcal{H
}'$, if $\mathcal{H}\preceq \mathcal{H}'$ and $\mathcal{H}^{\neg\neg}=\mathcal{H}'^{\neg\neg}$. 

Given a class of Heyting algebras $\bf{K}$, we denote the \textit{core-superalgebra closure of $\bf{K}$} as follows:
\begin{equation*}
    \bf{K}^{\uparrow}\coloneqq \{ \mathcal{H}' : \exists \mathcal{H}\in \bf{K}, \mathcal{H}\preceq_{\neg\neg} \mathcal{H}'\}.
\end{equation*}
We say that an algebra is \textit{closed under core-superalgebras} if $\bf{K}=\bf{K}^{\uparrow}$.

We say that a class of algebras $\bf{K}$ is a \textit{DNA-variety} if $\bf{K}$ is a variety and it is closed under core superalgebras.
\end{definition}

It is not difficult to see that arbitrary intersections of DNA-varieties are again DNA-varieties. Indeed, we have the following generalisation of Tarski's $\mathbb{HSP}$ theorem:

\begin{proposition}
For each class of Heyting algebras $\mathbf{K}$, we have that $\mathbb{HSP}^{\uparrow}(\bf{K})$ is the smallest DNA-variety containing $\bf{K}$.
\end{proposition}

In \cite{bezhanishvili_grilletti_quadrellaro_2021} the following facts about these logics and varieties were established:

\begin{itemize}
    \item DNA-logics form a complete lattice under inclusion;
    \item DNA-varieties form a complete lattice under inclusion;
    \item The lattices of DNA-logics and DNA-varieties are dually isomorphic.
    \item Every DNA-variety is generated by its regularly generated algebras.
\end{itemize}

As we will see in the next chapter, none of these properties is special of this situation: all of them generalise straightforwardly to a much broader setting.

\section{Polyatomic Logics: Basic Theory and Completeness}\label{Section: Basic Theory and Completeness of PA logic}

In this section we introduce the concept of a Polyatomic logic and develop its main theory. The work here is inspired by the approach taken in \cite{bezhanishvili_grilletti_quadrellaro_2021}, generalising it to the setting of any algebraizable logic containing appropriate terms. We first introduce the notion of $\mathsf{PAt}$-$f$-logic, and prove some of its basic properties. We then prove algebraic completeness of these systems with respect to specific kinds of algebras. Throughout this section we assume that we have an algebraizable logic $\vdash$ and a quasivariety $\mathbf{Y}$, its equivalent algebraic semantics, and two sets $\mu(x)$ and $\Delta(x,y)$ witnessing algebraizability. When necessary we assume a fixed set of variables $\mathsf{Prop}$.

\subsection{PAt-Logics}

Our first definition is better understood as a naming convention for a well-known concept:

\begin{definition}
    Let $f(x)$ be a unary term in the language of $\vdash$. We say that $f$ is a \textit{selector term} if it is idempotent, i.e., $f^{2}(x)\dashv\vdash f(x)$.
\end{definition}

Idempotent terms are extremely abundant in algebraic logic:
\begin{enumerate}
    \item The term $\Box p$, in $\mathsf{S4}$ modal logic, is idempotent;
    \item The term $\neg\neg p$ in $\mathsf{IPC}$ is idempotent;
    \item The term $\Box\dia p$ in $\mathsf{KTB}$ is idempotent;
    \item The term $!p$ in intuitionistic linear logic is idempotent.
\end{enumerate}

Our usage of the term ``selector term" stems from the way such terms appear both in the setting of translations and in DNA-logics, as we will have occasion to see throughout this article: they work to select specific elements as relevant. Using such terms we now introduce the main concept of this section:

\begin{definition}
Let ${\vdash_{*}}\in \Lambda(\vdash)$, and let $f$ be a selector term. We define the \textit{$\mathsf{PAt}$-$f$-variant} of $\vdash_{*}$, denoted $\vdash^{f}_{*}$, as follows: for each set $\Gamma\cup \phi \subseteq Tm(\fancyL_{\classY},\mathsf{Prop})$,
\begin{equation*}
    \Gamma\vdash^{f}_{*}\phi \iff \Gamma[f(\overline{p})/\overline{p}]\vdash_{*}\phi[f(\overline{p})/\overline{p}],
\end{equation*}
where for each formula, $\chi[f(\overline{p})/\overline{p}]$ is obtained by substituting, for each proposition letter $q$ ocurring in $\chi$, by $f(q)$, and $\Gamma[f(\overline{p})/\overline{p}]=\{\chi[f(\overline{p})/\overline{p}]:\chi\in \Gamma\}$. When $f$ is understood, we refer simply to the $\mathsf{PAt}$\textit{-variant} of $\vdash_{*}$.
\end{definition}

It is clear that the former is a generalisation from the notion of DNA-variant in Definition \ref{Negative variant of a logic}, and leads to the following notion:

\begin{definition}
Let $\vdash_{s}$ be an arbitrary subset of $\mathbb{P}(Tm(\fancyL_{Y},\mathsf{Prop}))\times Tm(\fancyL_{Y},\mathsf{Prop})$, and let $f$ be a selector term. We say that $\vdash_{s}$ is a $\mathsf{PAt}$-$f$-logic if it is the $\mathsf{PAt}$-$f$-variant of a logic ${\vdash_{*}}\in \Lambda(\vdash)$.
\end{definition}

As DNA-logics, these relations on formulas have the notable feature that they need not be closed under uniform substitution. We collect some observations about them in the next Proposition\footnote{In fact, the properties outlined here could have been taken as the definition of $\mathsf{PAt}$-$f$-logic.}:

\begin{proposition}\label{Properties of PA-logics}
Let $\vdash_{s}$ be a $\mathsf{PAt}$-$f$-logic, and ${\vdash_{*}} \in \Lambda(\vdash)$ some logic such that $\vdash_{*}^{f} \ = \ \vdash_{s}$. Then $\vdash_{s}$ is the least consequence relation in the language $\fancyL_{\classY}$ such that:
\begin{enumerate}
    \item ${\vdash_{*}} \subseteq {\vdash_{s}}$;
    \item For all $p\in \mathsf{Prop}$, we have $\vdash_{s} \Delta(f(p),p)$.
\end{enumerate}
\end{proposition}
\begin{proof}
First we show that $\vdash_{s}$ is indeed in the given conditions. The fact that $\vdash_{s}$ is a consequence relation, follows from $\vdash_{*}$ being one, and the definition of being a $\mathsf{PAt}$ variant. To see (1), note that if $\Gamma\vdash_{*}\phi$, then by closure under uniform substitution, $\Gamma[f(\overline{p})/\overline{p}]\vdash_{*}\phi[f(\overline{p})/\overline{p}]$, so $\Gamma\vdash_{s}\phi$. The third property follows from the fact that, $f^{2}(q)\dashv\vdash f(q)$, and hence, $\vdash_{*} \Delta(f^{2}(q),f(q))$ by definition of algebraizability, so by definition of being a variant, $\vdash_{s}\Delta(f(q),q)$, for each proposition letter $q$.

Now we show that $\vdash_{s}$ is least in these conditions. Suppose that $\Vdash$ is another consequence relation satisfying the above properties. Suppose that $\Gamma\vdash_{s}\phi$. By definition, then, $\Gamma[f(\overline{p})/\overline{p}]\vdash_{*}\phi[f(\overline{p})/\overline{p}]$; hence by assumption $\Gamma[f(\overline{p})/\overline{p}]\Vdash \phi[f(\overline{p})/\overline{p}])$ by the second property. Now by assumption, $\Vdash \Delta(f(p),p)$ for all atomic propositions. We also recall that ${\vdash_{*}} \subseteq {\Vdash}$, and the former is assumed to be algebraizable, hence satisfies that for each $n$-ary function symbol $f$, $\Delta(x_{0},y_{0}),...,\Delta(x_{n},y_{n})\vdash_{*} \Delta(f(x_{0},...,x_{n}),f(y_{0},...,y_{n}))$). So by induction on complexity of formulas, and the properties of $\Delta$, we get that
\begin{equation*}
    \Vdash \Delta(\phi[f(\overline{p})/\overline{p}],\phi),
\end{equation*}
and similar for any formula in $\Gamma$. So since $\Gamma[f(\overline{p})/\overline{p}]\Vdash \phi[f(\overline{p})/\overline{p}]$, noting that $\Delta(x,y),x\vdash_{*} y$, applying uniform substitution in $\vdash_{*}$ and the Cut rule in $\Vdash$, we obtain that $\Gamma\Vdash \phi$.\end{proof}

 \begin{example}\label{Box logics}
      In addition to the DNA-logics presented in \ref{DNA-logics}, we can also conceive of $\Box$-logics, as mentioned in the introduction, as a natural example. Let $L\in \mathbf{NExt}(\mathsf{S4})$ be a normal extension of $\mathsf{S4}$, in the sense of axiomatic extensions, and consider:
\begin{equation*}
    L^{\Box}=\{\phi : \phi[\Box \overline{p}/\overline{p}]\in L\}
\end{equation*}
Then we call $L^{\Box}$ the $\Box$-\textit{variant} of $L$. We will return to these logics in \ref{Polyatomic Logics as Generalised Companions}.
\end{example}

Given our hypothesis, these logics fit naturally within the framework of \textit{weak logics} discussed in \cite{quadrellaronakov}; for instance, we note that all of them will be \textit{algebraizable} (see \cite[Theorem 28]{quadrellaronakov}) in the sense outlined in that paper, given our assumption of algebraizability of the logic $\vdash_{*}$, and Proposition \ref{Properties of PA-logics}. We now proceed to analyse the collection of all $\mathsf{PAt}$-$f$-logics, for a given fixed $f$. As such, fix throughout a given selector term $f$ in the language of $\classY$. To start, we note the following lemma which was shown in \cite[Proposition 3.4.4]{minhatesedemestrado}:

\begin{proposition}\label{Complete Lattice Homomorphisms}
Let $(\vdash_{i})_{i\in I}$ be a collection of $\mathsf{PAt}$-logics. Then $\bigcap_{i\in I}\vdash_{i}$ is a $\mathsf{PAt}$-logic. Moreover, if $S=\bigcup_{i\in I}\vdash_{i}$, the least $\mathsf{PAt}$-logic containing $S$, is the least consequence relation generated by $S$, i.e,
\begin{equation*}
    \bigvee_{i\in I}\vdash_{i}= {\vdash_{S}}
\end{equation*}
\end{proposition}

In light of the previous facts, we can denote by $\Lambda^{f}(\vdash)$ the complete lattice of $\mathsf{PAt}$-variants of the logic $\vdash$. Moreover, we have the following:

\begin{corollary}\label{PA-logic map is complete Lattice Homomorphism}
The map $(-)^{f}:\Lambda(\vdash)\to \Lambda^{f}(\vdash)$ that assigns to each logic ${\vdash_{*}}\in \Lambda(\vdash)$ its $\mathsf{PAt}$-variant is a complete lattice homomorphism.
\end{corollary}
\begin{proof}
This follows from Proposition \ref{Complete Lattice Homomorphisms}.\end{proof}

\subsection{Regular Algebras and PAt-Quasivarieties}

Just like the syntactic aspects of the theory, also the semantic constructions of DNA-logic can be generalised. Let $\mathcal{A}\in \classY$ be some algebra. Then consider the set
\begin{equation*}
    \mathcal{A}^{f}\coloneqq \{ a\in \mathcal{A} : f(a)=a\},
\end{equation*}
which we call the set of $f$-\textit{regular elements}; when $f$ is understood we refer to them solely as regular elements \footnote{We note that this notion specialises the notion of \textit{core elements} of \cite{quadrellaronakov}; our case concerns what are there called ``polynomially definable" cores}. Using these elements we can define the semantics of our models:

\begin{definition}
Let $\mathcal{A}\in \classY$. We denote by $\langle \mathcal{A}^{f}\rangle$ the subalgebra of $\mathcal{A}$ generated by regular elements. We say that $\mathcal{A}$ is \textit{regularly generated} if $\mathcal{A}=\langle \mathcal{A}^{f}\rangle$. We say that a given model $(\mathcal{A},v^{f})$ is a \textit{polyatomic model} if $v^{f}$ maps propositional variables to regular elements. 
\end{definition}

Given an algebra $\mathcal{A}$ and an equation $\lambda\approx\gamma\in Eq(\fancyL_{\classY},\mathsf{Prop})$, we write $\mathcal{A}\vDash_{f}\lambda\approx\gamma$ if for each polyatomic model $(\mathcal{A},v^{f})$ we have $\mathcal{A},v^{f}\vDash \lambda\approx\gamma$. Given a class $\bf{K}$ of algebras, we similarly write $\bf{K}\vDash_{f}\lambda\approx\gamma$ to mean that for each $\mathcal{A}\in \bf{K}$ we have $\mathcal{A}\vDash_{f}\lambda\approx\gamma$.

For the sequel, we quickly outline what it means for an algebraic model to satisfy a sequent. Given a pair of the form $(\Gamma,\phi)$, we write  $(\mathcal{A},v)\vDash (\Gamma,\phi)$ if and only if,
\begin{equation*}
    \mathcal{A},v\vDash \mu[\Gamma] \text{ implies } \mathcal{A},v\vDash \mu(\phi).
\end{equation*}
We write $\mathcal{A}\vDash (\Gamma,\phi)$ if for each valuation $v$, we have $\mathcal{A},v\vDash (\Gamma,\phi)$. Given a collection of sequents $S$, we write $\mathcal{A}\vDash S$ to mean that for each $(\Gamma,\phi)\in S$, $\mathcal{A}\vDash (\Gamma,\phi)$. As above, we also write $\mathcal{A}\vDash_{f}(\Gamma,\phi)$, and $\mathcal{A}\vDash_{f} S$ if the algebra validates the sequent/set of sequents, when restricted to the polyatomic models. Similarly we write $\bf{K}\vDash (\Gamma,\phi)$, respectively $\bf{K}\vDash_{f}(\Gamma,\phi)$, when the validity holds for each $\mathcal{A}\in \bf{K}$.

For an arbitrary valuation on $\mathcal{A}$, $v$, we can define an $f$-\textit{variant} of this valuation, $v^{f}$, by defining $v^{f}(p)=v(f(p))$ for each proposition letter $p$; in this way we have that for any $\chi\in Tm(\fancyL_{\classY},\mathsf{Prop})$,
\begin{equation*}
    v^{f}(\chi)=v(\chi[f(\overline{p})/\overline{p}]).
\end{equation*}

The following lemma collects some basic properties of polyatomic models, and $f$-validity.

\begin{lemma}\label{Invariance of truth for f(p) valuation}
Let $\mathcal{A}\in \classY$ and ${\vdash_{*}}\in \Lambda(\vdash)$. Then we have:
\begin{enumerate}
    \item $\mathcal{A}\vDash (\Gamma[f(\overline{p})/\overline{p}],\phi[f(\overline{p})/\overline{p}])$ if and only if $\mathcal{A}\vDash_{f} (\Gamma,\phi)$;
    \item If $\mathcal{A}\vDash \ \vdash_{*}$ then $\mathcal{A}\vDash_{f} \ \vdash_{*}^{f}$;
    \item  $\mathcal{A}\vDash_{f}(\Gamma,\phi)$ if and only if $\langle \mathcal{A}^{f}\rangle\vDash_{f}(\Gamma,\phi)$;
    \item If $\mathcal{A}\vDash_{f}\ \vdash_{*}^{f}$ then $\langle \mathcal{A}^{f}\rangle \vDash  {\vdash_{*}}$.
\end{enumerate}
\end{lemma}
\begin{proof}
(1) Suppose that $\mathcal{A}\nvDash (\Gamma[f(\overline{p})/\overline{p}],\phi[f(\overline{p})/\overline{p}])$; let $v$ be a valuation, such that $\mathcal{A},v\vDash \mu[\Gamma[f(\overline{p}/\overline{p})]]$ but $\mathcal{A},v\nvDash \mu(\phi[f(\overline{p})/\overline{p}])$. By the Remark after \ref{Algebraizable logic}, this means that:
\begin{equation*}
    \mathcal{A},v\vDash (\mu[\Gamma])^{f} \text{ and } \mathcal{A},v\nvDash \mu(\phi)^{f}
\end{equation*}
where $(\mu[\Gamma])^{f}$ denotes the application throughout of the substitution applying $f$ to all proposition letters. Consider $v^{f}$ as defined above; then since for each formula $\chi$, we have $v^{f}(\chi)=v(\chi[f(p)/p])$, this means that 
$$\mathcal{A},v^{f}\nvDash (\Gamma,\phi).$$ For the converse, note that if $\mathcal{A},v^{f}\nvDash (\Gamma,\phi)$, then since $v^{f}(p)=v^{f}(f(p))$ (given the properties of $f$ as a selector term), then again using structurality of $\mu$ we obtain that  $\mathcal{A},v^{f}\nvDash (\Gamma[f(p)/p],\phi[f(p)/p])$, which implies the result.

(2) Suppose that $\mathcal{A}\nvDash_{f}(\Gamma,\phi)$ where $\Gamma\vdash_{*}^{f}\phi$. Then by the previous statement, $\mathcal{A}\nvDash (\Gamma[f(p)/p],\phi[f(p)/p])$; but the latter is by definition in $\vdash_{*}$, so $\mathcal{A}\nvDash \ \vdash_{*}$.

(3) Note that $\langle \mathcal{A}^{f}\rangle^{f}=\mathcal{A}^{f}$, since the former is by definition a subalgebra of $\mathcal{A}$. Hence a given valuation $v$ on $\mathcal{A}$ is regular if and only if it is a regular valuation on $\langle \mathcal{A}^{f}\rangle$. This yields the result immediately.

(4) Suppose that $\langle \mathcal{A}^{f}\rangle\nvDash \ \vdash_{*}$, hence, there is a sequent $\Gamma\vdash_{*}\phi$, and a valuation $v$, such that:
\begin{equation*}
    \langle \mathcal{A}^{f}\rangle,v\vDash \mu[\Gamma] \text{ and } \langle \mathcal{A}^{f}\rangle,v\nvDash  \mu(\phi)
\end{equation*}
Since $\langle \mathcal{A}^{f}\rangle$ is an algebra generated by $\mathcal{A}^{f}$, we can write each element $a\in \langle A^{f}\rangle$ in the following form: for some elements $b_{1},...,b_{n}\in A^{f}$, and for some $n$-ary term $\delta(y_{1},..,y_{n})$, we have that the valuation $v:\mathsf{Prop}\to \langle A^{f}\rangle$ such that $k(x_{i})=b_{i}$ yields
\begin{equation*}
    \overline{k}(\delta(y_{1},...,y_{n}))=a.
\end{equation*}
Proceeding in this way, we can ensure that for each $x_{i}$, a variable occurring in some term $\lambda$, which appears in an equation in $\mu[\Gamma]$ or $\mu(\phi)$, we can find a term $\delta_{x_{i}}(\overline{y})$ and regular elements $c_{1},...,c_{n_{i}}$ such that the valuation $k_{x_{i}}$ assigning $k_{x_{i}}(y_{i})=c_{i}$ yields that  $
    \overline{k}_{x_{i}}(\delta_{x_{i}}(\overline{y}))=v(x_{i})$. Note moreover that by renaming if necessary, we can ensure that all these variables are distinct.

    Having this in place,  
define a valuation $w:\mathsf{Prop}\to \langle \mathcal{A}^{f}\rangle$ as follows: for each $y\in \mathsf{Prop}$ if $y=y_{i_{j}}$ for some term $\delta_{x_{i}}(\overline{y})$ of some $x_{i}$ occurring as above, then $w(y)=c_{i_{j}}$.  Otherwise we let $w$ have the value of an arbitrary regular element. Then note that if $\lambda(x_{1},...,x_{n})$ is a term in the above conditions, then:
\begin{align*}
    \overline{v}(\lambda(x_{1},...,x_{n}))&=\lambda(v(x_{1}),...,v(x_{n}))\\
    &=\lambda(\overline{k}_{x_{1}}(\delta_{x_{1}}(\overline{y}_{1})),...,\overline{k}_{x_{n}}(\delta_{x_{n}}(\overline{y}_{n})))\\
    &=\lambda(\overline{w}(\delta_{x_{1}}(\overline{y}_{1})),...,\overline{w}(\delta_{x_{n}}(\overline{y}_{n})))\\
    &=\overline{w}(\lambda(\delta_{x_{1}}(\overline{y}_{1}),...,\delta_{x_{n}}(\overline{y}_{n}))).
\end{align*}
Thus, if we consider the substitution $\nu:\terms(\mathsf{Prop})\to \terms(\mathsf{Prop})$ defined by $\nu(x_{i})=\delta_{x_{i}}(\overline{y}_{i})$, we obtain that:
\begin{equation*}
    \langle \mathcal{A}^{f} \rangle, w\vDash \nu[\mu[\Gamma]] \text{ and } \langle \mathcal{A}^{f}\rangle,w\nvDash \nu[\mu(\phi)].
\end{equation*}
By structurality of $\mu$, this means that $\langle \mathcal{A}^{f}\rangle,w\nvDash (\nu[\Gamma],\nu(\phi))$.  Since the valuation $w$ takes values in regular elements, also $\langle \mathcal{A}^{f}\rangle,w\nvDash_{f} (\nu[\Gamma],\nu(\phi))$. Now by part (3) of this proposition, we then obtain that  $\mathcal{A}\nvDash_{f} (\nu[\Gamma],\nu(\phi))$. Since $\vdash_{*}$ is a standard logic, and thus, is closed under substitution, and since $\Gamma\vdash_{*}\phi$, we have that $\nu[\Gamma]\vdash_{*}\nu(\phi)$, and in turn, $\nu[\Gamma][f(\overline{p})/\overline{p}]\vdash_{*}\nu(\phi)[f(\overline{p})/\overline{p}]$. Hence:
\begin{equation*}
    \nu[\Gamma]\vdash_{*}^{f}\nu(\phi)
\end{equation*}
But then since $\mathcal{A}\nvDash_{f} (\nu[\Gamma],\nu(\phi))$, we have that $\mathcal{A}\nvDash_{f} \ \vdash_{*}^{f}$, which was to show.\end{proof}

We isolate one fact which was used in the previous proof, and will likewise be necessary at a later stage:

\begin{corollary}\label{Failure in regularly generated yields substitution instance}
Suppose that $\langle \mathcal{A}^{f}\rangle \nvDash (\Gamma,\phi)$. Then there is a substitution $\nu$, such that $\mathcal{A}\nvDash_{f} (\nu[\Gamma][f(\overline{p})/\overline{p}],\nu(\phi)[f(\overline{p})/\overline{p}])$.
\end{corollary}

As stated, our goal will be to prove an algebraic completeness result, reminiscent of Birkhoff and Maltsev's results. As such, we now turn to the analogues of quasivarieties in this setting.

\begin{definition}
Let $\mathcal{A},\mathcal{B}\in \classY$ be algebras. We say that $\mathcal{A}$ is a \textit{core superalgebra} of $\mathcal{B}$, and write $\mathcal{B}\preceq_{f}\mathcal{A}$ if $\mathcal{B}\preceq \mathcal{A}$ and $\mathcal{A}^{f}=\mathcal{B}^{f}$.
\end{definition}

\begin{definition}
Let $\bf{K}\in \Xi(\classY)$ be a subquasivariety of $\classY$, and $f$ a selector term. We define the \textit{$\mathsf{PAt}$-$f$-variant} of $\bf{K}$ as:
\begin{equation*}
    \bf{K}^{\uparrow} \coloneqq \{ \mathcal{A} : \exists \mathcal{B}\in \bf{K}, \mathcal{B}\preceq_{f} \mathcal{A}\}.
\end{equation*}
We say that $\bf{K}\in \Xi(\classY)$ is a \textit{$\mathsf{PAt}$-$f$-quasivariety},, if it is the $\mathsf{PAt}$-variant of some subquasivariety of $\classY$.

Furthermore, we say that a $\mathsf{PAt}$-$f$-quasivariety is a \textit{$\mathsf{PAt}$-$f$-variety}, relative to the selector term and $\classY$, if it is the $\mathsf{PAt}$-$f$-variant of a subvariety of $\classY$.
\end{definition}

This definition likewise parallels the one for DNA-varieties, introduced in \cite{bezhanishvili_grilletti_quadrellaro_2021}. We will now show that analogues of Maltsev and Tarski's results hold. Again, assume throughout a fixed selector term $f$, such that any mention of $\mathsf{PAt}$-logic should be taken to mean $\mathsf{PAt}$-$f$-logic. We will need the following fact, which easily follows from the definitions of products and reduced products:

\begin{lemma}\label{Functoriality of products and reduced products}
Let $(\mathcal{A}_{i})_{i\in I}$ and $(\mathcal{B}_{i\in I}$ be two collections of similar algebras, and assume that for each $i$, $\mathcal{A}_{i}\preceq \mathcal{B}_{i}$. Then:
\begin{itemize}
\item $\prod_{i\in I}\mathcal{A}_{i\in I}\preceq \prod_{i\in I}\mathcal{B}_{i\in I}$
\item Given a filter $F$ on $I$, $\prod_{i\in I}\mathcal{A}_{i\in I}/F \preceq \prod_{i\in I}\mathcal{B}_{i\in I}/F$
\end{itemize}
\end{lemma}

\begin{proposition}\label{Characterisation of PA-quasivarieties}
Let $\mathbf{K}\subseteq \classY$. Then:
\begin{enumerate}
    \item $\mathbf{K}$ is a $\mathsf{PAt}$-quasivariety if and only if it is closed under subalgebras, reduced products and core superalgebras.
    \item $\mathbf{K}$ is a $\mathsf{PAt}$-variety if and only if it is closed under subalgebras, homomorphic images, products and core superalgebras.
\end{enumerate}
\end{proposition}
\begin{proof}
The proof is a generalisation of \cite[Proposition 3.18]{bezhanishvili_grilletti_quadrellaro_2021}, where the role of Heyting algebra regular elements is replaced by that of $f$-regular elements. Indeed, we provide only the proof of the first statement, as the second follows by the arguments in that Proposition.

For the right to left direction, we simply note that if $\bf{K}$ is closed under the specified operations, then it forms a quasivariety; hence if it is further closed under core superalgebras it will be its own $\mathsf{PAt}$-variant. For the other direction, let $\bf{K}^{\uparrow}$ be some $\mathsf{PAt}$-variant. We check closure under the operations.
\begin{itemize}
    \item (Subalgebras): Suppose that $\mathcal{A}\preceq \mathcal{B}$ and $\mathcal{B}\in\bf{K}^{\uparrow}$. By definition, then, there is some $\mathcal{C}\in\bf{K}$ such that $\mathcal{C}\preceq\mathcal{B}$, with the same set of regular elements. Consider $\mathcal{C}\cap \mathcal{A}$; we have that this will be a subalgebra of $\mathcal{B}$ and of $\mathcal{A}$ (since intersections of subalgebras are again subalgebras), so it will be in $\bf{K}$. Moreover, note that $(\mathcal{C}\cap\mathcal{A})^{f}=(\mathcal{A})^{f}$: whenever $a\in A\cap C$, then $a\in A$, so if $f(a)=a$ in $\mathcal{A}\cap \mathcal{C}$, then the same will hold in $\mathcal{A}$, since the former is a subalgebra of $\mathcal{A}$; conversely, if $a\in A$, and $f(a)=a$, then already $f(a)=a$ in $\mathcal{B}$, so by assumption also $a\in C$, i.e. $a\in A\cap C$. This then shows that $\mathcal{A}\in \mathbf{K}^{\uparrow}$.
    \item (Reduced Products): Suppose that $(\mathcal{A}_{i})_{i\in I}$ are a family of algebras all in $\bf{K}^{\uparrow}$, and hence, that $(\mathcal{B}_{i})_{i\in I}$ is a family of algebras in $\bf{K}$ such that $\mathcal{B}_{i}\preceq_{f} \mathcal{A}_{i}$. Now consider $\prod_{i\in I}\mathcal{A}_{i}/R$. By closure of $\bf{K}$ under reduced products, we obtain that $\prod_{i\in I}\mathcal{B}_{i}/R\in \bf{K}$. Moreover, we can show that
    \begin{equation*}
        \big(\prod_{i\in I}\mathcal{B}_{i}/R\big)^{f}=\prod_{i\in I}\mathcal{B}_{i}^{f}/R.
    \end{equation*}
    Indeed, if $b/R$ is an element in the first set, then $S=\{i\in I : f(b(i))=b(i)\}\in R$; hence consider $c\in \prod_{i\in I}\mathcal{B}_{i}^{f}$, such that $c(i)=b(i)$ whenever $i\in S$, is equal to an arbitrary regular element in the remaining coordinates. Then $S\subseteq \{i\in I : c(i)=b(i)\}$, so $c/R=b/R$, and so $b/R\in \prod_{i\in I}\mathcal{B}_{i}^{f}/R$. Conversely if $c/R\in \prod_{i\in I}\mathcal{B}^{f}_{i}/R$, then surely $\{i\in I : f(c(i))=c(i)\}\in R$ (since it holds everywhere), so $c/R$ belongs to the first set.

    Now by this, and the fact that $\mathcal{A}_{i}^{f}=\mathcal{B}_{i}^{f}$ for each $i$ we get that:
    \begin{equation*}
        (\prod_{i\in I}\mathcal{A}_{i}/R)^{f}=\prod_{i\in I}\mathcal{A}_{i}^{f}/R=\prod_{i}\mathcal{B}_{i}^{f}/R=(\prod_{i\in I}\mathcal{B}_{i}/R)^{f}.
    \end{equation*}
    Since by Lemma \ref{Functoriality of products and reduced products},  $\prod_{i\in I}\mathcal{A}_{i}/R\in \mathbf{K}^{\uparrow}$ is a superalgebra of $\prod_{i\i I}\mathcal{B}_{i}/R$, it will then be a core superalgebra, as intended.
\end{itemize}
\end{proof}

As could be expected, $\mathsf{PAt}$-quasivarieties and $\mathsf{PAt}$-varieties form complete lattices; the former entails this if we recall that any poset which has all infima also has all suprema:

\begin{lemma}
Let $(\mathbf{K}_{i})_{i\in I}$ be a non-empty collection of $\mathsf{PAt}$-quasivarieties (resp. $\mathsf{PAt}$-varieties). Then $\bigcap_{i\in I}\mathbf{K}_{i}$ is a $\mathsf{PAt}$-quasivariety (resp. $\mathsf{PAt}$-variety).
\end{lemma}
\begin{proof}
In light of Proposition \ref{Characterisation of PA-quasivarieties}, we simply note that the intersection will be closed under all the operations.\end{proof}

Moreover, using the same ideas as in the proof of Proposition \ref{Characterisation of PA-quasivarieties}, it is not difficult to show an analogue of Maltsev's characterisation:

\begin{lemma}
    Given a class of algebras $\bf{K}\subseteq\classY$, this will be a $\mathsf{PAt}$-quasivariety if and only if there is a class $\bf{K}'\subseteq \classY$ such that:
\begin{align*}
\bf{K}=(\mathbb{ISP}_{R})^{\uparrow}(\bf{K}').
\end{align*}
Similarly, a class $\mathbf{K}$ is a $\mathsf{PAt}$-variety if and only if $\mathbf{K}=\mathbb{HSP}^{\uparrow}(\mathbf{K})$
\end{lemma}

The latter lemma allows us to explicitly describe the join of the lattice of $\mathsf{PAt}$-quasivarieties. Given a family $\bf{K}_{i}$ of $\mathsf{PAt}$-quasivarieties, it is not hard to see that
\begin{equation*}
    \bigvee_{i\in I}\bf{K}_{i} = (\mathbb{ISP}_{R})^{\uparrow}(\bigcup_{i\in I}\bf{K}_{i}).
\end{equation*}

We now proceed to show that the map assigning the $\mathsf{PAt}$-variant to a quasivariety is indeed a well-defined homomorphism. For that purpose, denote by $\Xi^{f}(\classY)$ the lattice of $\mathsf{PAt}$-quasviarieties of the quasivariety $\classY$.

\begin{proposition}\label{Pa-variety map is complete homomorphism}
The map $(-)^{\uparrow}:\Xi(\classY)\to \Xi^{f}(\classY)$ is a complete lattice homomorphism.
\end{proposition}
\begin{proof}
First let $(\mathbf{K}_{i})_{i\in I}$ be a collection of subquasivarieties of $\classY$. Then we show that:
\begin{equation*}
    \big(\bigcap_{i\in I} \mathbf{K}_{i}\big)^{\uparrow}=\bigcap_{i\in I}\mathbf{K}_{i}^{\uparrow}.
\end{equation*}
Indeed, if $\mathcal{A}\in (\bigcap_{i\in I} \mathbf{K}_{i})^{\uparrow}$ then $\mathcal{A}$ is the core superalgebra of some $\mathcal{B}$, such that $\mathcal{B}\in \mathbf{K}_{i}$ for each $i$; but since all are closed under core superalgebras, then clearly $\mathcal{A}\in \bigcap_{i\in I}\bf{K}_{i}$ as well. For the converse, if $\mathcal{B}_{i}\preceq_{f} \mathcal{A}$ for each $i$, consider $\bigcap_{i\in I}\mathcal{B}_{i}$; then surely we will have that $\bigcap_{i\in I}\mathcal{B}_{i}\preceq \mathcal{A}$, and since all these subalgebras share the same regular elements, we also have $\bigcap_{i\in I}\mathcal{B}_{i}\preceq_{f}\mathcal{A}$. Additionally, $\bigcap_{i\in I}\mathcal{B}_{i}\in \bigcap \mathbf{K}_{i}$, by closure under subalgebras, which shows that $\mathcal{A}\in (\bigcap_{i\in I}\bf{K}_{i})^{\uparrow}$.

For the preservation of the join, we show that:
\begin{equation*}
    \big(\bigvee_{i\in I}\mathbf{K}_{i}\big)^{\uparrow}=\bigvee_{i\in I}\mathbf{K}_{i}^{\uparrow}
\end{equation*}
Where the left hand side join is taken in the lattice of quasivarieties, and the right hand side join in the lattice $\mathsf{PAt}$-quasivarieties. Indeed, by our previous remark, it suffices to show that:
\begin{equation*}
    \mathbb{ISP}_{R}^{\uparrow}\big(\bigcup_{i\in I}\mathbf{K}_{i}\big) = \mathbb{ISP}_{R}^{\uparrow}\big(\bigcup_{i\in I}\mathbf{K}_{i}^{\uparrow}\big).
\end{equation*}

One inclusion follows from the inflationarity and monotonicity of the operators at play. For the other, assume that $\mathcal{B}\preceq_{f} \mathcal{A}$, where $\mathcal{B}$ is a subalgebra of $\mathcal{C}=\prod_{j\in J}\mathcal{D}_{j}/F$, where for each $j$, there is some $\mathcal{E}_{j}$ such that $\mathcal{E}_{j}\preceq_{f} \mathcal{D}_{j}$, and $\mathcal{E}_{j}\in \mathbf{K}_{j}$.

Now consider $\mathcal{E}=\prod_{j\in J}\mathcal{E}_{j}/F$. Note that $\mathcal{E}\in \mathbb{P}_{R}(\bigcup_{i\in I}\bf{K}_{i})$. Then by Lemma \ref{Functoriality of products and reduced products}, we have that $\prod_{j\in J}\mathcal{E}_{j}/F \preceq \mathcal{C}$, and indeed, by the arguments used in Proposition \ref{Characterisation of PA-quasivarieties}, $\prod_{j\in J}\mathcal{E}_{j}/F \preceq_{f} \mathcal{C}$. Since $\mathcal{B}\preceq \mathcal{C}$, we can look at $\mathcal{G}= \mathcal{B}\cap \prod_{j\in J}\mathcal{E}_{j}/F$, which is a subalgebra of both of these algebras. Then we claim that $\mathcal{B}$ is a core superalgebra of $\mathcal{G}$: indeed, if $a\in B$, and $f(a)=a$, then surely $a\in \mathcal{C}$, and hence, $a\in \prod_{j\in J}\mathcal{E}_{j}/F$ since this has the same regular elements as $\mathcal{C}$. So $a\in G$ by definition. Thus, since $\mathcal{A}$ is a core superalgebra of $\mathcal{B}$, it will also be a core superalgebra of $\mathcal{G}$, and $\mathcal{G}$ is in $\mathbb{SP}_{R}(\bigcup_{i\in I}\mathbf{K}_{i})$. So $\mathcal{A}\in \mathbb{ISP}_{R}^{\uparrow}(\bigcup_{i\in I}\mathbf{K}_{i})$, as desired. This shows the other inclusion, and concludes the proof.\end{proof}

We conclude this section by mentioning some basic facts which will allow us a sharper completeness theorem further along:

\begin{corollary}\label{Regular generation of PA-Quasivarieties}
Every $\mathsf{PAt}$-Quasivariety is generated by its regularly generated algebras.
\end{corollary}
\begin{proof}
Note that if $\mathbf{K}$ is a $\mathsf{PAt}$-Quasivariety, and $\mathcal{B}\in \bf{K}$ is not regularly generated, then $\langle \mathcal{B}^{f}\rangle \leq \mathcal{B}$, and so $\langle \mathcal{B}^{f}\rangle \in \mathbf{K}$ is a regularly generated algebra which has $\mathcal{B}$ as its core-superalgebra.\end{proof}

\subsection{Dual equivalence between PAt-logics and quasivarieties}

Given $\mathbf{K}\in \Xi(\classY)$ and ${\vdash_{*}}\in \Lambda(\vdash)$, we recall the following two maps, witnessing the algebraic completeness of quasivarieties with respect to logics.

\begin{align*}
    \mathsf{QVar}(\vdash_{*}) &\coloneqq \{ \mathcal{A} : {\forall} (\Gamma,\phi)\in {\vdash_{*}}, \mathcal{A}\vDash (\Gamma,\phi)\}\\
    \mathsf{Log}(\mathbf{K}) &\coloneqq \{(\Gamma,\phi) : \bf{K}\vDash (\Gamma,\phi)\}
\end{align*}

As a result of completeness, we have that $\mathsf{QVar}(-):\Lambda(\vdash)\to \Xi(\classY)$ is a dual isomorphism, where $\mathsf{Log}(-)$ is the inverse. Consequently, given a quasivariety $\mathbf{K}\in \Xi(\classY)$, we often refer to $\mathsf{Log}(\mathbf{K})$, as the \textit{dual logic} of $\mathbf{K}$, and similarly, to $\mathsf{QVar}(\vdash_{*})$ as the \textit{dual quasivariety}.

In a similar fashion, we can define the $f$-Polyatomic version of these maps:
\begin{align*}
    \mathsf{QVar}^{f}(\vdash_{*})&\coloneqq \{\mathcal{A} :\forall (\Gamma,\phi)\in \  \vdash_{*},  \mathcal{A}\vDash_{f} (\Gamma,\phi)\}\\
    \mathsf{Log}^{f}(\mathbf{K}) &\coloneqq \{(\Gamma,\phi) : \bf{K}\vDash_{f} (\Gamma,\phi)\}
\end{align*}

Our goal in this section will be to show that $\mathsf{QVar}^{f}(-):\Lambda^{f}(\vdash)\to \Xi^{f}(\classY)$ is also a dual isomorphism. First, we note the following, which follows by similar arguments to those given in \cite[Lemma 3.22]{bezhanishvili_grilletti_quadrellaro_2021}, making use of Lemma \ref{Invariance of truth for f(p) valuation}.

\begin{proposition}\label{Preservation of PA-validity under operations}
The $\mathsf{PAt}$-validity of a pair $(\Gamma,\phi)$ is preserved under the operations of taking isomorphisms, subalgebras, reduced products, and core superalgebras.
\end{proposition}

\begin{corollary}\label{Maps are well-defined 1}
Given a logic ${\vdash_{*}}\in \Lambda(\vdash)$, the class $\mathsf{QVar}^{f}(\vdash_{*})$ is a $\mathsf{PAt}$-quasivariety.
\end{corollary}
\begin{proof}
By Proposition \ref{Preservation of PA-validity under operations}, we have that $\mathsf{QVar}^{f}(\vdash_{*})$ is closed under subalgebras, reduced products and core superalgebras. Hence by Proposition \ref{Characterisation of PA-quasivarieties}, we get the result.\end{proof}

For ease of notation, given a quasivariety $\bf{K}$ we sometimes write $\vdash_{\bf{K}}$ to denote the dual logic of this quasivariety.

\begin{proposition}\label{Maps are well-defined 2}
Given $\mathbf{K}\in \Xi(\classY)$, $\mathsf{Log}^{f}(\mathbf{K})$ is the $\mathsf{PAt}$-variant of $\mathsf{Log}(\bf{K})$.
\end{proposition}
\begin{proof}
Let $\vdash_{s}$ be the $\mathsf{PAt}$-variant of $\vdash_{\bf{K}}$. Assume that $\Gamma\vdash_{s}\phi$. Then by definition, $\Gamma[f(\overline{p})/\overline{p}]\vdash_{\mathbf{K}}\phi(f(\overline{p})/\overline{p})$. Now if $\mathcal{A}\in \bf{K}$, then by Proposition \ref{Invariance of truth for f(p) valuation}, we have that $\mathcal{A}\vDash_{f}(\Gamma,\phi)$, hence $(\Gamma,\phi)\in \mathsf{Log}^{f}(\bf{K})$. Conversely, if $(\Gamma,\phi)\in \mathsf{Log}^{f}(\bf{K})$, and $\mathcal{A}\in \bf{K}$, then we have that $\mathcal{A}\vDash_{f}(\Gamma,\phi)$. By the above cited Proposition, then $\mathcal{A}\vDash (\Gamma[f(\overline{p})/\overline{p}],\phi(f(\overline{p})/\overline{p}))$. Hence $\Gamma[f(\overline{p})/\overline{p}]\vdash_{\bf{K}}\phi(f(\overline{p})/\overline{p})$, and so by definition $\Gamma\vdash_{s}\phi$.\end{proof}

Using Corollary \ref{Maps are well-defined 1} and Proposition \ref{Maps are well-defined 2}, we have that the maps outlined above are indeed well-defined. In order to obtain the dual isomorphism, we will need to establish the following lemma:

\begin{lemma}\label{Basic Commutativity Result}
    Let ${\vdash_{*}}\in \Lambda(\vdash)$ and $\bf{K}\in \Xi(\classY)$. Then we have that:
    \begin{enumerate}
        \item $\mathsf{QVar}^{f}(\vdash_{*}^{f})=\mathsf{QVar}^{\uparrow}(\vdash_{*})$
    \item $\mathsf{Log}^{f}(\mathbf{K}^{\uparrow})=(\mathsf{Log}(\mathbf{K}))^{f}$
    \end{enumerate}
\end{lemma}
\begin{proof}
    (1) Let $\mathcal{A}\in \mathsf{QVar}^{f}(\vdash_{*}^{f})$. By definition, then $\mathcal{A}\vDash_{f} \ \vdash_{*}^{f}$, so by Lemma \ref{Invariance of truth for f(p) valuation}, we have that $\langle \mathcal{A}^{f}\rangle \vDash \ \vdash_{*}$. Then we have that $\langle \mathcal{A}^{f}\rangle\in \mathsf{QVar}(\vdash_{*})$, and since $\mathcal{A}$ is a core-superalgebra of $\langle \mathcal{A}^{f}\rangle$, $\mathcal{A}\in \mathsf{QVar}^{\uparrow}(\vdash_{*})$. Conversely, assume that $\mathcal{A}\in \mathsf{QVar}(\vdash_{*})^{\uparrow}$. Hence for some $\mathcal{B}\in \mathsf{QVar}(\vdash_{*})$, $\mathcal{B}\preceq_{f}\mathcal{A}$; since $\mathcal{B}\vDash \ \vdash_{*}$, again by Lemma \ref{Invariance of truth for f(p) valuation}, $\mathcal{B}\vDash_{f} \ \vdash_{*}^{f}$, so $\mathcal{B}\in \mathsf{QVar}^{f}(\vdash_{*}^{f})$; since this is a $\mathsf{PAt}$-quasivariety, it is closed under core superalgebras, so $\mathcal{A}\in \mathsf{QVar}^{f}(\vdash_{*}^{f})$.

    (2) Suppose that $(\Gamma,\phi)\notin \mathsf{Log}^{f}(\bf{K}^{\uparrow})$. Hence, there is some $\mathcal{A}\in \bf{K}^{\uparrow}$, such that $\mathcal{A}\nvDash_{f} (\Gamma,\phi)$. Hence by Lemma \ref{Invariance of truth for f(p) valuation}, we have that $\langle \mathcal{A}^{f}\rangle\nvDash_{f} (\Gamma,\phi)$. By construction, then there is some $\mathcal{B}\in \bf{K}$ such that $\mathcal{B}\preceq_{f}\mathcal{A}$; thus, by using the same Lemma, $\mathcal{B}\nvDash_{f} (\Gamma,\phi)$, and from there $\mathcal{B}\nvDash (\Gamma[f(\overline{p})/\overline{p}],\phi[f(\overline{p})/\overline{p}])$. This implies that $(\Gamma,\phi)\notin (\mathsf{Log}(\bf{K}))^{f}$. Conversely if $(\Gamma,\phi)\notin (\mathsf{Log}(\bf{K}))^{f}$, then $(\Gamma[f(\overline{p})/\overline{p}],\phi[f(\overline{p})/\overline{p}])\notin \mathsf{Log}(\bf{K})$, so for some $\mathcal{A}\in \bf{K}$, $\mathcal{A}\nvDash (\Gamma[f(\overline{p})/\overline{p}],\phi[f(\overline{p})/\overline{p}])$; then $\mathcal{A}\nvDash_{f} (\Gamma,\phi)$, and since $\mathcal{A}\in \bf{K}\subseteq \bf{K}^{\uparrow}$, then $(\Gamma,\phi)\notin \mathsf{Log}^{f}(\bf{K}^{\uparrow})$.
\end{proof}

\begin{corollary}\label{Dual Isomorphism for PAt logics and quasivarieties}
    For each $\vdash_{s}\in \Lambda^{f}(\vdash)$ and $\bf{K}\in \Xi^{f}(\classY)$ we have that:
    \begin{enumerate}
        \item $\mathsf{Log}^{f}(\mathsf{QVar}^{f}(\vdash_{s}))= {\vdash_{s}}$;
        \item $\mathsf{QVar}^{f}(\mathsf{Log}^{f}(\bf{K}))=\bf{K}$.
    \end{enumerate}
\end{corollary}
\begin{proof}
    We show (1), since (2) follows using similar arguments. Let $\vdash_{*}$ be such that ${\vdash_{*}^{f}}={\vdash_{s}}$. Then we have that:
    \begin{equation*}
        \mathsf{Log}^{f}(\mathsf{QVar}^{f}(\vdash_{s}))=\mathsf{Log}^{f}(\mathsf{QVar}^{\uparrow}(\vdash_{*}))=(\mathsf{Log}(\mathsf{QVar}(\vdash_{*})))^{f}={\vdash_{s}}
    \end{equation*}
    where the first and second equalities use Lemma \ref{Basic Commutativity Result}.
\end{proof}

Using these facts, we can obtain the desired definability theorems and algebraic completeness. The following tells us that $\mathsf{PAt}$-quasivarieties are defined by the $\mathsf{PAt}$-logic to which they are associated.

\begin{theorem}[Definability Theorem]
For $\mathbf{K}\in \Xi^{f}(\classY)$, and for every algebra $\mathcal{A}\in \classY$:
\begin{equation*}
    \mathcal{A}\in \mathbf{K} \iff \mathcal{A}\vDash_{f} \mathsf{Log}^{f}(\mathbf{K})
\end{equation*}
\end{theorem}
\begin{proof}
If $\mathcal{A}\in \mathbf{K}$, then certainly $\mathcal{A}\vDash_{f} \mathsf{Log}^{f}(\mathbf{K})$. Conversely, if $\mathcal{A}\notin \bf{K}$, by Lemma \ref{Dual Isomorphism for PAt logics and quasivarieties}, we have that $\bf{K}=\mathsf{QVar}^{f}(\mathsf{Log}^{f}(\mathbf{K}))$. So $\mathcal{A}\notin \mathsf{QVar}^{f}(\mathsf{Log}^{f}(\mathbf{K}))$, hence $\mathcal{A}\nvDash_{f} \mathsf{Log}^{f}(\mathbf{K})$.
\end{proof}

As a corollary we have the following Birkhoff theorem analogue:

\begin{theorem}[$\ \mathsf{PAt}$-Birkhoff]
A class $\mathbf{K}$ of algebras is a $\mathsf{PAt}$-quasivariety if and only if it is $\mathsf{PAt}$-definable by a collection of quasi-equations.\end{theorem}

As another immediate consequence of Corollary \ref{Dual Isomorphism for PAt logics and quasivarieties}, we obtain the following completeness result:

\begin{theorem}[Algebraic completeness]
Every $\mathsf{PAt}$-logic $\vdash_{s}$ is complete with respect to its corresponding $\mathsf{PAt}$-quasivariety, i.e, for every pair $(\Gamma,\phi)$:
\begin{equation*}
    \Gamma\vdash_{s} \phi \iff \mathsf{QVar}^{f}(\vdash_{s})\vDash_{f} (\Gamma,\phi)
\end{equation*}
\end{theorem}

In light of this, together with Corollary \ref{PA-logic map is complete Lattice Homomorphism} and Proposition \ref{Pa-variety map is complete homomorphism}, we then have:

\begin{corollary}\label{Completeness mapping}
The map $\mathsf{QVar}^{f}$ is a dual isomorphism between the complete lattices $\Lambda^{f}(\vdash)$ and $\Xi^{f}(\classY)$.
\end{corollary}

Moreover, in light of Corollary \ref{Regular generation of PA-Quasivarieties} we obtain the following stronger form of completeness, which allows us to focus on the regularly generated models.

\begin{corollary}
Every $\mathsf{PAt}$-logic $\vdash_{f}$ is complete with respect to the class of regularly generated algebras validating the logic.
\end{corollary}

We have thus laid out the basics of the theory of $\mathsf{PAt}$-logics. This generalises the situation as known in the literature in a very straightforward way. However, we would also like to remark that these logics could be useful for the study of many well-known logical systems which lack uniform substitution: logics based on team semantics, such as inquisitive logic \cite{Ciardelli2018-ql} and dependence logic \cite{Yang2016}, dynamic logics such as public announcement logic and dynamic epistemic logic \cite{sep-dynamic-epistemic}, Buss' logic of ``pure provability" \cite{Buss1990} and various two-dimensional modal logics \cite{aqvisttwodimensional,daviesandhumberstone}. We leave an analysis of the applicability of this framework for such cases for future work.

\section{Translations and Adjunctions}\label{Section: Translations and Polyatomic Logics}

In this section we make concrete what we mean by translations, from an algebraic point of view. The connection between translations and adjunctions was introduced in \cite{Moraschini2018,moraschiniphd}, and developed in much greater generality than we present it here. In subsection \ref{Algebraic Translations and Adjunctions}, we present the parts of the theory which are relevant for our analysis and relate it to a number of relevant examples. Additionally in subsection \ref{Adjunctions and Selective translations} we introduce some classes of translations, namely selective, strongly selective and sober translations, which will set up the groundwork for the development of a Generalised Blok-Esakia theory, in the next section.

\subsection{Algebraic Translations and Adjunctions}\label{Algebraic Translations and Adjunctions}

In \cite{Moraschini2018}, the concept of a \textit{contextual translation} was introduced with the goal of providing a characterisation of adjunctions between generalised quasivarieties. This was inspired by work of McKenzie \cite{McKenzie2017} and Dukarm \cite{Dukarm1988}, and related to work by Freyd \cite{freyd}, which describes categorical equivalence through two ``deformations" of categories of algebras. The key idea of this approach is that right adjoints correspond to translations between the equational consequence relations of the categories of algebras. Below we recall the definition of contextual translation, according to the presentation of  \cite{Moraschini2018}, as well as some key results, which we provide without proof. We also note that our discussion focus only on \textit{unary translations}.

\begin{definition}
Let $\classX$ and $\classY$ be classes of algebras, and let $X$ be a set of variables. We say that a map $\zeta:\fancyL_{\classX}\to Tm(\fancyL_{\classY},X)$ is a \textit{symbol assignment} if for each $n$, $\zeta$ sends $n$-ary function symbols from $\fancyL_{\classX}$ to $n$-ary terms of $Tm(\fancyL_{\classY},X)$.
\end{definition}

Given a symbol assignment $\zeta$ we define a \textit{translation map} $\zeta_{*}:Tm(\fancyL_{\classX},X)\to Tm(\fancyL_{\classY},X)$ by stipulating that:
\begin{enumerate}
    \item $\zeta_{*}(x_{i})=x_{i}$ for variables $x_{i}\in X$;
    \item $\zeta_{*}(c)=\zeta(c)$ for constants $c\in \fancyL_{\classX}$;
    \item If $\phi_{1},...,\phi_{n}\in Tm(\fancyL_{\classX},X)$, and $g(x_{1},...,x_{n})\in \fancyL_{\classX}$, then,
    \begin{equation*}
      \zeta_{*}(g(\phi_{1},...,\phi_{n}))=\zeta(g)(\zeta_{*}(\phi_{1}),...,\zeta_{*}(\phi_{n}))
    \end{equation*}
\end{enumerate}

We furthermore write $\zeta^{*}$ for the lifting of $\zeta_{*}$ to sets of equations. We define $\zeta^{*}:\mathbb{P}(Eq(\fancyL_{\classX},X))\to \mathbb{P}(Eq(\fancyL_{\classY},X))$ by setting, for $\Phi\subseteq Eq(\fancyL_{\classX},X)$:
\begin{equation*}
    \zeta^{*}(\Phi)=\{\zeta_{*}(\delta)\approx \zeta_{*}(\gamma) : \delta\approx\gamma\in \Phi\}
\end{equation*}

For simplicity of notation, when making definitions and explanations we will let $\zeta$ denote the translation; however, elsewhere we will keep the distinction of $\zeta,\zeta_{*},\zeta^{*}$ clear. We also note that whilst above we have kept the set $X$ fixed throughout, in the definition of the translation map any set $Z$ can chosen, so long as $Z$ extends the set used for the symbol assignment. When $X$ is not important, we will omit it from notations of the translation map.

More than a merely syntactic assignment, though, we want our semantics to reflect this translation in some sense. Hence we will need the notion of a  \textit{contextual translation}.

\begin{definition}
Let $\classX$ and $\classY$ be two classes of algebras. We say that a pair $(\zeta,\Theta)$, where $\zeta$ is a translation, and $\Theta\subseteq Eq(\fancyL_{\classY},\{x\})$ is a finite subset, is a \textit{contextual translation} if the following holds:
\begin{enumerate}
    \item For every set $\Phi\cup\{\lambda\approx \gamma\}\subseteq Eq(\fancyL_{\classX},X)$, we have:
    \begin{equation*}
       \text{ if }  \Phi \eqrelX \lambda\approx \gamma, \text{ then } \zeta^{*}(\Phi)\cup \bigcup_{x\in X}\Theta(x)\eqrelY\zeta^{*}(\lambda\approx \gamma)
    \end{equation*}
    \item For every $n$-ary operation symbol $f(y_{1},...,y_{n})\in \fancyL_{\classX}$, we have:
    \begin{equation*}
        \Theta(x_{1})\cup...\cup \Theta(x_{n})\eqrelY \Theta(\zeta_{*}(f(x_{1},...,x_{n})))
    \end{equation*}
\end{enumerate}
In this case we refer to the set of equations $\Theta$ as the \textit{context} of the translation.
\end{definition}

\begin{example}(KGG and GMT translations)\label{Example of KGG and GMT translations 1}
The KGG and GMT translations we met in Definition \ref{Double Negation Translation} and Definition \ref{Heyting algebras and S4} can be seen as contextual translations: in both cases we amend the clauses on proposition letters, by letting instead, for $p$ proposition letters:
\begin{equation*}
    K(p)=p \text{ and } GMT(p)=p,
\end{equation*}
and taking the contexts
\begin{equation*}
    \Theta_{K}=\{p\approx \neg\neg p\} \text{ and } \Theta_{GMT}=\{p\approx \Box p\}.
\end{equation*}
In both cases, these translations were shown to be contextual in \cite{Moraschini2018}; this can essentially be derived from  a fact which can be derived from Theorem \ref{Translation theorem for the double negation translation}.
\end{example}

The introduction of a context has the following effect: we take the equations we are interested in, when looking at a given algebra, and we interpret these equations in a smaller algebra, comprised only of the elements satisfying the equations in the context. Suggesting a connection we will draw in the next section, we refer generically to these as the $\Theta$-\textit{regular elements}. The next definition shows how this can be interpreted algebraically:

\begin{definition}
Let $\classY$ be a class of similar algebras, $\fancyL\subseteq Tm(\fancyL_{\classY},X)$, and $\theta\subseteq Eq(\fancyL_{\classY},\{x\})$. We say that $\theta$ is \textit{compatible with $\fancyL$} if for each $g(x_{1},...,x_{n})\in \fancyL$ we have that:
\begin{equation*}
    \theta(x_{1})\cup...\cup\theta(x_{n})\vDash_{\classY} \theta(g(x_{1},...,x_{n})).
\end{equation*}
Let $\mathcal{A}\in \classY$ be some algebra. Then we let $\theta(\mathcal{A})$ be the following structure,
\begin{equation*}
    \theta(\mathcal{A})\coloneqq \{a\in A : \mathcal{A}\vDash \theta(a)\}
\end{equation*}
equipped with the operations in $\fancyL'$: for each $n$-ary operation symbol $g(x_{1},...,x_{n})$, and $a_{1},...,a_{n}\in \theta(\mathcal{A})$, $g^{\theta(\mathcal{A})}(a_{1},...,a_{n})=g^{\mathcal{A}}(a_{1},...,a_{n})$. We call this the \textit{algebra of $\theta$-regular elements of $\mathcal{A}$}.
\end{definition}

Now assume that $\fancyL=\fancyL_{\classX}$. Then the assignment $\mathcal{A} \mapsto \theta(\mathcal{A})$ constitutes the object part of a functor $\theta:\classY\to \classX$. Given a homomorphism of $\fancyL$-algebras $f:\mathcal{A}\to \mathcal{B}$, we can define $\theta(f)$ as the restriction,  as the restriction, and obtain that this is a well-defined homomorphism in the language $\fancyL$. This defined $\theta$ as a functor. We can now see the connection between the notion of a contextual translation and such a functor: if $\classX$ and $\classY$ are two classes of algebras, and $\langle\zeta,\Theta\rangle:Tm(\fancyL_{\classX})\to Tm(\fancyL_{\classY})$ is a contextual translation between them, consider the following algebraic language
\begin{equation*}
    \fancyL_{\zeta}\coloneqq \{\zeta(\psi)\in Tm(\fancyL_{\classY}) : \psi\in \fancyL_{\classX}\}.
\end{equation*}

By definition of being a contextual translation we obtain that $\Theta$ is a set of equations, compatible with $\fancyL_{\zeta}$. Thus, this yields a functor in the manner described above, which we denote by $\theta_{\zeta}$ (when $\zeta$ is understood, we omit it). More than that, we have the following characterisation (for a proof see \cite{Moraschini2018}, Theorem 5.1 and Lemma 5.4):

\begin{theorem}\label{Explicit form of right adjoint functors in PA translations}
Let $\classX$ and $\classY$ be two quasivarieties. If $\langle \zeta,\Theta\rangle:Tm(\fancyL_{\classX})\to Tm(\fancyL_{\classY})$ is a contextual translation, then $\theta:\classY\to\classX$ is a right adjoint functor.
\end{theorem}

We also remark that being a right adjoint, the functor $\theta$ has a left adjoint, which we denote $\mathcal{F}$. An explicit description is proved in \cite[Corollary 5.2]{Moraschini2018}, and is captured by the following proposition.

\begin{notation}
    
Given a quasivariety $\bf{K}$ and $\Psi\subseteq \terms(\fancyL_{\bf{K}},X)^{2}$, we write $\mathsf{Cg}_{\bf{K}}(\Phi)$ for the congruence generated by these pairs of terms. Recall also that by general universal algebra, for each $\mathcal{A}$, and each set $X$ such that $|X|\geq |A|$ there is a congruence $\Psi$ on $\terms(\fancyL_{\bf{K}},X)$ such that $\mathcal{A}\cong \terms(\fancyL_{\bf{K}},X)/\Psi$. \footnote{Note that this depends on the size of $\mathcal{A}$, but by our earlier remark, we take the translation to be defined on any expansion of the set of terms by a larger collection of propositional variables.}
\end{notation}

\begin{proposition}\label{Explicit construction of the left adjoint}
Let $\classX$ and $\classY$ be two quasivarieties, and  $\langle\zeta,\Theta\rangle:Tm(\fancyL_{\classX})\to Tm(\fancyL_{\classY})$ a contextual translation, and $\theta$ be the induced right adjoint functor. Then the left adjoint functor, $\mathcal{F}:\classX\to \classY$, acts on objects as follows: if $\mathcal{A}\in \classY$ and $\Psi$ is a congruence on $\terms(\fancyL_{\classX},X)$ such that $\mathcal{A}\cong \terms(\fancyL_{\classX},X)/\Psi$. Then
\begin{equation*}
    \mathcal{F}(\mathcal{A})=\terms_{\fancyL_{\classY}}(X)/\mathsf{Cg}_{\classY}(\zeta^{*}(\Psi)\cup \bigcup_{x\in X}\Theta(x))
\end{equation*}
Where $\zeta^{*}(\Psi)=\{(\zeta_{*}(\lambda),\zeta_{*}(\gamma)) : (\lambda,\gamma)\in \Psi\}$.
\end{proposition}

Whilst this description can be useful, and we will make use of it in the next section, working with presentations of algebras can be cumbersome, making this construction difficult to work in practice. However, in many situations -- especially in the presence of a duality result -- we can have a better grasp of the left adjoint, as we will discuss regarding the GMT translation in \ref{GMT and classic Blok-Esakia}.

\subsection{Adjunctions and Selective Translations}\label{Adjunctions and Selective translations}

Having presented the basics of Moraschini's work, in this section we proceed to establish some basic facts regarding the algebras involved in the adjunctions described above. For this section we will assume some knowledge of category theory; we refer the reader to \cite{MacLane1978,Awodey2010-eu} for the necessary background. For the rest of this section, let $\classX$ and $\classY$ be two quasivarieties, and assume that $\langle \zeta,\Theta\rangle:Tm(\fancyL_{\classX})\to Tm(\fancyL_{\classY})$ is a unary contextual translation, with associated functor $\theta$.

\begin{notation}
We write:
\begin{equation*}
    \mathcal{A},\Theta\vDash \lambda\approx \gamma
\end{equation*}
to mean that $\mathcal{A}$ validates the equation $\lambda\approx\gamma$ in the context $\Theta$, that is, in any valuation where all variables ocurring in $\lambda,\gamma$ are assumed to satisfy $\Theta$.\qed \end{notation}

The first proposition we show will be heavily used throughout the rest of this article, and can be seen as a generalisation to any contextual translation of Theorem \ref{Translation theorem for the double negation translation}:

\begin{proposition}\label{Correctness of Translation in the unit}
For each $\mathcal{A}\in \classY$ and $\lambda\approx\gamma \in Eq(\fancyL_{\classX},X)$,
\begin{equation*}
    \mathcal{A},\Theta\vDash \zeta^{*}(\lambda\approx \gamma) \iff \theta(\mathcal{A}) \vDash\lambda\approx\gamma. 
\end{equation*}
\end{proposition}
\begin{proof}
First suppose that $\mathcal{A},\Theta\vDash \zeta^{*}(\lambda\approx \lambda)$. Let $v:\terms(\fancyL_{\classX},X)\to \theta_{\zeta}(A)$ be an arbitrary valuation. Then define a new valuation:
\begin{align*}
    w:\terms(\fancyL_{\classY},X) &\mapsto A\\
    x &\mapsto v(x)
\end{align*}
which is defined on variables, and lifted to all terms in $Tm(\fancyL_{\classY},X)$. Note that since $v$ was an assignment taking values in regular elements, then $w$ will also take values in regular elements. Moreover, we can see by induction on the construction of terms, and the construction of the translation, that for each $\delta\in Tm(\fancyL_{\classX},X)$:
\begin{equation*}
    \overline{v}(p\delta(\overline{x}))=\overline{w}(\zeta_{*}(\delta(\overline{x}))).
\end{equation*}
Since by assumption $\mathcal{A},\Theta\vDash \zeta^{*}(\lambda\approx\gamma)$, we obtain that $\theta(\mathcal{A}),w\vDash \delta\approx\gamma$, as intended. The other direction follows by similar arguments, noting that any valuation $v$ on $\mathcal{A}$ taking values in regular elements will be precisely defined on $\theta(\mathcal{A})$.\end{proof}

In this article, our interest will not be so much on general contextual translations, but rather in specific translations which essentially ``internalise" the context. These are the subject of our next definition:

\begin{definition}
Let $\langle\zeta,\Theta,f\rangle$ be a contextual unary translation together with a unary term $f(x)\in \fancyL_{\classY}$ called the \textit{selector}. We say that the triple $\langle\zeta,\Theta,f\rangle$ is a \textit{selective translation} if:
\begin{enumerate}
    \item $\eqrelY \Theta(f(x))$, i.e, selected elements are regular;
    \item $\Theta(x)\eqrelY f(x)\approx x$.
\end{enumerate}
\end{definition}

\begin{example}\label{Kolmogorov and GMT translations are selective}
The GMT and KGG translations we have encountered are both selective. Let the selector term be $f(x)\coloneqq \Box x$. Then by the transitivity and reflexivity axioms, for any S4-modal algebra $\mathcal{A}$, and $a\in A$, $\Box a =\Box\Box a$. By this fact and given the context is $\Box x$, we have the conditions. The arguments for the KGG translation are similar, using the selector term $\neg\neg x$, and using the fact that on Heyting algebras the equation $\neg\neg\neg\neg x\approx \neg\neg x$ holds.\end{example}

As the reader may have suspected, for selective translations the context $\Theta$ is redundant. The following lemma will be used, many times implicitly, throughout the paper: 

\begin{lemma}\label{Identity of the notions of regular elements}
    Let $\mathcal{A}\in \classY$. Then:
    \begin{equation*}
    \theta(\mathcal{A})=\mathcal{A}^{f}
    \end{equation*}
\end{lemma}
\begin{proof}
Assume that $a\in \theta(\mathcal{A})$. By assumption, then $\mathcal{A}\vDash \Theta(a)$. By the second condition of selectivity, we have that then $\mathcal{A}\vDash f(a)=a$. Hence $a\in \mathcal{A}^{f}$. Conversely, if $a\in \mathcal{A}^{f}$, then note that since by first condition of selectivity, we have that $\mathcal{A}\vDash \Theta(f(a))$, hence we obtain that $\mathcal{A}\vDash \Theta(a)$.
\end{proof}

The former allows us to dispense with the context $\Theta$, provided one makes a small modification in the interpretation of variables on the translation: rather than $\zeta_{*}(x)=x$, we define $\zeta_{*}(x)=f(x)$. Keeping in mind this proviso, we will use the simpler definition of selective translation throughout the rest of the paper.

In addition to capturing internally the meaning of contexts, the addition of the selector term has some categorical consequences. The following Proposition is an example of this:

\begin{proposition}\label{Selective translation preserves surjections}
Let $\langle \zeta,f\rangle:Tm(\fancyL_{\classX})\to Tm(\fancyL_{\classY})$ be a selective translation. Then the right adjoint functor, $\theta$, preserves surjective homomorphisms.
\end{proposition}
\begin{proof}
Suppose that $h:\mathcal{A}\to \mathcal{B}$ is a surjective homomorphism. Let $a \in \theta(\mathcal{B})$ be some regular element. Since $h$ is surjective, let $a'\in A$ be an element such that $h(a')=a$. Then note that:
\begin{equation*}
    \theta(h)(f(a')) = f(h(a')) = f(a) = a,
\end{equation*}
where the second equality follows from Corollary \ref{Identity of the notions of regular elements}, and the fact that the element is assumed to be regular. So $\theta(h)$ is surjective as intended.\end{proof}

One of the nice properties of the relationship between translations and adjunctions is that often one can translate natural properties of the adjunction into natural properties of the translation, and vice-versa. This was exploited in \cite{Moraschini2018} and \cite{moraschiniphd}, with a special focus on properties, such as the deduction theorem or the inconsistency lemma, which are relevant for abstract algebraic logic. Our focus, motivated by the analysis carried out in subsequent sections, is on properties of the unit and counit maps of the adjunction. To see this, we first recall some basic facts about these constructions.

Given  $F:\mathbb{C}\to \mathbb{D}$ a functor between two categories, we say that $F$ is \textit{faithful} if for all $\mathcal{A},\mathcal{B}\in \mathbb{C}$ and all maps $f,g:\mathcal{A}\to \mathcal{B}$: if $F(f)=F(g)$ then $f=g$; we say that it is \textit{full} if whenever $f:F(\mathcal{A})\to F(\mathcal{B})$ is a map on $\mathbb{D}$, then there is some map $g:\mathcal{A}\to \mathcal{B}$ such that $F(g)=f$. We say that $F$ is \textit{fully faithful} if it is both full and faithful. Additionally, we recall that all adjunctions have two natural transformations, $\eta:\mathbf{1}_{\mathbb{C}}\implies G(F(-))$ and $\varepsilon:F(G(-))\implies \mathbf{1}_{\mathbb{D}}$ which we call the \textit{unit} and the \textit{counit}, if for all $\mathcal{B}\in \mathbb{C}$ we have that the composition,
\begin{equation*}
    F(\mathcal{B})\xrightarrow{F(\eta_{\mathcal{B}})} F(G(F(\mathcal{B}))) \xrightarrow{\varepsilon_{F(\mathcal{B}})} F(\mathcal{B})
\end{equation*}
is equal to the identity on $\mathcal{B}$, and also, for all $\mathcal{A}\in \mathbb{D}$ we have that the composition,
\begin{equation*}
    G(\mathcal{A})\xrightarrow{\eta_{G(\mathcal{A})}} G(F(G(\mathcal{A}))) \xrightarrow{G(\varepsilon_{\mathcal{A}})} G(\mathcal{A})
\end{equation*}
is equal to the identity on $\mathcal{A}$. We say that $\eta$ is \textit{pointwise injective} (resp. surjective) if for each $\mathcal{A}\in \mathbb{C}$, $\eta_{\mathcal{A}}$ is injective (surjecitve). Similarly we say that $\varepsilon$ is pointwise injective/surjective\\

We will now provide a more explicit description of the \textit{unit} of the adjunction outlined in the previous section. For this purpose, note that given a congruence $\Psi\subseteq \terms(\fancyL_{\classX},Z)$, and $\mathcal{A}\cong \terms(\fancyL_{\classX},Z)/\Psi$ (what is sometimes referred to as a \textit{presentation} of $\mathcal{A}$ on the generators $Z$, with relations determined by $\Psi$), we can construct a map $m:\mathcal{A}\to \theta(\mathcal{F}(\mathcal{A}))$. To do so, for $\lambda\in Tm(\fancyL_{\classY},Z)$, write $[\lambda]_{\mathcal{F}(\mathcal{A})}$ for the equivalence class of $\lambda$ under the congruence:
\begin{equation*}
    \mathsf{Cg}_{\classY}(\zeta^{*}(\Phi)\cup \{(f(x),x) : x\in Z\}).
\end{equation*}
Note that because the translation is selective this is equivalent to the description given in Proposition \ref{Explicit construction of the left adjoint}. Similarly, when $\gamma\in Tm(\fancyL_{\classX},Z)$, we write $[\gamma]_{\mathcal{A}}$ to mean the equivalence class under the congruence $\Psi$ as given above.

Now for $x\in Z$, let $m(x)=[f(x)]_{\mathcal{F}(\mathcal{A})}$, where $f$ is the unary selector term; then by the universal mapping property this defines a map $m^{*}:Tm(\fancyL_{\classX},Z)\to \theta(\mathcal{F}(\mathcal{A}))$, which by compatibility of $\theta$ is such that:
\begin{equation*}
    m^{*}(\lambda)=[\zeta_{*}(\lambda)]_{\mathcal{F}(\mathcal{A})}.
\end{equation*}
Now note that if $(\lambda,\gamma)\in \Phi$, then by construction, $(\zeta_{*}(\lambda),\zeta_{*}(\gamma))\in \zeta^{*}(\Phi)$, which shows that $m^{*}(\lambda)=m^{*}(\gamma)$. Hence by standard algebraic facts, the map $m^{*}$ factors through $\mathcal{A}$. In other words, we have a well-defined homomorphism $m_{\mathcal{A}}:\mathcal{A}\to \theta(\mathcal{F}(\mathcal{A}))$, which acts by
\begin{equation*}
    m_{\mathcal{A}}([\lambda]_{\mathcal{A}})=[\zeta_{*}(\lambda)]_{\mathcal{F}(\mathcal{A})}.
\end{equation*}

\begin{lemma}\label{Explicit description of the unit}
    The map $m_{\mathcal{A}}:\mathcal{A}\to \theta(\mathcal{F}(\mathcal{A}))$ is the unit of the adjunction $\mathcal{F}\vdash \theta$.
\end{lemma}
\begin{proof}
Suppose that $g:\mathcal{A}\to \theta(\mathcal{B})$ is a homomorphism. Let $\pi_{\mathcal{A}}:\terms(\fancyL_{\classX},Z)\to \mathcal{A}$ be the projection map. Define
\begin{align*}
k:\terms(\fancyL_{\classY},Z)\to \mathcal{B}
\end{align*}
as follows: for each $x\in Z$, $k(z)=g(\pi_{\mathcal{A}}(z))$. By the universal mapping property, this lifts to a homomorphism $k^{*}$ on $\terms(\fancyL_{\classY},Z)$. Now assume that $(\zeta_{*}(\lambda),\zeta_{*}(\gamma))\in \zeta^{*}(\Phi)$, where $\lambda(x_{1},...,x_{n})$ and $\gamma(x_{1},...,x_{m})$ are two terms in $Tm(\fancyL_{\classX},Z)$. Then by definition, $(\lambda,\gamma)\in \Phi$, so $g(\pi_{\mathcal{A}}(\lambda))=g(\pi_{\mathcal{A}}(\gamma))$. It follows since these are homomorphisms, and since they take value in $\theta(\mathcal{B})$, that:
\begin{equation*}
    \zeta_{*}(\lambda)^{\mathcal{B}}(g(\pi_{\mathcal{A}}(x_{1})),...,g(\pi_{\mathcal{A}}(x_{n})))=\zeta_{*}(\gamma)^{\mathcal{B}}(g(\pi_{\mathcal{A}}(x_{1})),...,g(\pi_{\mathcal{A}}(x_{n}))),
\end{equation*}
hence $k^{*}(\zeta_{*}(\lambda))=k^{*}(\zeta_{*}(\gamma))$. Additionally, if $z\in Z$, taking the selector term $f$, we have,
\begin{equation*}
    k^{*}(f(z))=f^{\mathcal{B}}(g(\pi_{\mathcal{A}}(z)))=g(\pi_{\mathcal{A}}(z)),
\end{equation*}
since $g(\pi_{\mathcal{A}}(z))\in \theta(\mathcal{B})$. Hence $k^{*}$ factors through $\mathcal{F}(\mathcal{A})$, defining a map:
\begin{equation*}
    k_{\mathcal{F}(\mathcal{A})}:\mathcal{F}(\mathcal{A})\to \mathcal{B}.
\end{equation*}

such that whenever $\lambda\in \fancyL_{\classX}$, $k_{\mathcal{F}(\mathcal{A})}([\zeta^{*}(\lambda)]_{\mathcal{F}(\mathcal{A})})=g([\lambda]_{\mathcal{A}}$. Hence we have that:
\begin{equation*}
    \theta(k_{\mathcal{F}(\mathcal{A})})\circ m_{\mathcal{A}}=g;
\end{equation*}
since if $\lambda\in Tm(\fancyL_{\classX},Z)$, we have
\begin{align*}
    \theta(k_{\mathcal{F}(\mathcal{A})})(m_{\mathcal{A}}[\lambda]_{\mathcal{A}})
    &=\theta(k_{\mathcal{F}(\mathcal{A})})([\zeta_{*}(\lambda)]_{\mathcal{F}(\mathcal{A})})\\
    &=k_{\mathcal{F}(\mathcal{A})}([\zeta_{*}(\lambda)]_{\mathcal{F}(\mathcal{A})})\\
    &=g([\lambda]_{\mathcal{A}})
\end{align*}
Which shows that the diagram in Figure \ref{fig:unitdiagram} commutes. It is not difficult to see that the map $k_{\mathcal{F}(\mathcal{A})}$ is unique in these conditions: if $p:\mathcal{F}(\mathcal{A})\to \mathcal{B}$ was some arrow making the diagram commute, then note that whenever $x\in Z$, $m_{\mathcal{A}}([x]_{\mathcal{A}})=[f(x)]_{\mathcal{F}(\mathcal{A})}$, and hence by assumption we have that
\begin{equation*}
    g([x]_{\mathcal{A}})=p([f(x)]_{\mathcal{F}(\mathcal{A})})=k([f(x)]_{\mathcal{F}(\mathcal{A})});
\end{equation*}
however, since $(f(x),x)\in \mathsf{Cg}_{\classY}(\zeta^{*}(\Phi)\cup\{(f(x),x) : x\in Z\})$, we have that $[f(x)]_{\mathcal{F}(\mathcal{A})})=[x]_{\mathcal{F}(\mathcal{A})}$, and so $k([x]_{\mathcal{F}(\mathcal{A})})=p([x]_{\mathcal{F}(\mathcal{A})})$. This immediately implies that $k=p$, concluding the proof.\end{proof}

\begin{figure}[h]
    \centering
\begin{tikzcd}
\mathcal{A} \arrow[r, "m_{\mathcal{A}}"] \arrow[rd, "g"'] & \theta(\mathcal{F}(\mathcal{A})) \arrow[d, "\theta(k_{\mathcal{F}(\mathcal{A})})"] \\
                                                & \mathcal{B}                                                                  
\end{tikzcd}    \caption{Unit diagram}
    \label{fig:unitdiagram}
\end{figure}

Using this explicit description, we can prove, as promised, a connection between properties of the unit and the structure of the translation.

\begin{definition}
    Let $\langle \zeta,f\rangle: Tm(\fancyL_{\classX})\to Tm(\fancyL_{\classY})$ be a selective translation. We say that $\langle \zeta,f\rangle$ is \textit{faithful} if for every collection $\Phi\cup\{\lambda\approx \gamma\}\subseteq Eq(\fancyL_{\classY},X)$ we have
    \begin{equation*}
        \Phi\eqrelX \lambda\approx \gamma, \text{ iff } \zeta^{*}(\Phi)\cup \{(f(x),x) : x\in X\} \eqrelY\zeta^{*}(\lambda\approx \gamma).
    \end{equation*}
\end{definition}

The following appeared originally in \cite[Lemma 6.4]{moraschiniphd}, where a different proof was given:

\begin{proposition}\label{Unit is a split mono}
If $\langle \zeta,f\rangle$ is a selective translation as above, then $\zeta$ is faithful if and only if for each $\mathcal{A}\in \classX$, $\eta_{\mathcal{A}}$ is injective.
\end{proposition}
\begin{proof}
Assume that $\eta$ is pointwise injective. One half of the above condition is true of every contextual translation, so we focus on the other one. Suppose that $\Phi\nvDash_{\classX}\lambda\approx \gamma$. Let $\mathcal{A}\in\classX$ be some algebra, such that $\mathcal{A},v\vDash \Phi$, but $\mathcal{A},v\nvDash \lambda\approx\gamma$. Now define a valuation $w:X\to \theta(\mathcal{F}(\mathcal{A}))$ by letting $w(x)=\eta_{\mathcal{A}}(v(x))$. By induction, we can then show that for $\delta\in Tm(\fancyL_{\classX},X)$ $w(\delta)=\eta_{\mathcal{A}}(v(\delta))$. Note that we have that if $\delta\approx\varepsilon\in \Phi_{eq}$, since $v(\delta)=v(\varepsilon)$, then $w(\delta)=w(\varepsilon)$; additionally, if $w(\lambda)=w(\gamma)$, then by injectivity $v(\lambda)=v(\gamma)$, which is a contradiction. So $\theta(\mathcal{F}(\mathcal{A})),w\vDash \Phi$, and $\theta(\mathcal{F}(\mathcal{A})),w\nvDash \lambda\approx\gamma$. But then using the arguments from Proposition \ref{Correctness of Translation in the unit}, we get that $\mathcal{F}(\mathcal{A}),v'\vDash \zeta^{*}(\Phi)$ whilst $\mathcal{F}(\mathcal{A}),v'\nvDash \zeta^{*}(\lambda\approx\gamma)$. This shows that:
\begin{equation*}
    \zeta^{*}(\Phi)\cup \{f(x)\approx x: x\in X\}\nvDash_{\classY}\zeta^{*}(\lambda\approx \gamma)
\end{equation*}
as desired.

Now assume that $\eta$ is not pointwise injective, say for an algebra $\mathcal{A}\cong \terms(\fancyL_{\classX},X)/\Phi$. By Lemma \ref{Explicit description of the unit}, this means that there are two terms $\lambda,\gamma\in Tm(\fancyL_{\classX},X)$ such that $(\lambda,\gamma)\notin \Phi$, and $(\zeta_{*}(\lambda),\zeta_{*}(\gamma))\in \mathsf{Cg}(\zeta^{*}(\Phi)\cup \{f(x)\approx x: x\in X\})$. Hence first note that $\Phi\nvDash_{\classX}\lambda\approx\gamma$, since the identity valuation witnesses this on $\mathcal{A}$. On the other hand, let $\mathcal{B}\in \classY$ be any algebra, and let $v$ be a valuation such that $\mathcal{B},v\vDash \zeta^{*}(\Phi)\cup \{f(x)\approx x :x\in X\}$. Define a map $p:Tm(\fancyL_{\classY},X)\to B$ by letting $p(x)=v(x)$, and lifting accordingly. Then by the assumption of $\bf{B}$, we obtain that the kernel of $p$ contains $\mathsf{Cg}(\zeta^{*}(\Phi)\cup \{f(x)\approx x: x\in X\})$; hence this factors as a homomorphism $p^{*}:\mathcal{F}(\mathcal{A})\to \mathcal{B}$, such that for each $\delta\in Tm(\fancyL_{\classY},X)$,  $p^{*}([\delta]_{\mathcal{F}(\mathcal{A})})=v(\delta)$. Now since $[\zeta_{*}(\lambda)]_{\mathcal{F}(\mathcal{A})}=[\zeta_{*}(\gamma)]_{\mathcal{F}(\mathcal{A})}$, we then obtain that $v(\zeta_{*}(\lambda))=v(\zeta_{*}(\gamma))$. This shows that $\mathcal{B},v\vDash \zeta_{*}(\lambda)\approx\zeta_{*}(\gamma)$, and thus establishes that:
\begin{equation*}
    \zeta^{*}(\Phi)\cup \{f(x)\approx x: x\in X\}\vDash_{\classY} \zeta_{*}(\lambda)\approx\zeta_{*}(\gamma) \text{ and } \Phi\nvDash_{\classX}\lambda\approx\gamma
\end{equation*}
This concludes the proof.\end{proof}

Most translations one is interested tend to be faithful translations. Indeed, as noted in \ref{Subsection: Classical Translation}, the GMT, double negation and Goldblatt\footnote{For the GMT and double negation translation the reader may consult \cite{Moraschini2018}; for a treatment of the Goldblatt translation as a selective translation, we refer the reader to \cite[Chapter 5]{minhatesedemestrado}, where this is discussed in detail.} translation are all faithful. However, these properties do not suffice for many of our purposes in the next section. As such we introduce specialised classes of translations which are obtained by strengthening the properties which hold of the adjunction.

\begin{definition}
    Let $\langle \zeta,f\rangle$ be a selective translation. We say that $\zeta$ is \textit{strongly selective} if additionally we have:
\begin{enumerate}
    \item[3.] The unit $\eta$ of the adjunction is an isomorphism.
\end{enumerate}
\end{definition}

Before we turn to some examples, we note that our description of the unit allows us to derive a sufficient condition for strong selectivity:

\begin{definition}
    Let $\langle \zeta,f\rangle:Tm(\fancyL_{\classX},X)\to Tm(\fancyL_{\classY},Y)$ be a selective translation. We say that $\zeta$ is \textit{essentially full} if for each $n$-ary term of the form $f(\lambda(f(x_{1}),...,f(x_{n}))\in Tm(\fancyL_{\classY},Y)$, there is some $n$-ary term $\gamma(x_{1},...,x_{n})\in Tm(\fancyL_{\classX},X)$ such that:
    \begin{equation*}
        \vDash_{\classY} f(\lambda(f(x_{1}),...,f(x_{n})) \approx \zeta_{*}(\gamma(x_{1},...,x_{n})).
    \end{equation*}
\end{definition}

\begin{lemma}\label{Essentially full implies surjective}
    Let $\langle \zeta,f\rangle$ be an essentially full translation, as above. Then for each $\mathcal{A}\in \classX$, $\eta_{\mathcal{A}}$ is surjective.
\end{lemma}
\begin{proof}
Let $\mathcal{A}\in \classX$ be arbitrary, assume that $\mathcal{A}\cong \terms(\fancyL_{\classX},X)/\Phi$, and let $[\lambda(x_{1},...,x_{n})]_{\theta(\mathcal{F}(\mathcal{A}))}\in \theta(\mathcal{F}(\mathcal{A}))$ be some element. Now because $(f(x_{i}),x_{i})\in \mathsf{Cg}(\zeta^{*}(\Phi)\cup \{f(x)\approx x : x\in X\})$, we have that:
\begin{equation*}
    [\lambda(x_{1},...,x_{n})]_{\mathcal{F}(\mathcal{A})}=[\lambda(f(x_{1}),...,f(x_{n})]_{\mathcal{F}(\mathcal{A})};
\end{equation*}
moreover, since the element is regular:
\begin{equation*}
    [\lambda(x_{1},...,x_{n})]_{\mathcal{F}(\mathcal{A})}=[f(\lambda(f(x_{1}),...,f(x_{n})))]_{\mathcal{F}(\mathcal{A})}.
\end{equation*}
Now by essential fullness, we have some $\gamma(x_{1},...,x_{n})\in Tm(\fancyL_{\classX},X)$ such that we obtain
\begin{equation*}
    [\zeta^{*}(\gamma(x_{1},...,x_{n})]_{\mathcal{F}(\mathcal{A})}=[\lambda(x_{1},...,x_{n})]_{\mathcal{F}(\mathcal{A})}.
\end{equation*}
Hence by our description of the unit map in Lemma \ref{Explicit description of the unit}, we have that $\eta_{\mathcal{A}}([\gamma]_{\mathcal{A}})=[\zeta^{*}(\gamma)]_{\mathcal{F}(\mathcal{A})}=[\lambda]_{\mathcal{F}(\mathcal{A})}$. This shows that $\eta_{\mathcal{A}}$ is surjective.\end{proof}

\begin{example}\label{KGG translation is strongly selective}
    The KGG is a strongly selective translation. This fact was shown in \cite[Lemma 4.2]{Tur2008}, where the authors showed that the unit is an isomorphism. However, having established the translation is faithful, one can also use Lemma \ref{Essentially full implies surjective} to show this. By induction on formulae, we can show that if $\phi$ is equivalent to a formula of the form$   \neg\neg \psi(\neg\neg\overline{p})$, then there is some classical logic formula $\psi$ such that $K(\psi)=\phi$: if $\phi=\neg\neg p$, then $\psi=p$ works; for $\bot$ and $\top$ this trivially holds. Now if this holds for $\chi_{i}$, where we assume we have $\mu_{i}$ classical formulas such that $K(\mu_{i})=\neg\neg\chi_{i}(\neg\neg\overline{p})$ then if $\phi=\chi_{0}\wedge \chi_{1}$, note that:
    \begin{equation*}
        \neg\neg(\chi_{0}(\neg\neg\overline{p})\wedge \chi_{1}(\neg\neg\overline{p}))=\neg\neg\chi_{0}(\neg\neg\overline{p})\wedge \neg\neg \chi_{1}(\neg\neg\overline{p})=K(\mu_{0})\wedge K(\mu_{1})=K(\mu_{0}\wedge \mu_{1})
    \end{equation*}
    and
    \begin{equation*}
        \neg\neg(\chi_{0}(\neg\neg\overline{p})\vee \chi_{1}(\neg\neg\overline{p}))=\neg\neg (K(\mu_{0})\vee K(\mu_{1}))=K(\mu_{0}\vee \mu_{1})
    \end{equation*}
    and
    \begin{equation*}
        \neg\neg(\chi_{0}((\neg\neg\overline{p})\rightarrow\chi_{1}(\neg\neg\overline{p}))=\neg\neg \chi_{0}(\neg\neg\overline{p})\rightarrow\neg\neg \chi_{1}(\neg\neg\overline{p}) = K(\mu_{0})\rightarrow K(\mu_{1}) = K(\mu_{0}\rightarrow \mu_{1}).
    \end{equation*}
This then shows the essential fullness, which entails that KGG is a strongly selective translation.

The GMT translation is also strongly selective; the property of essential fullness is well-known in this case (see e.g. \cite[Theorem 14.9]{Chagrov1997-cr}). Indeed, more is true, as we will show in Section \ref{GMT and classic Blok-Esakia}.
\end{example}

Not all selective translations are strongly selective; it was shown in \cite[Proposition 5.3.34]{minhatesedemestrado} that the Goldblatt translation, whilst faithful, is not strongly selective.

We can further strengthen the adjunction by imposing conditions on the \textit{counit}, and a preservation requirement on the left adjoint.

\begin{definition}
    Let $\langle \zeta,f\rangle$ be a selective translation. $\zeta$ is \textit{sober} if it is strongly selective and additionally we have:
\begin{enumerate}
    \item[4.] The counit $\varepsilon$ of the adjunction is pointwise injective.
    \item[5.] The left adjoint $\mathcal{F}$ preserves injective homomorphisms.
\end{enumerate}
\end{definition}

For these conditions we have not been able to find syntactic counterparts.

\begin{example}
    In the next section we will show that the GMT translation is sober. On the other hand it was shown in \cite{minhatesedemestrado} that the double negation translation is not sober, since the counit is not an isomorphism.
\end{example}

The intuition behind sober translations, in particular, condition 4, lies originally in the \textit{duality theory} of intuitionistic, classical and modal logic. Through duality, given the adjunction $\mathcal{F}\vdash \theta$, we can obtain an adjunction on the dual categories $\mathcal{F}^{*}\vdash\theta^{*}$, which usually allows for a more transparent description of the overall transformations. For instance, the map $\theta^{*}:\mathbf{S4}\to \mathbf{ES}$ from $S4$-modal spaces to Esakia spaces, operates by collapsing all clusters, which may exist on these spaces, forming the so-called ``skeleton". However, the map $\theta^{*}:\mathbf{ES}\to \mathbf{ST}$ from Esakia spaces to Stone spaces, associated to the adjunction of the KGG translation, operates by looking at a poset $P$, which is required to have maximal elements, and considering $\mathsf{Max}(P)$. This can be seen as ``forgetting" all points of the partially ordered set which are not maximal. This is the mechanism exploited in \cite{minhatesedemestrado} to show that the counit is not an injection, and hence, the translation is not sober. Morally, then, sober translations are those where the adjunction cannot ``erase" any structure which is meaningful from the point of view of the translation, and can be seen as an intuitive justification for why sobriety appears (in the next section) to be crucial in guaranteeing that several structures are well-defined.

\section{A General Blok-Esakia Theory}\label{Section: General Blok-Esakia Theory}

We now make use of the classes of translations introduced in the previous section to develop a general ``Blok-Esakia theory". By this we mean a theory which explains the way that the translation operates on the lattices of extensions of the translated systems, and discussing which structural properties of the systems -- such as FMP, tabularity, decidability, amongst others -- can be transferred along the translation.

With this goal in mind, in Section \ref{GMT and classic Blok-Esakia} we recall the basics of this theory, and provide parallels which are used in Section \ref{Strongly Selective and Sober translations} to derive the main results.

\subsection{The GMT Translation and Classic Blok-Esakia}\label{GMT and classic Blok-Esakia}

In this subsection we recall some facts from ``Blok-Esakia theory". All of the facts concerning companions, the adjunction between Heyting algebras and $\mathbf{S4}$-algebras, as well as well as the transfer properties between companions, are well-known. We refer the reader to \cite{Chagrov1992,Chagrov1997-cr} for an in-depth discussion of all of these results. We also recommend \cite{stronkowskiblokesakia}, \cite{Esakiach2019HeyAlg} and especially \cite{antoniocleani} for good proofs of some facts which we claim below. We provide full proofs of some facts which turn out to have all the relevant ideas which are needed for the general case, and we prefer to outline them in this more concrete setting.

We begin with the notion of a modal companion which grounds the whole endeavour:

\begin{definition}
For $M\in \mathbf{NExt}(\mathsf{S4})$ and $L\in \Lambda(\mathsf{IPC})$, logics in the sense of axiomatic extensions, we say that $M$ is a \textit{modal companion} of $L$ iff:
\begin{equation*}
\phi\in L \iff GMT(\phi)\in M.
\end{equation*}
\end{definition}

It is well known -- and can be seen, by combining Theorem \ref{Translation theorem for the double negation translation} with some facts we mention below -- that $\mathsf{S4}$ is a modal companion of $\mathsf{IPC}$. Nevertheless, it is far from the only one. Another well-known companion is the logic known as $\mathsf{Grz}$: this can be axiomatised over $\mathsf{S4}$ as:
\begin{equation*}
  \mathsf{Grz}\coloneqq \mathsf{K}\oplus \Box(\Box(p\rightarrow \Box p)\rightarrow p)\rightarrow p. 
\end{equation*}

It was shown by Grzegorczyk \cite{Grzegorczyk1967-bb} that this logic is a modal companion of $\mathsf{IPC}$. Additionally, it was shown by Esakia \cite{esakiapaperabouts41} that the logic $\mathsf{S4.1}$ axiomatised over $\mathsf{S4}$ by $\Box\dia p\rightarrow\dia\Box p$ is also a companion of $\mathsf{IPC}$. Given this diversity of companions, one generally makes use of three maps, carrying logics to logics. The first two take us from axiomatic extensions of $\mathsf{IPC}$ to axiomatic extensions respectively of $\mathsf{S4}$ and $\mathsf{Grz}$, and are often denoted by $\tau$ and $\sigma$:
\begin{align*}
    \tau: \Lambda(\mathsf{IPC})&\to \mathbf{NExt}(\mathsf{S4})\\
    L &\mapsto \mathsf{S4}\oplus \{GMT(\phi) : \phi\in L\}\\
    \sigma: \Lambda(\mathsf{IPC})&\to \mathbf{NExt}(\mathsf{Grz})\\
    L &\mapsto \mathsf{Grz}\oplus \{GMT(\phi) : \phi\in L\}
\end{align*}

By definition, we have that given $L\in \Lambda(\mathsf{IPC})$, $\tau(L)$ and $\sigma(L)$ are logics, though we have no guarantee that they are -- as desired -- modal companions, since they may end up proving more translations. In the opposite direction, one defines the following map, from logics $M\in \mathbf{NExt}(\mathsf{S4})$ to sets of intuitionistic formulas:

\begin{equation*}
    \rho(M)=\{\phi:GMT(\phi)\in M\}
\end{equation*}

The definition is suggested by the notion of modal companion, and one would want to say that $\rho(M)\in \Lambda(\mathsf{IPC})$ is a logic, but this is again not immediate, since closure under substitution cannot be transferred -- at least without further considerations -- along the translation. A way to prove that these maps are well-defined, and to obtain a better grip of what they capture proceeds by looking at the algebraic semantics of these logics, and associates to these syntactic assignments some semantic ones. To do so, we start by describing in detail the adjunction arising from this translation.

Let $(-)_{\Box}:\mathbf{S4}\to\mathbf{HA}$ be the functor which maps an S4-algebra $\mathcal{B}$ to $\mathcal{B}_{\Box}$ (see Definition \ref{Heyting algebras and S4}), and which takes every homomorphism of S4-algebras $f:\mathcal{B}\to \mathcal{B}'$ o $f\restriction_{\Box}:\mathcal{B}_{\Box}\to \mathcal{B}'_{\Box}$, the restriction of $f$, following the general recipe outlined in  \ref{Algebraic Translations and Adjunctions}. We thus know that $(-)_{\Box}$ is a right adjoint functor, and from Proposition \ref{Selective translation preserves surjections}\footnote{The special case of this Proposition applying to the setting of the GMT translation was known, and can be found for example in \cite[Theorem 2.25]{Esakiach2019HeyAlg}.} we obtain that it preserves surjective homomorphisms in addition to all limits (including injective homomorphisms and products, which are specific kinds of limits).

Its corresponding left adjoint functor $B(-):\mathbf{HA}\to \mathbf{S4}$ is obtained as by taking the \textit{Boolean envelope} (\cite[Construction 2.5.7]{Esakiach2019HeyAlg}, see also \cite[pp.25]{gehrkevangoolnewbook} for a broader discussion): for each $\mathcal{H}$ a Heyting algebra, let $\mathsf{Pr}(\mathcal{H})$ denote the set of prime filters of $\mathcal{H}$. Define a map $\phi:H\to \mathbb{P}(\mathsf{Pr}(\mathcal{H}))$, where $\mathbb{P}(\mathsf{Pr}(\mathcal{H}))$ is the Boolean algebra of subsets, by setting $\phi(a)=\{x\in \mathsf{Pr}(\mathcal{H}) : a\in x\}$. $\phi$ is an injective lattice homomorphism, and we can let $B(\mathcal{H})$ be the Boolean subalgebra of $\mathbb{P}(\mathsf{Pr}(\mathcal{H}))$ generated by $\phi[\mathcal{H}]$. As a result, every element $a\in B(\mathcal{H})$ can be written in the form
\begin{equation*}
    \bigwedge_{i=1}^{n} \neg c_{i} \vee d_{i}
\end{equation*}
where $c_{i},d_{i}\in H$. We then define a $\Box$-modality on this Boolean algebra, by setting, for each $a\in B(\mathcal{H})$, where $a=\bigwedge_{i=1}^{n} \neg c_{i} \vee d_{i}$:
\begin{equation*}
    \Box a = \bigwedge_{i=1}^{n} c_{i}\rightarrow d_{i},
\end{equation*}
where $c_{i}\rightarrow d_{i}$ is the Heyting implication.

On maps, given $f:\mathcal{H}\to \mathcal{H}'$ a Heyting algebra homomorphism, we let $B(f):B(\mathcal{H})\to B(\mathcal{H}')$ be the unique lift of the map $f$ to the Boolean envelope, which can be shown to preserve the modality since $f$ is a Heyting homomorphism. Then we have that (see \cite[Proposition 2.5.9]{Esakiach2019HeyAlg}):
\begin{itemize}
    \item The functor $B(-)$ is a left adjoint, and hence preserves all colimits, and also preserves all finite limits.
    \item For each product $\prod_{i\in I}\mathcal{H}_{i}$ of Heyting algebras, $B(\prod_{i\in I}\mathcal{H}_{i})$ is a subalgebra of $\prod_{i\in I}B(\mathcal{H}_{i})$.
\end{itemize}

Having the above explicit description also makes it possible to prove, without too much difficulty, the following two facts \cite[Theorem 2.5.11]{Esakiach2019HeyAlg}:
\begin{itemize}
    \item The unit map $\eta$ is an isomorphism.
    \item The counit map $\varepsilon$ is pointwise injective.
\end{itemize}

Hence, we obtain that:

\begin{corollary}\label{The GMT translation is sober}
    The GMT translation is sober.
\end{corollary}
\begin{proof}
    We have noted in Example \ref{KGG translation is strongly selective} that the GMT translation is strongly selective, which in light of the above yields that it is sober. 
\end{proof}

Note that as a result of the unit being an isomorphism, whenever $\bf{H}\nvDash \phi$, then $B(\bf{H})\nvDash GMT(\phi)$; this is because since $\bf{H}\cong\theta(B(\bf{H}))$, this follows from Theorem \ref{Translation theorem for the double negation translation}. This thus allows us, as promised, to show that $\mathsf{S4}$ is a modal companion.

Using these transformations we can readily define appropriate maps on the lattice of varieties\footnote{We recall that varieties correspond to axiomatic extensions of our systems, and in this section we are focusing on this setting, it being the classical place of analysis in the literature}. Given a variety $\bf{K}$, denote the lattice of subvarieties of $\bf{K}$ as $\Xi(\bf{K})$. Then given $\bf{K}\in \Xi(\bf{HA})$, we write
\begin{equation*}
    \tau(\mathbf{K})=\{ \mathcal{B}\in \mathbf{S4} : \mathcal{B}_{\Box}\in \mathbf{K}\}.
\end{equation*}

The ambiguity in denoting this as $\tau$ is intentional, and is explained by the next proposition. It was essentially shown by Dummett and Lemmon \cite{Dummett1959}, though it is presented here in a different way.

\begin{proposition}\label{Key Properties of the tau-map in the S4 case}
Let $L\in \Lambda(\mathsf{IPC})$ and $\mathbf{K}\in \Xi(\mathbf{HA})$. Then:
\begin{enumerate}
    \item $\tau(\mathbf{K})$ is a variety.
    \item $\tau(\mathsf{Var}(L))=\mathsf{Var}(\tau(L))$. Hence $\tau(\mathsf{Log}(\mathbf{K}))=\mathsf{Log}(\tau(\mathbf{K}))$
    \item $\tau$ is a complete homomorphism on the lattice of logics.
    \item $\tau(L)$ is the least modal companion of $L$;
\end{enumerate}
\end{proposition}
\begin{proof}
To see (1), suppose that $\mathcal{A}\in \mathsf{Var}(\tau(\mathbf{K}))$. By Tarski's HSP theorem, assume that $\mathcal{A}$ is a homomorphic image of $\mathcal{B}$, which is a subalgebra of $\prod_{i\in I}\mathcal{B}_{i}$, and $(\mathcal{B}_{i})_{\Box}\in \bf{K}$. Since $(-)_{\Box}$ commutes with all limits and surjective homomorphisms, then $\mathcal{A}_{\Box}$ is a homomorphic image of $\mathcal{B}_{\Box}$ which is a subalgebra of $\prod_{i\in I}(\mathcal{B}_{i})_{\Box}$. Hence $\mathcal{A}_{\Box}\in K$, and thus, $\mathcal{A}\in \tau(K)$.

To see (2), we see that for each $\mathbf{S4}$-algebra $\mathcal{B}$, $\mathcal{B}\in \tau(\mathsf{Var}(L))$ if and only if $\mathcal{B}_{\Box}\in \mathsf{Var}(L)$, if and only if $\mathcal{B}_{\Box}\vDash L$ if and only if $\mathcal{B}\vDash \{GMT(\phi) : \phi\in L\}$ (by Theorem \ref{Translation theorem for the double negation translation}), if and only if $\mathcal{B}\in \mathsf{Var}(\tau(L))$. The second statement follows by the first-one, using algebra-logic duality.

To see (3) first note that $\tau(\bigcap_{i\in I}\mathbf{K}_{i})=\bigcap_{i\in I}\tau(\mathbf{K}_{i})$ by definition: $\mathcal{A}\in \tau(\bigcap_{i\in I}\mathbf{K}_{i})$ if and only if $\mathcal{A}_{\Box}\in \bigcap_{i\in I}\mathbf{K}_{i}$, if and only if $\mathcal{A}_{\Box}\in \mathbf{K}_{i}$ for each $i$, if and only if $\mathcal{A}\in \tau(\mathbf{K}_{i})$ for each $i$, if and only if $\mathcal{A}\in \bigcap_{i\in I}\tau(\mathbf{K}_{i})$. For the join, first set $L_{i}=\mathsf{Log}(\bf{K}_{i})$. Note that since
\begin{equation*}
   \mathsf{S4}\oplus\{GMT(\phi) : \phi\in \bigcup_{i\in I}L_{i}\}\subseteq \mathsf{S4}\oplus\{GMT(\phi) : \phi\in \bigvee_{i\in I}L_{i}\}
\end{equation*}
then we have that $\bigvee_{i\in I}\tau(L_{i})\subseteq \tau(\bigvee_{i\in I}L_{i})$. To show the other inclusion, by algebraic completeness and the fact showed in (1) that these are really varieties, it suffices to show that $\mathsf{Var}(\bigvee_{i\in I}\tau(L_{i}))\subseteq \mathsf{Var}(\tau(\bigvee_{i\in I}L_{i}))$, which by (2) means that, $\bigvee_{i\in I}\tau(\bf{K}_{i})\subseteq \tau(\bigvee_{i\in I}\bf{K}_{i})$. Assume that $\mathcal{A}$ is a homomorphic image of $\mathcal{B}$, which is a subalgebra of $\mathcal{C}=\prod_{i\in I}\mathcal{D}_{i}$, and $\mathcal{D}_{i}$ all belong to $\bf{K}_{i}$. Using the fact that $(-)_{\Box}$ preserves all limits and surjective homomorphisms, we obtain that $\mathcal{B}_{\Box}\in \bigvee_{i\in I}\bf{K}_{i}$, and so $\mathcal{B}\in \tau(\bigvee_{i\in I}\bf{K}_{i})$ as desired.

To see (4), notice that by definition if $\phi\in L$ then $GMT(\phi)\in \tau(L)$. Conversely, if $\phi\notin L$, let $\mathcal{H}\in \mathsf{Var}(L)$ be such that $\mathcal{H}\nvDash \phi$; because $\eta$ is an isomorphism, we know that there is some $\mathcal{B}$ such that $\mathcal{B}_{\Box}\cong \mathcal{H}$; hence, $\mathcal{B}\in \tau(\mathsf{Var}(L))$. By (2), then $\mathcal{B}\in \mathsf{Var}(\tau(L))$, i.e, $\mathcal{B}\vDash \tau(L)$, and $\mathcal{B}\nvDash GMT(\phi)$. This proves that $\tau(L)$ is a modal companion. It is clear to see that it must be least.\end{proof}

We now turn to the other direction. Paralleling the assignment $\rho$ above, we can define:
\begin{align*}
    \rho(\mathbf{K})\coloneqq \{\mathcal{B}_{\Box} : \mathcal{B}\in \mathbf{K}\}.
\end{align*}

The following proposition, essentially established by Esakia \cite{esakiapaperabouts41}, tells us that this is a well-defined map, and also that the map $\rho$ on logics gives us genuine modal companions. We call the attention to the reader to the different properties of the adjunction we will use here, compared to $\tau$:

\begin{proposition} \label{Key Properties of the rho-map in S4}
Let $\mathbf{K}\in \Xi(\mathbf{S4})$ and $M\in \Lambda(\mathsf{S4})$. Then:
\begin{enumerate}
    \item $\rho(\mathbf{K})$ is a variety.
    \item $\rho(\mathsf{Var}(M))=\mathsf{Var}(\rho(M))$. Hence the map $\rho:\mathbf{NExt}(\mathsf{S4})\to \Lambda(\mathsf{IPC})$ defined as $\rho(L)=\{\phi : GMT(\phi)\in L\}$ is well-defined.
    \item $\rho$ is a surjective complete homomorphism on the lattice of varieties.
    \item For all $N\in \mathbf{NExt}(\mathsf{S4})$, and $L\in \Lambda(\mathsf{IPC})$, $N$ is a modal companion of $L$ if and only if $\rho(N)=L$. Hence, for all $\mathbf{K}\in \mathsf{Var}(\mathsf{S4})$ and $\mathbf{P}\in \mathsf{Var}(\mathsf{HA})$ we have that if $\mathsf{Log}(\bf{K})$ is a modal companion of $\mathsf{Log}(\mathbf{P})$ then $\rho(\mathbf{K})=\mathbf{P}$.
\end{enumerate}
\end{proposition}
\begin{proof}
(1) Suppose that $\mathcal{A}\in \mathsf{Var}(\rho(\mathbf{K}))$. By Tarski's HSP theorem, $\mathcal{A}$ is a homomorphic image of $\mathcal{B}$, which is a subalgebra of $\prod_{i\in I}(\mathcal{C}_{i})_{\Box}$, where $\mathcal{C}_{i}\in \mathbf{K}$. Note that since $(-)_{\Box}$ preserves products, the latter is $(\prod_{i\in I}\mathcal{C}_{i})_{\Box}$. Hence, since $B(-)$ preserves homomorphic images, $B(\mathcal{A})$ is a homomorphic image of $B(\mathcal{B})$ which is a subalgebra of $B((\prod_{i\in I}\mathcal{C}_{i})_{\Box})$; since the counit is an injective homomorphism, $B((\prod_{i\in I}\mathcal{C}_{i})_{\Box})$ is a subalgebra of $\prod_{i\in I}\mathcal{C}_{i}$, and so $B(\mathcal{B})$ is as well. Thus, $B(\mathcal{A})\in \mathbf{K}$, since the latter is a variety. Since $(B(\mathcal{A}))_{\Box}\cong \mathcal{A}$, given the unit is an isomorphism, we have that $\mathcal{A}\in \rho(\mathbf{K})$.

To see (2), first note that if $\mathcal{H}\in \rho(\mathsf{Var}(M))$ then $\mathcal{H}=\mathcal{B}_{\Box}$ where $\mathcal{B}\vDash M$; hence, if $GMT(\phi)\in M$, then $\mathcal{H}\vDash \phi$, so $\mathcal{H}\vDash \rho(M)$. Thus $H\in \mathsf{Var}(\rho(M))$. Conversely, assume that $\phi\in \rho(M)$, and let $\mathcal{H}$ be such that $\mathcal{B}_{\Box}=\mathcal{H}$ and $\mathcal{B}\vDash M$; since $\phi\in \rho(M)$, then $GMT(\phi)\in M$, so $\mathcal{B}\vDash GMT(\phi)$, and hence, $\mathcal{H}\vDash \phi$ so:
\begin{equation*}
    \rho(M)\subseteq \mathsf{Log}(\rho(\mathsf{Var}(M))
\end{equation*}
Hence by algebraic completeness, and the fact that $\rho(\mathsf{Var}(M))$ is a variety:
\begin{equation*}
    \mathsf{Var}(\rho(M))\subseteq \mathsf{Var}(\mathsf{Log}(\rho(\mathsf{Var}(M))))=\rho(\mathsf{Var}(M))
\end{equation*}
This shows the other inclusion. Hence we have that $\rho(\mathsf{Var}(M))=\mathsf{Var}(\rho(M))$, which by logic-algebra duality and (1) implies the remaining facts.

(3) The verification that this is a complete homomorphism proceeds by similar arguments to those of Proposition \ref{Key Properties of the tau-map in the S4 case} (3). To see that it is surjective, simply note that if $\mathbf{K}$ is any variety of Heyting algebras, then $\tau(\mathbf{K})=\{\mathcal{B} : \rho(\mathcal{B})\in \mathbf{K}\}$ is such that $\rho\tau(\mathbf{K})=\mathbf{K}$.

(4) Note that given any such $N$ and $L$, we have that $\rho(N)=L$ if and only if ``$\phi\in L$ if and only if $GMT(\phi)\in M$" if and only if $M$ is a modal companion of $L$. To see the last statement, assume that $\mathsf{Log}(\bf{K})$ is a modal companion of $\mathsf{Log}(\bf{P})$. Then $\rho(\mathsf{Log}(\bf{K}))=\mathsf{Log}(\bf{P})$. Hence $\mathsf{Log}(\rho(\bf{K}))=\mathsf{Log}(\bf{P})$, so by algebraic completeness, $\rho(\bf{K})=\bf{P}$.\end{proof}

The final map that appears in this context is usually called $\sigma$. Whereas, by the contents of the previous propositions, $\tau$ is the \textit{least} modal companion, $\sigma$ outlines the greatest. Its definition on algebras is thus as follows:
\begin{align*}
    \sigma(\mathbf{K})\coloneqq \mathbb{HSP}\{B(\mathcal{H}) : \mathcal{H}\in \mathbf{K}\}
\end{align*}

Unlike the remaining cases, the fact that $\sigma(\bf{K})$ is dual to $\sigma(\mathsf{Log}(\bf{K}))$ depends on special properties of $\mathsf{S4}$ and $\mathsf{IPC}$. Namely, one needs to make use of ``Blok's Lemma", or an analogous tool, which shows that every algebra $\mathcal{A}\vDash \sigma(L)$ will belong to $\sigma(\mathsf{Var}(L))$. This was established idempendently by Blok \cite{Blok1976VarietiesOI} and Esakia \cite{esakiapaperatconference}, and it relies heavily on the varieties $\sigma(\bf{K})$ being in fact varieties of $\bf{Grz}$-algebras. Nevertheless, the \textit{existence} of greatest modal companions can be established generically, just relying on properties of a sober translation:

\begin{proposition}\label{Maximal Grz-companion}
For any $L\in \Lambda(\mathsf{IPC})$, $\mathsf{Log}(\sigma(\mathsf{Var}(L))$ is the greatest modal companion of $L$.
\end{proposition}
\begin{proof}
First we note that this is a modal companion: assume that $\phi\in L$. Since $\sigma(\mathsf{Var}(L))$ is generated by $B(\mathcal{H})$ for $\mathcal{H}\in \mathsf{Var}(L)$, if $\mathcal{H}$ is such an algebra, then $\mathcal{H}\vDash \phi$, so $B(\mathcal{H})\vDash GMT(\phi)$ (given the unit is an isomorphism, and so $(B(\mathcal{H}))_{\Box}\cong \mathcal{H}$). Hence $GMT(\phi)\in \mathsf{Log}(\sigma(\mathsf{Var}(L)))$. Conversely, if $\phi\notin L$, let $\mathcal{H}\in \mathsf{Var}(L)$ be such that $\mathcal{H}\nvDash \phi$; hence $B(\mathcal{H})\nvDash GMT(\phi)$. But since $B(\mathcal{H})\in \sigma(\mathsf{Var}(L))$ we have that $GMT(\phi)\notin \mathsf{Log}(\sigma(\mathsf{Var}(L)))$.

To see that it is the greatest modal companion, suppose that $M$ is an arbitrary modal companion of $L$. To show that $M\subseteq \mathsf{Log}(\sigma(\mathsf{Var}(L)))$ it suffices to show that $\sigma(\mathsf{Var}(L))\subseteq \mathsf{Var}(M)$. In turn to show this, it suffices to show that $\{B(\mathcal{H}) : \mathcal{H}\in \mathsf{Var}(L)\}\subseteq \mathsf{Var}(M)$. So let $\mathcal{H}\in \mathsf{Var}(L)$ be arbitrary. Since $M$ is a modal companion of $L$, then $\rho(\mathsf{Var}(M))=\mathsf{Var}(L)$. Hence $\mathcal{H}\cong \mathcal{B}_{\Box}$ for some $\mathcal{B}\in \mathsf{Var}(M)$. Moreover, we know that:
\begin{equation*}
    B(\mathcal{B}_{\Box})\preceq \mathcal{B},
\end{equation*}
which holds through the counit map. Hence $B(\mathcal{B}_{\Box})\in \mathsf{Var}(M)$. But then $\sigma(\mathcal{H})\in \mathsf{Var}(M)$, which shows the result.\end{proof}

To conclude this section, we discuss what has been one of the main applications of the maps $\rho$, $\tau$ and $\sigma$ in the literature: \textit{preservation theorems}. These establish that a given axiomatic extension $L\in \Lambda(\mathsf{IPC})$ has a property if and only if one (or all) of its modal companions have it, and similarly, that an axiomatic extension $M\in \mathbf{NExt}(\mathsf{S4})$ has a property if and only if $\rho(M)$ has it. For an extended discussion of such preservation results see \cite{Chagrov1992}. We mention here, without proof, a short selection of such results. We note that that we say that an axiomatic extension $L$ has the \textit{Finite Model Property} if $\mathsf{Var}(L)$ is generated by its finite algebras.

\begin{lemma}
    Let $L\in \Lambda(\mathsf{IPC})$ and $M\in \mathbf{NExt}(\mathsf{S4})$. Then:
    \begin{enumerate}
        \item If $L$ has the FMP, or is decidable; then $\tau(L)$ has the same property.
        \item If $M$ has any of the following properties: tabularity, FMP, decidability, local tabularity, interpolation; then $\rho(M)$ has the same property.
    \end{enumerate}
\end{lemma}

\subsection{Strongly Selective and Sober Translations}\label{Strongly Selective and Sober translations}

The outline of last section should give us a good idea of how to generalise the basic Blok-Esakia theory for many translations. For that purpose, assume throughout this section that $\classX$ and $\classY$ are quasivarieties, where $\classX$ is the algebraic semantics of a logic $\vdash_{\classX}$ (not necessarily the equivalent algebraic semantics), with a set of equations $\mu_{\classX}(x)$ witnessing this fact, and $\classY$ is the equivalent algebraic semantics of $\vdash_{\classY}$, with sets $\mu_{\classY}(x)$ and $\Delta_{\classY}(x,y)$. We assume throughout that $\overline{\zeta}=\langle \zeta,f\rangle:Tm(\fancyL_{\classX})\to Tm(\fancyL_{\classY})$ is a selective translation, and assume that $X$ is a set of variables. We also need a special assumption which is met in all cases we consider. This is that, essentially, the translation commutes with the algebraization \footnote{Note that the translation is not necessarily a substitution, which means that this assumption is not trivial}. More concretely, given any formula $\phi\in \terms(\fancyL_{\classX},X)$, we assume that:
\begin{equation*}
    \zeta^{*}(\mu_{\classX}(\phi))\Dashv \vDash_{\classY} \mu_{\classY}(\zeta_{*}(\phi))
\end{equation*}
Throughout, we use $\theta$ to refer to the right adjoint functor associated to this translation, and $\mathcal{F}$ to denote the corresponding left adjoint functor.

First we will define two maps generalising $\rho$ and $\tau$ from the previous section:

\begin{definition}
Let $\rho$ be the following map: for ${\vdash_{t}} \in \Lambda(\vdash_{\classY})$:
\begin{equation*}
    \rho(\vdash_{t})\coloneqq \{(\Gamma,\phi)\in \mathbb{P}(Tm(\fancyL_{\classX}))\times Tm(\fancyL_{\classX})  : \zeta_{*}[\Gamma]\vdash_{t} \zeta_{*}(\phi)]\}
\end{equation*}

Also, define $\tau$ as the following map, for $\vdash_{s}\in \Lambda(\vdash_{\classX})$:
\begin{equation*}
    \tau(\vdash_{s}) \coloneqq   {\vdash_{\classY}} \oplus \ \{(\zeta_{*}[\Gamma],\zeta_{*}(\phi))\in \mathbb{P}(Tm(\fancyL_{\classY}))\times Tm(\fancyL_{\classY})  : \Gamma \vdash_{s}\phi\}
\end{equation*}
\end{definition}

These maps arise in the context of the following definition:

\begin{definition}
Let ${\vdash_{t}}\in \Lambda(\vdash_{\classY})$ and ${\vdash_{s}}\in \Lambda(\vdash_{\classX})$. We say that $\vdash_{t}$ is a \textit{$\zeta$-companion} of $\vdash_{s}$ if:
\begin{equation*}
    \Gamma\vdash_{s} \phi \iff \zeta_{*}[\Gamma]\vdash_{t} \zeta_{*}(\phi)
\end{equation*}
Given any $\vdash_{s}$, we denote by $\zeta(\vdash_{s})$ the collection of $\zeta$-companions of this logic.\end{definition}

Following \ref{GMT and classic Blok-Esakia}, we also make the following definition on algebras: for $\mathbf{K}\in \Xi(\classX)$:
\begin{align*}
    \tau(\mathbf{K}) \coloneqq \{\mathcal{A} : \theta(\mathcal{A})\in K\}
\end{align*}

The following then corresponds to to Proposition \ref{Key Properties of the tau-map in the S4 case}.

\begin{proposition}\label{Properties of PA-companions under selective translations}
Let ${\vdash_{s}} \in \Lambda(\vdash_{\classX})$ and $\mathbf{K}\in \Xi(\classX)$. Assume that $\zeta$ is a strongly selective translation.
\begin{enumerate}
    \item $\tau(\mathbf{K})$ is a quasivariety.
    \item $\tau(\mathsf{QVar}(\vdash_{s}))=\mathsf{QVar}(\tau(\vdash_{s}))$. Hence $\tau(\mathsf{Log}(\mathbf{K}))=\mathsf{Log}(\tau(\mathbf{K}))$.
    \item $\tau$ is a complete lattice homomorphism.
    \item $\tau(\vdash_{s})$ is the least $\zeta$-companion of $\vdash_{s}$.
\end{enumerate}
\end{proposition}
\begin{proof}
(1) and (3) follow from the same arguments as before, using the fact that the unit is an isomorphism and the fact given by Proposition \ref{Selective translation preserves surjections}, that the right adjoint $\theta$ preserves surjective homomorphisms. For (2) assume that $\mathcal{A}\in \tau(\mathsf{QVar}(\vdash_{s}))$, hence $\theta(\mathcal{A})\in \mathsf{QVar}(\vdash_{s})$. Assume that $\mathcal{A},v\vDash \mu_{\classY}[\zeta_{*}[\Gamma]]$. hence, by assumption on the algebraization, $\mathcal{A},v\vDash \zeta^{*}[\mu_{\classX}[\Gamma]]$. Since the translation is selective, we can define a valuation $w:X\to \theta(\mathcal{A})$ by setting for $x\in X$, $w(x)=v(f(x))$, and then, in an analogous fashion to Proposition \ref{Correctness of Translation in the unit}, obtain that for each $\lambda\in Tm(\fancyL_{\classX},X)$, $w(\lambda)=v(\zeta^{*}(\lambda))$. Thus we obtain that $\theta(\mathcal{A}),w\vDash \mu_{\classX}[\Gamma]$. By assumption then $\theta(\mathcal{A}),w\vDash \mu_{\classX}(\phi)$, and so $\mathcal{A},v\vDash \zeta^{*}(\mu_{\classX}(\phi))$; by assumption then $\mathcal{A},v\vDash \mu_{\classY}[\zeta_{*}(\phi)]$. This shows that $\mathcal{A}\vDash (\zeta_{*}[\Gamma],\zeta_{*}(\phi))$. Similarly, if $\mathcal{A}\in \mathsf{QVar}(\tau(\vdash^{*}))$, then we show that $\theta(\mathcal{A})\in \mathsf{QVar}(\vdash^{*})$ using the converse arguments. By logic-algebra duality we get the result for logics.

For (4), assume that $\Gamma\nvdash_{s}\phi$. Let $\mathcal{A}\in \mathsf{QVar}(\vdash_{s})$ witness this. Since the translation is strongly selective, and hence, the unit is an isomorphism, we know that there exists some algebra $\mathcal{B}$ such that $\theta(\mathcal{B})\cong \mathcal{A}$. Hence, $\mathcal{B}\in \tau(\mathsf{QVar}(\vdash^{*}))$. By the same arguments as above, then $\mathcal{B}\nvDash (\zeta_{*}[\Gamma],\zeta_{*}(\phi))$. But by (2) we have that then $\mathcal{B}\in \mathsf{QVar}(\tau(\vdash^{*}))$, hence, by completeness, $(\zeta_{*}[\Gamma],\zeta_{*}(\phi))\notin \tau(\vdash^{*})$. The fact that it is least is by definition.
\end{proof}

Hence, all strongly selective translations admit a notion of a least $\zeta$-companion. If we strengthen this requirement to sober translations we can look at the following definition for $\bf{K}\in \Xi(\classY)$:

\begin{equation*}
    \rho(\mathbf{K})=\{\theta(\mathcal{A}) : \mathcal{A}\in \mathbf{K}\}.
\end{equation*}
Then we can prove the following:

\begin{proposition}\label{Selective and sober Translations Induce varieties}
Let $\langle \zeta,f\rangle$ be a sober translation, $\bf{K}\in \classY$ and $\vdash_{t}\in \Lambda(\vdash_{\classY})$. Then:
\begin{enumerate}
    \item $\rho(\mathbf{K})$ is a quasivariety.
    \item $\rho(\mathsf{QVar}(\vdash_{t}))=\mathsf{QVar}(\rho(\vdash_{t}))$. Hence $\rho(\mathsf{Log}(\bf{K}))=\mathsf{Log}(\rho(\bf{K}))$.
    \item $\rho:\Lambda(\classX)\to \Lambda(\classY)$ is a surjective complete homomorphism.
\end{enumerate}
\end{proposition}
\begin{proof}
The proof runs exactly the same way as in \ref{Key Properties of the rho-map in S4}, except we use Maltsev's theorem, instead of Tarski's HSP theorem; crucially we use the property that the counit is an isomorphism as in that proof.\end{proof}

We also derive that the definition of $\rho$ on logics, as given above, is well-defined. Importantly, we also obtain a useful criterion for being a $\zeta$-companion, which will be exploited in the next section.

\begin{proposition}\label{Properties of rho map in counits}
For each ${\vdash_{t}}\in \Lambda(\vdash_{\classY})$, $\rho(\vdash_{t})$ is a logic in $\Lambda(\vdash_{\classX})$. Moreover,  ${\vdash_{t}}\in \Lambda(\vdash_{\classY})$ is a $\zeta$-companion of ${\vdash_{s}}\in \Lambda(\vdash_{\classX})$ if and only if $\rho(\vdash_{t})=\vdash_{s}$. Hence, for all $\mathbf{K}\in \Xi(\classY)$ and $\mathbf{P}\in \Xi(\classX)$ we have that if $\mathsf{Log}(\mathbf{K})$ is a modal companion of $\mathsf{Log}(\mathbf{P})$ then $\rho(\mathbf{K})=\mathbf{P}$.
\end{proposition}

Hence, for sober translations we have that all finitary extensions of $\vdash_{\classX}$ have $\zeta$-companions, and the syntactic maps witnessing this transformation have a concrete semantic meaning. For these translations we can also prove the existence of greatest $\zeta$-companions: for $\mathbf{K}\in \Xi(\classX)$, define:
\begin{equation*}
    \sigma(\mathbf{K})\coloneqq \mathbb{ISP}_{R}\{\mathcal{F}(\mathcal{B}) : \mathcal{B}\in \mathbf{K}\}
\end{equation*}

Then we have the following, which has the same proof as in Proposition \ref{Maximal Grz-companion}:

\begin{proposition}
For any logic ${\vdash_{s}} \in \Lambda(\vdash_{\classX})$, $\mathsf{Log}(\sigma(\mathsf{QVar}(\vdash_{s})))$ is the greatest $\zeta$-companion of $\vdash_{s}$.
\end{proposition}

Finally we turn to the preservation results. We focus here only on preservation results along the map $\rho$. These appear to be those which are most amenable to transfer along the categorical lines we have here sketched. Additionally, there is good motivation to have special care about this direction: in general, when translating a system $\vdash_{\classX}$ to a system $\vdash_{\classY}$, the expectation will be that one can reason using well-understood tools of $\vdash_{\classY}$, to better study $\vdash_{\classX}$. Below we recall that a given logic $\vdash_{s}$ is \textit{locally tabular} if $\mathsf{QVar}(\vdash_{s})$ is locally finite: every finitely generated subalgebra of an algebra $\mathcal{B}\in \mathsf{QVar}(\vdash_{s})$ is finite.

\begin{lemma}
    Let ${\vdash_{s}}\in \Lambda(\vdash_{\classX})$ and ${\vdash_{t}}\in \Lambda(\vdash_{\classY})$, and let $\langle \zeta,f\rangle$ be a sober translation. Then if $\vdash_{t}$ has any of the following properties: tabularity, FMP, decidability or local tabularity; then $\rho(\vdash_{t})$ has the same property.
    \end{lemma}
\begin{proof}
    Decidability is trivial. Assume that $\vdash_{t}$ has the FMP; we wish to show that $\rho(\vdash_{s})$ has it as well. Indeed assume that $(\Gamma,\phi)\notin \rho(\vdash_{s})$; hence by definition $(\zeta_{*}[\Gamma],\zeta_{*}(\phi))\notin {\vdash_{s}}$, so there is some finite algebra $\mathcal{B}\nvDash (\zeta_{*}[\Gamma],\zeta_{*}(\phi))$. Then  $\theta(\mathcal{B})$ will be again finite, so we obtain that $\theta(\mathcal{B})\vDash \rho(\vdash_{t})$ and it is a finite countermodel model. For tabularity, assume that $\mathsf{QVar}(\vdash_{s})=\mathbb{ISP}_{R}(\mathcal{A})$; using the fact that $\theta$ commutes with all limits and surjective homomorphisms, then $\rho(\mathsf{QVar}(\vdash_{s}))=\mathbb{ISP}_{R}(\theta(\mathcal{A}))$.
    
    Finally for local tabularity, assume that $\mathcal{C}\preceq \theta(\mathcal{B})$ is a finitely generated subalgebra. Hence we can present $\mathcal{C}\cong \terms(\fancyL,A)/\Phi$ where $A$ is a finite subset. Since $\mathcal{F}$ preserves injective homomorphisms, we have that $\mathcal{F}(\mathcal{C})\preceq \mathcal{F}(\theta(\mathcal{B}))\preceq \mathcal{B}$, and by Proposition \ref{Explicit construction of the left adjoint}, $$\mathcal{F}(\mathcal{C})\cong \terms(\fancyL_{\classY},A)/\mathsf{Cg}_{\classY}(\zeta^{*}(\Phi)\cup \{f(x)\approx x: x\in A\})$$. Hence $\mathcal{F}(\mathcal{C})$ is also finitely generated, and is a subalgebra of $\mathcal{B}$; so by local finiteness, $\mathcal{F}(\mathcal{C})$ is finite, hence $\theta(\mathcal{F}(\mathcal{C}))\cong \mathcal{C}$ is finite as well.\end{proof}

\section{Connecting Blok-Esakia and Polyatomic Logics}\label{Connecting Blok-Esakia and Polyatomic Logics}

In the previous sections we have introduced Polyatomic logics and Blok-Esakia theory as tools to study translations. In both cases the use of a selector term turned out to be important to derive some basic results. In this section we investigate a connection between these two concepts. In particular we will show that whilst an isomorphism between lattices of varieties might not hold in general -- for many classes of translations, such as sober translations, an isomorphism exists between the lattice of extensions of one system and the lattice of $\mathsf{PAt}$-$f$-logics of the other, where $f$ is the same selector of the translation. Additionally, we show that the schematic fragment -- used in inquisitive logic \cite{Ciardelli2018-ql} and adapted to DNA-logics in \cite{bezhanishvili_grilletti_quadrellaro_2021} and weak logics in \cite{quadrellaronakov} -- appears in this setting as the greatest $\zeta$-companion.

To highlight these connections, as in the previous section, we begin by discussing these results in the setting of the double negation and GMT translations, and then proceed to generalise to broader classes.

\subsection{Polyatomic Logics as Generalised Companions}\label{Polyatomic Logics as Generalised Companions}

First of all we recall from \cite[Definition 4.11]{bezhanishvili_grilletti_quadrellaro_2021} two crucial notions in the analysis of the lattice of DNA-logics: that of the \textit{least} and \textit{greatest} variant. More concretely, an analysis was given of the following concepts:

\begin{definition}
Let $L\in \Lambda(\mathsf{IPC})$. We say that $L$ is:
\begin{itemize}
    \item \textit{DNA-minimal} if whenever $S$ is an intermediate logic and $L^{\neg\neg}=S^{\neg\neg}$ then $L\subseteq S$.
    \item \textit{DNA-maximal} if whenever $S$ is an intermediate logic and $L^{\neg\neg}=S^{\neg\neg}$ then $S\subseteq L$.
\end{itemize}
\end{definition}

\begin{definition}\label{Definition of schematic fragment}
Let $L\in \Lambda^{\neg\neg}(\mathsf{IPC})$. We define the \textit{schematic fragment} of $L$, denoted $Schem(L)$ as follows:
\begin{equation*}
    Schem(L)=\{ \phi : \forall \overline{\psi}, \  \phi[\overline{\psi}/\overline{p}]\in L\}
\end{equation*}
\end{definition}

The following facts were shown in \cite[Proposition 4.5, Proposition 4.16]{bezhanishvili_grilletti_quadrellaro_2021}

\begin{proposition}
    For each $L\in \Lambda^{\neg\neg}(\mathsf{IPC})$, $Schem(L)\in \Lambda(\mathsf{IPC})$ is the greatest logic, in the sense of axiomatic extension, such that $Schem(L)^{\neg\neg}=L$. Moreover:
    \begin{equation*}
    \mathsf{Var}(Schem(L))=\mathsf{Var}(\{\langle H_{\neg} \rangle : \mathcal{H}\in \mathsf{Var}^{\neg\neg}(L)\})
\end{equation*}
\end{proposition}

In other words, the variety generated by the schematic fragment is precisely the variety generated by the regularly generated subalgebras of those $\mathcal{H}$ which belong to the DNA-variety of $L'$. Additionally, we have the following, shown in Proposition \cite[Proposition 4.12]{bezhanishvili_grilletti_quadrellaro_2021}:

\begin{proposition}
For each $L\in \Lambda(\mathsf{IPC})$, we have that:
\begin{enumerate}
    \item $L$ is DNA-minimal if and only if $\mathsf{Var}(L)=\mathsf{Var}^{\neg\neg}(L^{\neg\neg})$;
    \item $L$ is DNA-maximal if and only if $L=Schem(L^{\neg\neg})$.
\end{enumerate}
\end{proposition}

Hence, schematic fragments provide a concrete syntactic description of the greatest DNA-variant. This is not of course a very concrete decription -- as mentioned in \cite{bezhanishvili_grilletti_quadrellaro_2021} the schematic fragment of $\mathsf{Inq}$, inquisitive propositional logic, the $\neg\neg$-variant of $\mathsf{IPC}$, is the well-known \textit{Medvedev Logic}, which has no known recursive axiomatisation. Nevertheless, knowing the properties associated to it can provide us with insight on the nature of the translation and the logics at play, as exploited for instance in \cite{quadrellaronakov}.

The reader might have noticed a symmetry between the notions here introduced here of minimality and maximality, and the notions of the least modal companion and the greatest modal companion discussed in \ref{GMT and classic Blok-Esakia}. Having the two sets of concepts it might be natural to ask if there is any relationship between them. We will show that this is the case, and first illustrate this by the example of the GMT translation.

Recall from Example \ref{Box logics} that we discussed $\Box$-variants of $\mathsf{S4}$-systems, and hence we can easily generalise the definitions given above (formal and general definitions will be given in the next section).

\begin{proposition}\label{Box minimal logics if and only if least modal companion}
For each variety $\mathbf{K}\in \Xi(\mathbf{S4})$:
\begin{equation*}
    \tau(\rho(\mathbf{K}))=\mathbf{K}^{\uparrow}
\end{equation*}
Consequently, a logic in the sense of axiomatic extension $L\in \mathbf{NExt}(\mathsf{S4})$ is $\Box$-minimal if and only if it is the least modal companion of $\rho(L)$. 
\end{proposition}
\begin{proof}
First assume that $\mathcal{B}\in \tau(\rho(\mathbf{K}))$; then $\mathcal{B}_{\Box}=\mathcal{C}_{\Box}$ where $\mathcal{C}\in \mathbf{K}$; hence, up to isomorphism, $B(\mathcal{C}_{\Box})$ is a subalgebra of $\mathcal{B}$, and also a subalgebra of $\mathcal{C}$, since the counit is injective. Thus $B(\mathcal{C}_{\Box})\in \mathbf{K}$, and  $\mathcal{B}$ is a core superalgebra of $B(\mathcal{C}_{\Box})$; so $\mathcal{B}\in \mathbf{K}^{\uparrow}$. Conversely, if $\mathcal{B}\in \mathbf{K}^{\uparrow}$, then $\mathcal{C}\preceq \mathcal{B}$ where $\mathcal{C}\in \mathbf{K}$, and $\mathcal{B}_{\Box}\cong \mathcal{C}_{\Box}$; hence, $\mathcal{B}_{\Box}\in \rho(\mathbf{K})$ by closure under isomorphisms, hence, $\mathcal{B}\in \tau(\rho(\mathbf{K}))$.

To see the second statement, first note that:
\begin{equation*}
    (\mathsf{Var}(L))^{\uparrow}=\mathsf{Var}^{\Box}(L^{\Box})
\end{equation*}
which follows essentially by Lemma \ref{Basic Commutativity Result}. Hence using the first part of this proposition:
\begin{equation*}
    \mathsf{Var}(L)=\mathsf{Var}^{\Box}(L^{\Box}) \iff \mathsf{Var}(L)=\tau(\rho(\mathsf{Var}(L))
\end{equation*}
Which using Proposition \ref{Key Properties of the tau-map in the S4 case} yields the result.\end{proof}

Just like least modal companions are $\Box$-minimal, we can see that greatest modal companions are $\Box$-maximal. As in Definition \ref{Definition of schematic fragment} we can define the schematic fragment of a $\Box$-logic $L$ and then the following is not hard to see:

\begin{proposition}\label{Box logics are maximal}
    For each $L\in \Lambda^{\Box}(\mathsf{S4})$, we have that $Schem(L)\in \Lambda(\mathsf{S4})$ is the greatest logic such that $Schem(L)^{\Box}=L$.
\end{proposition}

    Additionally, the following can be shown by a proof paralleling \cite[Theorem 4.16]{bezhanishvili_grilletti_quadrellaro_2021}; we will show it in full generality in the next section:

    \begin{lemma}\label{schematic Box logics are generated by maximal}
        For each $L\in \mathsf{S4}$, we have that:
        \begin{equation*}
            \mathsf{Var}(Schem(L^{\Box}))=\mathbb{HSP}(\{\langle \mathcal{B}_{\Box}\rangle : \mathcal{B}\in \mathsf{Var}(L)\})
        \end{equation*}
\end{lemma}

\begin{corollary}\label{Greatest modal companion in the case of GMT}
For each variety $\bf{K}\in \Xi(\mathbf{S4})$ we have that:
\begin{equation*}
    \mathsf{Log}(\sigma(\rho(\bf{K})))=Schem(\mathsf{Log}(\bf{K})^{\Box})
\end{equation*}
Consequently, a logic in the sense of axiomatic extension $L\in \mathbf{NExt}(\mathsf{S4})$ is $\Box$-maximal if and only if it is greatest modal companion of $\rho(L)$.
\end{corollary}
\begin{proof}
First note that given $\mathcal{B}$ an $\mathbf{S4}$-algebra, by the description of the left adjoint we have that:
\begin{equation*}
    B(\mathcal{B}_{\Box})\cong \langle \mathcal{B}_{\Box} \rangle
\end{equation*}
Hence by Lemma \ref{schematic Box logics are generated by maximal}, and the definition of $\sigma$, we have that:
\begin{equation*}
    \sigma(\rho(\bf{K}))=\mathbb{HSP}(\{B(\mathcal{B}_{\Box}) : \mathcal{B}\in \bf{K}\})=\mathbb{HSP}(\{\langle \mathcal{B}_{\Box}\rangle : \mathcal{B}\in \bf{K}\})=\mathsf{Var}(Schem(\mathsf{Log}(\bf{K})^{\Box}))
\end{equation*}
which implies the first statement. To see the second one, note that given $L\in \mathbf{NExt}(\mathsf{S4})$, then $L=Schem(L^{\Box})$ if and only if $\mathsf{Var}(L)=\mathsf{Var}(Schem(L^{\Box}))$, and in light of the above, setting $\bf{K}=\mathsf{Var}(L)$, and using Lemma \ref{Basic Commutativity Result}, the latter holds if and only if $\mathsf{Var}(L)=\sigma(\rho(\mathsf{Var}(L)))$. By algebraic completeness and Proposition \ref{Key Properties of the rho-map in S4} this holds if and only if $L=\mathsf{Log}(\sigma(\mathsf{Var}(\rho(L))))$, why by Proposition \ref{Maximal Grz-companion}, holds precisely if $L$ is the greatest modal companion of $\rho(L)$.\end{proof}

The above corollary provides a rather surprising bridge between, on one hand, schematic fragments, typically studied in the context of weak logics, and modal companions on the other hand. In a sense, this can be seen as explaining the key difficulties in establishing the last part of the Blok-Esakia theorem, and the role of the $\mathsf{Grz}$ axiom in establishing both the Blok-Esakia theorem and many of its variants, such as for bi-intuitionistic logic or modal intuitionistic logic \cite{Wolter2014}, since as noted above, the schematic fragment may happen to be a logic as pathological as the Medvedev logic. However, if we now consider a world where we did not know the axiom $\mathsf{Grz}$, we can imagine the possibilities of still studying the relationship between $\mathsf{S4}$ and $\mathsf{IPC}$ through the lens of variants. Namely, we could -- as we will do in the next section -- derive the following result:

\begin{proposition}
The lattice $\Lambda(\mathsf{IPC})$ is isomorphic to $\Lambda^{f}(\mathsf{S4})$.
\end{proposition}

Let us exemplify this briefly: consider for instance the logic $\mathsf{LC}\in \Lambda(\mathsf{IPC})$, which is axiomatised by the axiom $p\rightarrow q\vee q\rightarrow p$. It can be shown by semantic methods that $\tau(\mathsf{LC})=\mathsf{S4}.3$, the system $\mathsf{S4}$ together with the axiom $\Box(\Box p\rightarrow q)\vee \Box(\Box q\rightarrow p)$. The lattice of extensions of this logic has a countable, though somewhat complicated structure; by contrast, in light of the previous result, one has that $\Lambda(\mathsf{LC})\cong \Lambda^{\Box}(\mathsf{S4}.3)\cong \Lambda(\mathsf{Grz.3})$, which is known to be isomorphic to an infinite descending chain (see \cite[pp.427]{Chagrov1997-cr}). Hence, the study of $\Box$-variant extensions could presumably, in a setting where the $\mathsf{Grz}$ axiom was not known, be carried out in a more straightforward fashion. A similar but more striking example occurs when we look at $\mathsf{CPC}$: its least modal companion is the system $\mathsf{S5}$, which has infinitely many extensions. By contrast, $\mathsf{S5}^{\Box}$ has no proper extensions.

\subsection{Connecting Companions and Variants}

In this section we conclude our discussion by establishing the connection between $\zeta$-companions and $\mathsf{PAt}$-logics. Assume throughout a selective translation $\langle \zeta,f\rangle$, and quasivarieties $\classX$ and $\classY$ in the same conditions as in Section \ref{Section: General Blok-Esakia Theory}.

\begin{definition}
Let ${\vdash_{t}}\in \Lambda(\vdash_{\classY})$. We say that ${\vdash_{t}}$ is:
\begin{itemize}
    \item \textit{$f$-minimal} if whenever ${\vdash_{m}} \in \Lambda(\vdash_{\classY})$ and  ${\vdash_{m}^{f}}={\vdash_{t}^{f}}$ then ${\vdash_{t}}\subseteq {\vdash_{m}}$.
    \item \textit{$f$-maximal} if whenever ${\vdash_{m}} \in \Lambda(\vdash_{\classY})$ and  ${\vdash_{m}^{f}}={\vdash_{t}^{f}}$ then ${\vdash_{m}}\subseteq {\vdash_{t}}$.
\end{itemize}
\end{definition}

\begin{lemma}\label{Semantic meaning of being f-minimal}
Let $\vdash_{t}\in \Lambda(\vdash_{\classY})$. Then $\vdash_{t}$ is $f$-minimal if and only if $\mathsf{QVar}(\vdash_{t})=\mathsf{QVar}^{f}(\vdash_{t}^{f})$.    
\end{lemma}
\begin{proof}
Note that in light of Lemma \ref{Invariance of truth for f(p) valuation}, if $\mathcal{A}\vDash {\vdash_{t}}$ then $\mathcal{A}\vDash_{f} {\vdash_{t}^{f}}$. So first assume that $\vdash_{t}$ is $f$-minimal. Consider $\mathsf{Log}(\mathsf{QVar}^{f}(\vdash_{t}^{f}))$. Then note that the $f$-variant of this logic will simply be $\vdash_{t}^{f}$, which by minimality implies that:
\begin{equation*}
    {\vdash_{t}} \subseteq  \mathsf{Log}(\mathsf{QVar}^{f}(\vdash_{t}^{f}))
\end{equation*}
Which implies equality.

Conversely, assume that $\mathsf{QVar}(\vdash_{t})=\mathsf{QVar}^{f}(\vdash_{t}^{f})$. Suppose that $\vdash_{m}^{f}=\vdash_{t}^{f}$. Note that we have that $\mathsf{QVar}(\vdash_{m})\subseteq (\mathsf{QVar}(\vdash_{m}))^{\uparrow}$; but by Lemma \ref{Basic Commutativity Result}, this means that $\mathsf{QVar}(\vdash_{m})\subseteq \mathsf{QVar}^{f}(\vdash_{m}^{f})=\mathsf{QVar}(\vdash_{t})$.\end{proof}

\begin{definition}
Let ${\vdash_{*}}\in \Lambda^{f}(\vdash_{\classY})$. We define the \textit{schematic fragment} of ${\vdash_{*}}$, denoted $Schem(\vdash_{*})$ as follows:
\begin{equation*}
    Schem(\vdash_{*})=\{ (\Gamma,\phi) : \forall \overline{\psi}, \  \Gamma[\overline{\psi}/\overline{p}]\vdash_{*}\phi[\overline{\psi}/\overline{p}])\}
\end{equation*}
\end{definition}

We begin by proving that when we are working with sufficiently nice translations -- such sober translations -- we have something analogous to the Blok-Esakia isomorphism, relating logics of one system to polyatomic logics of the other.

\begin{proposition}\label{Map tau dualises semantically}
Assume that $\zeta$ is a strongly selective translation, and $\mathbf{K}\in \Xi(\classX)$. Then $\tau(\mathbf{K})$ is a $\mathsf{PAt}$-quasivariety. Moreover, the assignment:
\begin{equation*}
    \tau: \Xi(\classX) \to \Xi^{f}(\classY)
\end{equation*}
is injective.
\end{proposition}
\begin{proof}
By Proposition \ref{Properties of PA-companions under selective translations}, we know that $\tau(\mathbf{K})$ is a quasivariety. Moreover, it is easy to see that it is closed under core superalgebras: if $\mathcal{A}\in \tau(\mathbf{K})$, by definition $\theta(\mathcal{A})\in \bf{K}$; hence if  $\mathcal{A}\preceq_{f} \mathcal{B}$, then $\theta(\mathcal{A})=\theta(\mathcal{B})$ by Lemma \ref{Identity of the notions of regular elements}; so $\theta(\mathcal{B})\in \mathbf{K}$, and thus $\mathcal{B}\in \tau(\mathbf{K})$. To see that $\tau$ is injective, notice that if $\mathbf{K}\neq \mathbf{K}'$, and $\mathcal{A}\in \mathbf{K}$ and $\mathcal{A}\notin \mathbf{K}'$, then $\mathcal{F}(\mathcal{A})\in \tau(\mathbf{K})$ by definition of $\tau$, and the fact that $\theta(\mathcal{F}(\mathcal{A}))\cong \mathcal{A}$. On the other hand, if $\mathcal{F}(\mathcal{A})\in \tau(\mathbf{K}')$, then $\theta(\mathcal{F}(\mathcal{A}))\in \mathbf{K}'$. So $\tau(\mathbf{K})\neq \tau(\mathbf{K}')$.\end{proof}

If moreover we assume that the translation is sober, we have the following result:

\begin{proposition}($\mathsf{PAt}$-Blok Esakia Isomorphism) \label{Weak Blok-Esakia}
Assume that $\zeta$ is a sober translation. Then the assignment $\tau:\Xi(\classX)\to \Xi^{f}(\classY)$ is a lattice isomorphism.
\end{proposition}
\begin{proof}
Let $\mathbf{K}$ be an arbitrary $\mathsf{PAt}$-quasivariety in $\Lambda^{f}(\classY)$. We will show that $\tau\rho(\mathbf{K})=\mathbf{K}$. Indeed, if $\mathcal{A}\in \mathbf{K}$, then $\theta(\mathcal{A})\in \rho(\mathbf{K})$, so by definition $\mathcal{A}\in \tau\rho(\mathbf{K})$. Conversely, if $\mathcal{A}\in \tau\rho(\mathbf{K})$, then $\theta(\mathcal{A})\in \rho(\mathbf{K})$. By assumption, then $\theta(\mathcal{A})=\theta(\mathcal{B})$ for some $\mathcal{B}\in \bf{K}$. Hence note that $\mathcal{F}(\theta(\mathcal{A}))=\mathcal{F}(\theta(\mathcal{B}))$. Since the translation is sober, by sobriety, $\mathcal{F}(\theta(\mathcal{B}))\preceq \mathcal{B}$, so $\mathcal{F}(\theta(\mathcal{B}))\in \bf{K}$. But by sobriety again, $\mathcal{F}(\theta(\mathcal{A}))\preceq \mathcal{A}$, and so we get that $\mathcal{A}$ is a core superalgebra of $\mathcal{F}(\theta(\mathcal{B}))$. Since $\mathbf{K}$ is a $\mathsf{PAt}$-quasivariety, $\mathcal{A}\in \bf{K}$. This shows that $\tau$ is surjective. The fact that $\tau$ is an injective homomorphism follows from Proposition \ref{Map tau dualises semantically}.\end{proof}

Using the algebraic completeness proved in Section \ref{Section: Translations and Polyatomic Logics} we thus obtain:

\begin{corollary}
If $\langle \zeta,f\rangle$ is a sober translation, then $\Lambda(\vdash_{\classX})\cong \Lambda^{f}(\vdash_{\classY})$.
\end{corollary}

The following also follows like Proposition \ref{Box minimal logics if and only if least modal companion}, using Lemma \ref{Semantic meaning of being f-minimal} and the $\mathsf{PAt}$-Blok Esakia isomorphism:

\begin{corollary}\label{f-minimal if and only if least modal companion}
    If $\langle \zeta,f\rangle$ is a sober translation, then a logic ${\vdash_{t}}\in \Lambda(\vdash_{\classY})$ is $f$-minimal if and only if it is the least modal companion of $\rho(\vdash_{t})$.
\end{corollary}

This isomorphism can thus serve as a natural correspondence for the study of the relationship between two systems. As an example application, we get another proof that the Double Negation Translation is not sober, and one which makes the fact stand out sharply: if so, then by the above result we would have that $\Lambda(\vdash_{\mathsf{CPC}})\cong \Lambda^{\neg\neg}(\vdash_{\mathsf{IPC}})$, which is absurd since there are infinitely many DNA-logics \cite[Theorem 5.11]{bezhanishvili_grilletti_quadrellaro_2021}.

The rest of this section will be dedicated to to obtaining a corresponding result to Corollary \ref{Box minimal logics if and only if least modal companion} for $f$-maximal logics. First, given $\bf{K}\in \Xi(\classX)$, we define

    \begin{equation*}
        \sigma(K)=\mathsf{QVar}(\{ \mathcal{F}(\mathcal{A}) : \mathcal{A}\in \mathbf{K}\})
\end{equation*}

To proceed, we will need to collect some properties about the schematic fragment which are easy to see: 

\begin{lemma}
For each $\mathsf{PAt}$-logic ${\vdash_{*}}\in \Lambda^{f}(\vdash_{\classY})$, $Schem(\vdash_{*})$ is a logic in $\Lambda(\vdash_{\classY})$. Moreover, it is the greatest logic which has $\vdash_{*}$ as its $\mathsf{PAt}$-variant.
\end{lemma}
\begin{proof}
The verification that this is a logic is straightforward. It is also clear that the schematic fragment will have $\vdash_{*}$ as its $\mathsf{PAt}$-variant. Now suppose that $\vdash_{t}$ has $\vdash_{*}$ as a $\mathsf{PAt}$-variant. Assume that $\Gamma\vdash_{t}\phi$ and let $\overline{\psi}\in Tm(\fancyL_{\classY})$. Then we have $\Gamma[\overline{\psi}/\overline{p}]\vdash_{t}\phi[\overline{\psi}/\overline{p}]$ by uniform substitution, hence, $\Gamma[\overline{\psi}/\overline{p}][f(\overline{q})/\overline{q}]\vdash_{t}\phi[\overline{\psi}/\overline{q}][f(\overline{q})/\overline{q}]$. This means that $\Gamma[\overline{\psi}/\overline{p}]\vdash_{*}\phi[\overline{\psi}/\overline{p}]$. Hence $(\Gamma,\phi)\in Schem(\vdash_{*})$.\end{proof}

Moreover, we can show that schematic fragments are correspond exactly to the regularly generated quasivarieties. 

\begin{proposition}\label{Semantics of schematic fragment}
Let ${\vdash_{t}}\in \Lambda(\vdash_{\classY})$. Then ${\vdash_{t}^{f}}=Schem(\vdash_{t}^{f})$ if and only if $\mathsf{QVar}(\vdash_{t})=\mathsf{QVar}(\{\langle \mathcal{B}^{f}\rangle :\mathcal{B}\in \mathsf{QVar}(\vdash_{t})\})$.
\end{proposition}
\begin{proof}
First we prove right to left. Note that by maximality of the schematic fragment amongst logics with the same variant, ${\vdash_{t}}\subseteq Schem(\vdash_{t}^{f})$.So suppose that $\Gamma\nvdash_{t}\phi$. Hence by assumption, we can find $\mathcal{A}\in \mathsf{QVar}(\vdash_{t})$, a regularly generated algebra, such that $\mathcal{A}\nvDash (\Gamma,\phi)$; hence by Lemma \ref{Invariance of truth for f(p) valuation}, $\langle A^{f}\rangle \nvDash (\Gamma,\phi)$; so by Corollary \ref{Failure in regularly generated yields substitution instance}, we obtain a substitution $\nu$, such that $\mathcal{A}\nvDash_{f} (\nu[\Gamma],\nu(\phi))$, and so again by Lemma \ref{Invariance of truth for f(p) valuation}, $\mathcal{A}\nvDash (\nu[\Gamma][f(\overline{q})/\overline{q}],\nu(\phi)[f(\overline{q})/\overline{q}])$. Hence, since $\mathcal{A}\in \mathsf{QVar}(\vdash_{t})$,  we have that $\Gamma[\sigma(\overline{p})/\overline{p}][f(\overline{q})/\overline{q}]\nvdash_{t}\phi[\sigma(\overline{p})/\overline{p}][f(\overline{q})/\overline{q}]$; so $\Gamma[\sigma(\overline{p})/\overline{p}]\nvdash_{t}^{f}\phi[\sigma(\overline{p})/\overline{p}]$ and by definition
\begin{equation*}
    (\Gamma,\phi)\notin Schem(\vdash_{t}^{f})
\end{equation*}
as desired.

Now assume that ${\vdash_{t}}=Schem(\vdash_{t}^{f})$. First define:
\begin{equation*}
    \mathsf{QVar}_{R}(\vdash_{t})=\mathsf{QVar}(\{\langle B^{f}\rangle : \mathcal{B}\in \mathsf{QVar}(\vdash_{t})\}).
\end{equation*}
Then it is straightforward to see that
\begin{equation*}
    \mathsf{QVar}_{R}(\vdash_{t})\subseteq \mathsf{QVar}(\vdash_{t}).
\end{equation*}

For the other inclusion, we show that $\mathsf{Log}(\mathsf{QVar}_{R})$ has $\vdash_{t}^{f}$ as its $\mathsf{PAt}$-variant. Indeed we have:
\begin{align*}
    (\mathsf{Log}(\mathsf{QVar}_{R}(\vdash_{t})))^{f}&=\mathsf{Log}^{f}(\mathsf{QVar}(\{\langle B^{f}\rangle : B\in \mathsf{QVar}(\vdash_{t})\})^{\uparrow})\\
    &= \mathsf{Log}^{f}(\mathsf{QVar}^{f}(\vdash_{t}))\\
    &= \vdash_{t}^{f}
\end{align*}
Where the first inclusion follows by the commutativity of the operators from Lemma \ref{Basic Commutativity Result}, the second follows by Corollary \ref{Regular generation of PA-Quasivarieties} from the fact that every $\mathsf{PAt}$-quasivariety is generated as a $\mathsf{PAt}$-quasivariety by its regular elements and the final equality follows by the algebraic completeness result. Hence we conclude that $\mathsf{Log}(\mathsf{QVar}_{R}(\vdash_{t}))$ has $\vdash_{t}^{f}$ as its $\mathsf{PAt}$-variant, and thus we know that $\mathsf{Log}(\mathsf{QVar}_{R}(\vdash_{t}))\subseteq Schem(\vdash_{t}^{f})$. Hence $\mathsf{QVar}(Schem((\vdash_{*})^{f}))\subseteq \mathsf{QVar}_{R}(\vdash_{*})$, which shows the result.\end{proof}

Moreover, schematic fragments are also companions. To simplify notation, given $\vdash_{s}\in \Lambda(\vdash_{\classX})$, we denote by:
\begin{equation*}
    \mathbf{Sch}(\vdash_{s})\coloneqq Schem(\mathsf{Log}^{f}(\mathsf{QVar}(\tau(\vdash_{s}))))
\end{equation*}

\begin{lemma}\label{Schematic fragment is a PA companion}
For each ${\vdash_{s}}\in \Lambda(\vdash_{\classX})$ we have that $\mathbf{Sch}(\vdash_{s})$ is a $\zeta$-companion of $\vdash_{s}$.
\end{lemma}
\begin{proof}
Suppose that $\Gamma\vdash_{s}\phi$. Then by definition, $(\zeta_{*}[\Gamma],\zeta_{*}[\phi])\in \tau(\vdash_{s})$. Now suppose that $\overline{\psi}\in Tm(\fancyL_{\classY})$ is any finite collection of formulas. We want to show that if $\mathcal{A}\in \mathsf{QVar}(\tau(\vdash_{s}))$, then $\mathcal{A}\vDash_{f} (\zeta_{*}[\Gamma][\overline{\psi}/\overline{p}],\zeta_{*}[\phi][\overline{\psi}/\overline{p}])$ . Hence by Lemma \ref{Invariance of truth for f(p) valuation}, we want to check that $\mathcal{A}\vDash (\zeta_{*}[\Gamma][\overline{\psi}/\overline{p}][f(\overline{q})/\overline{q}],\zeta_{*}[\overline{\psi}/\overline{p}][f(\overline{q})/\overline{q}])$. But we know that $\tau(\vdash_{s})$ is a logic, and hence is closed under uniform substitution; so since $\mathcal{A}\vDash (\zeta_{*}[\Gamma],\zeta_{*}[\phi])$, the result follows.

Conversely, suppose that $(\Gamma,\phi)\notin \ \vdash_{s}$. Hence for some $\mathcal{B}\in \mathsf{QVar}(\vdash_{s})$, $\mathcal{B}\nvDash (\Gamma,\phi)$. Note that since the unit is an isomorphism, $\mathcal{F}(\mathcal{B})\vDash \tau(\vdash_{s})$, and since $\mathcal{B}\nvDash (\Gamma,\phi)$, say through a valuation $v$, we can define a valuation $w$ on $\mathcal{F}(\mathcal{B})$ by setting $w(p)=v(p)$. Set $\mathcal{C}=\mathcal{F}(\mathcal{B})$, and note that by the same arguments as used in Proposition \ref{Correctness of Translation in the unit}, we can then show that $\mathcal{C},w\nvDash (\zeta^{*}[\Gamma],\zeta^{*}[\phi])$, and the valuation takes values in $f$-regular elements; hence $\mathcal{C}\nvDash_{f} (\zeta^{*}[\Gamma],\zeta^{*}[\phi])$, which by Lemma \ref{Invariance of truth for f(p) valuation} means that $\langle \mathcal{C}^{f}\rangle \nvDash_{f} (\zeta_{*}[\Gamma],\zeta_{*}[\phi])$. But now by Corollary \ref{Failure in regularly generated yields substitution instance}, we have that there is a substitution $\nu$, such that $\mathcal{C}\nvDash_{f} (\nu[\zeta_{*}[\Gamma]],\nu[\zeta_{*}[\phi]]$. This means that $(\zeta[\Gamma],\zeta[\phi])\notin \mathbf{Sch}(\vdash_{*})$.\end{proof}

The next lemma shows the basic facts necessary to establish the fact that schematic fragments correspond to greatest modal companions:

\begin{lemma}\label{Minimality of the sigma map}
Let $\mathbf{K}\in \Lambda(\classY)$. Let $\mathcal{A}\in \classY$ and $\mathcal{B}\in \classX$. Then:
\begin{enumerate}
    \item If $\mathsf{Log}(\mathbf{K})$ is a $\zeta$-companion of ${\vdash_{s}} \in \Lambda(\classX)$, then if $\bf{S}=\mathsf{QVar}(\vdash_{s})$ then $\sigma(\bf{S})\subseteq \bf{K}$.
    \item If $\mathcal{A}=\mathcal{F}(\mathcal{B})$ then $\mathcal{A}$ is regularly generated.
    \item If $\mathcal{A}$ is regularly generated then $\mathcal{A}\cong \mathcal{F}(\theta(\mathcal{A}))$.
\end{enumerate}
\end{lemma}
\begin{proof}
For (1) it suffices to show that $\{\mathcal{F}(\mathcal{A}) : \mathcal{A}\in \mathbf{S}\}\subseteq \mathbf{K}$. By Proposition \ref{Weak Blok-Esakia}, we have that $\rho(\mathbf{K})=\mathbf{S}$. Hence if $\mathcal{A}\in \mathbf{S}$, then $\mathcal{A}=\theta(\mathcal{B})$ for some $\mathcal{B}\in \mathbf{K}$. But then $\mathcal{F}(\mathcal{A})=\mathcal{F}(\theta(\mathcal{B}))\preceq \mathcal{B}$, given the translation is sober. So indeed, $\mathcal{F}(\mathcal{A})\in \mathbf{K}$, as intended.

For (2) first note that $\theta(\mathcal{A})\cong \mathcal{B}$, since the unit is an isomorphism. Hence, consider $\langle \mathcal{B}\rangle$, the subalgebra generated in $\mathcal{A}$ by $\mathcal{B}$. Clearly we have that $\theta(\langle\mathcal{B}\rangle)\cong \mathcal{B}$. But then we have that $\mathcal{A}\cong \mathcal{F}(\mathcal{B})$ is a subalgebra of $\langle\mathcal{B}\rangle$, since the counit is injective. Hence $\langle \mathcal{B}\rangle$ and $\mathcal{A}$ are subalgebras of each other, hence, the same algebra.

For (3), suppose that $\mathcal{A}$ is regularly generated. Then first note that $\theta(\mathcal{A})\subseteq \mathcal{F}(\theta(\mathcal{A})$, since $\theta(\mathcal{A})\cong \theta(\mathcal{F}(\theta(\mathcal{A}))\subseteq \mathcal{F}(\theta(\mathcal{A}))$. Hence, $\mathcal{F}(\theta(\mathcal{A}))$ is a subalgebra of $\mathcal{A}$ which contains the regular elements; since $\mathcal{A}$ is regularly generated, $\mathcal{A}\cong \mathcal{F}(\theta(\mathcal{A}))$.
\end{proof}

Using all of these lemmas we conclude the following:

\begin{corollary}\label{Specific structure of the greatest companion}
If $\langle\zeta,f\rangle$ be a sober translation, then for each $\vdash_{s}\in \Lambda(\vdash_{\classX})$, $\mathbf{Sch}(\vdash_{s})$ is the greatest $\zeta$-companion of $\vdash_{s}$.
\end{corollary}
\begin{proof}
By Lemma \ref{Schematic fragment is a PA companion}, we have that $\mathbf{Sch}(\vdash_{s})$ is a $\zeta$-companion. Moreover, by lemma \ref{Semantics of schematic fragment}, we have that $\mathsf{QVar}(\mathbf{Sch}(\vdash_{s}))$ is generated by its regularly generated algebras. Then in light of Lemma \ref{Minimality of the sigma map}, we note that
\begin{equation*}
    \{\langle \mathcal{B}^{f}\rangle : \mathcal{B}\in \mathsf{QVar}(\mathbf{Sch}(\vdash_{s}))\} = \{\mathcal{F}(\mathcal{A}) : \mathcal{A} \in \mathsf{QVar}(\vdash_{s})\}.
\end{equation*}
To see this, note that by Lemma \ref{Minimality of the sigma map}.1, $\{\mathcal{F}(\mathcal{A}) :\mathcal{A}\in \mathsf{QVar}(\vdash_{*})\}\subseteq \mathsf{QVar}(\mathbf{Sch}(\vdash_{s}))$, and by part 2 of that lemma, all of the former are regularly generated algebras. On the other hand, if $\mathcal{B}\in \mathsf{QVar}(\mathbf{Sch}(\vdash_{s}))$ is a regularly generated algebra, then first note that $\theta(\mathcal{B})\in \mathsf{QVar}(\vdash_{s})$, by Proposition \ref{Properties of rho map in counits}, and the fact that $\mathbf{Sch}(\vdash_{s})$ is a $\zeta$-companion. Hence $\mathcal{B}\cong \mathcal{F}(\theta(\mathcal{B}))\in \{\mathcal{F}(\mathcal{A}) : \mathcal{A}\in \mathsf{QVar}(\vdash_{s})\}$.

Hence we obtain.
\begin{equation*}
    \mathsf{QVar}(\mathbf{Sch}(\vdash_{s})) = \mathsf{QVar}(\{\mathcal{F}(\mathcal{A}) : \mathcal{A}\in \mathsf{QVar}(\vdash_{s})\}),
\end{equation*}
Now let $\bf{K}$ be any quasivariety such that $\mathsf{Log}(\bf{K})$ is a $\zeta$-companion of $\vdash_{s}$. By Lemma \ref{Minimality of the sigma map}.1, we have that  \begin{equation*}
    \{\mathcal{F}(\mathcal{A}) : \mathcal{A}\in \mathsf{QVar}(\vdash_{*})\}\subseteq \bf{K}
\end{equation*}
and so $\mathsf{QVar}(\mathbf{Sch}(\vdash_{*}))\subseteq \mathbf{K}$. But then by completeness:
\begin{equation*}
    \mathsf{Log}(\mathbf{K})\subseteq \mathsf{Log}(\mathsf{QVar}(\mathbf{Sch}(\vdash_{*}))) = \mathbf{Sch}(\vdash_{*}).
\end{equation*}
This shows that $\mathbf{Sch}$ is the greatest $\zeta$-companion, as was to show.\end{proof}

Hence, using the same arguments as in Corollary \ref{Greatest modal companion in the case of GMT} we can now derive the following.

\begin{corollary}
    If $\langle \zeta,f\rangle$ is a sober translation, then a logic ${\vdash_{t}}\in \Lambda(\vdash_{\classY})$ is $f$-maximal if and only if it is the greatest modal companion of $\rho(\vdash_{t})$.
\end{corollary}

These results also allow us to re-evaluate what makes the case of the Blok-Esakia isomorphism special. Given we know that for any sober translation we have a greatest companion, given by the schematic fragment of the least companion (for example), it cannot be the \textit{existence} of such a greatest companion, nor, as just shown, its semantics. Rather, it appears that what is special about this is that the $\mathsf{Grz}$ axiom \textit{axiomatises} the greatest companion, forcing the join of two schematic fragments to again be a schematic fragment, something which does not seem guaranteed in general, even for sober translations. As such, we introduce the following definition, to demarcate this special property of the GMT translation.

\begin{definition}
Let $\overline{\zeta}$ be a sober translation. We say that this is a \textit{BE-translation} if there is some logic ${\vdash_{*}}\in \Lambda(\vdash_{\classY})$ such that for each ${\Vdash}\in \Lambda(\vdash_{\classY})$, ${\Vdash}\in \Lambda(\vdash_{*})$ if and only if there is some ${\vdash^{*}}\in \Lambda(\vdash_{\classX})$ such that $\mathbf{Sch}(\vdash^{*})={\Vdash}$.
\end{definition}

As discussed in pp.\pageref{Maximal Grz-companion}, the GMT translation and  its close progeny (the translaton for bi-intuitionistic logic, amongst others), are examples of BE-translations. In Figure \ref{fig:typesoftranslationsandexamples}, we summarise the types of translations we have encountered.

\begin{figure}[h]
    \centering
\tikzset{every picture/.style={line width=0.75pt}} 

\begin{tikzpicture}[x=0.75pt,y=0.75pt,yscale=-1,xscale=1]

\draw   (259.98,308.14) .. controls (163.06,307.12) and (85.21,238.03) .. (86.09,153.81) .. controls (86.97,69.6) and (166.25,2.15) .. (263.17,3.17) .. controls (360.09,4.18) and (437.95,73.27) .. (437.06,157.49) .. controls (436.18,241.7) and (356.9,309.15) .. (259.98,308.14) -- cycle ;
\draw   (260.45,263.52) .. controls (177.71,262.65) and (111.24,203.67) .. (112,131.77) .. controls (112.75,59.88) and (180.43,2.3) .. (263.17,3.17) .. controls (345.91,4.03) and (412.38,63.02) .. (411.62,134.91) .. controls (410.87,206.8) and (343.19,264.38) .. (260.45,263.52) -- cycle ;
\draw   (260.98,212.78) .. controls (194.36,212.09) and (140.85,164.6) .. (141.46,106.71) .. controls (142.06,48.83) and (196.56,2.47) .. (263.17,3.17) .. controls (329.79,3.86) and (383.3,51.35) .. (382.7,109.24) .. controls (382.09,167.12) and (327.59,213.48) .. (260.98,212.78) -- cycle ;
\draw   (262.09,106.45) .. controls (228.08,106.09) and (200.76,82.68) .. (201.06,54.16) .. controls (201.36,25.64) and (229.17,2.81) .. (263.17,3.17) .. controls (297.18,3.52) and (324.51,26.93) .. (324.21,55.45) .. controls (323.91,83.97) and (296.1,106.8) .. (262.09,106.45) -- cycle ;
\draw   (261.48,165.07) .. controls (210.02,164.53) and (168.69,127.85) .. (169.16,83.14) .. controls (169.63,38.43) and (211.72,2.63) .. (263.17,3.17) .. controls (314.63,3.7) and (355.96,40.39) .. (355.49,85.09) .. controls (355.02,129.8) and (312.93,165.61) .. (261.48,165.07) -- cycle ;

\draw (186.5,265) node [anchor=north west][inner sep=0.75pt]  [font=\small,color={rgb, 255:red, 87; green, 83; blue, 83 }  ,opacity=1 ] [align=left] {\textit{Contextual Translations}};
\draw (198,220) node [anchor=north west][inner sep=0.75pt]  [font=\small,color={rgb, 255:red, 87; green, 83; blue, 83 }  ,opacity=1 ] [align=left] {\textit{Selective Translations}};
\draw (158.5,138.5) node [anchor=north west][inner sep=0.75pt]   [align=left] {$\displaystyle K_{\neg \neg }$};
\draw (210,120) node [anchor=north west][inner sep=0.75pt]  [font=\small,color={rgb, 255:red, 87; green, 83; blue, 83 }  ,opacity=1 ] [align=left] {\textit{Sober Translations}};
\draw (209.5,60) node [anchor=north west][inner sep=0.75pt]  [font=\small,color={rgb, 255:red, 87; green, 83; blue, 83 }  ,opacity=1 ] [align=left] {\textit{BE Translations}};
\draw (240.5,28.5) node [anchor=north west][inner sep=0.75pt]   [align=left] {$\displaystyle GMT$};
\draw (182,167) node [anchor=north west][inner sep=0.75pt]  [font=\small,color={rgb, 255:red, 87; green, 83; blue, 83 }  ,opacity=1 ] [align=left] {\textit{S. Selective Translations}};
\draw (155.5,193.5) node [anchor=north west][inner sep=0.75pt]   [align=left] {$\displaystyle G$};

\end{tikzpicture}

    \caption{Types of Translations and Examples}
    \label{fig:typesoftranslationsandexamples}
\end{figure}
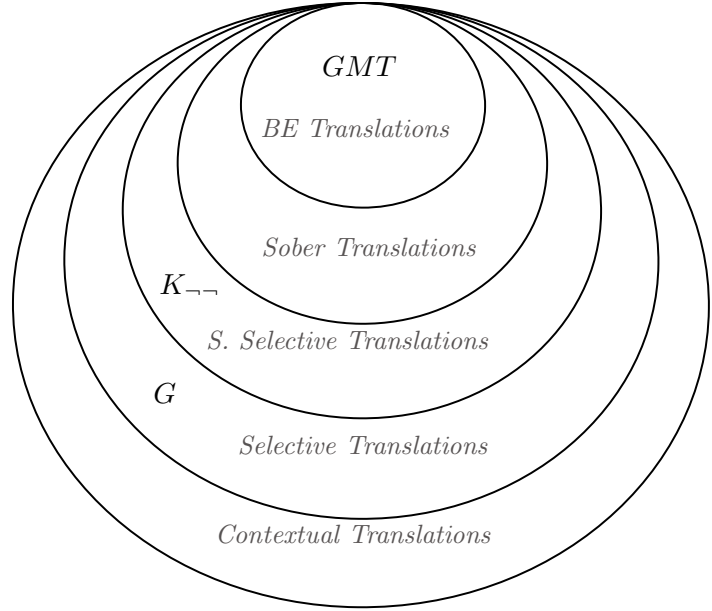

These translations are some of those we have previously encountered, and they also witness the strictness of some of these inclusions. We also note that the inclusion of selective in contextual translations is also strict: $K_{2}$, the translation between Kleene algebras and Distributive lattices mentioned in \cite{Moraschini2018}, is a contextual translation but not selective. The other translations appeared throughout the paper and their properties were discussed in \ref{Example of KGG and GMT translations 1}  \ref{KGG translation is strongly selective} \ref{Kolmogorov and GMT translations are selective}, and Proposition \ref{The GMT translation is sober}, and can briefly be summarises as follows:
\begin{itemize}
    \item The GMT translation is a BE-translation;
    \item The KGG translation is strongly selective but not sober;
    \item The Goldblatt translation is selective and faithful but not strongly selective.
\end{itemize}

We leave open whether there are any natural sober translations which are not BE-translations.

\section{Conclusion}\label{Section: Conclusion}

In this paper, we introduced the concept of Polyatomic logics, and initiated a generalised study of ``Blok-Esakia theory" for a large class of translations, and proved some results connecting the two concepts. We briefly summarise our main results. In Section \ref{Section: Basic Theory and Completeness of PA logic} we introduced $\mathsf{PAt}$-logics, and their algebraic semantics, $\mathsf{PAt}$-quasivarieties, and provided an algebraic completeness result. In Section \ref{Section: Translations and Polyatomic Logics} we introduced several classes of translations, namely selective, strongly selective and sober translations, and through a categorical analysis of the adjunction associated to this translation provided some necessary and sufficient conditions for translations to be in some of these classes. In Section \ref{Section: General Blok-Esakia Theory}, we made use of these various classes of translations to introduce a general Blok-Esakia theory, providing a notion of companion suitable for selective translations, and providing generalisations of the semantic and syntactic assignments $\rho$ and $\tau$ which occurr in the study of the modal companions of $\mathsf{IPC}$; we also proved that assuming a translation is sober, some properties, including FMP, tabularity and local tabularity are preserved by $\rho$. In Section \ref{Connecting Blok-Esakia and Polyatomic Logics} we proved the existence of a general conection between Polyatomic logics and general Blok-Esakia theory, including a variant of the Blok-Esakia isomorphism which connects the lattice of extensions of one system with the lattice of Polyatomic logics of the other system; we also showed that the notions of $f$-minimal and $f$-maximal logics, which appear naturally in the study of Polyatomic logics, correspond for sober translations, respectively, to the notions of greatest and least companions.

Beyond these specific results, in our view one of the main contributions of this article lies in the autonomous treatment of translations as a subject of mathematical analysis, continuing the work of \cite{Moraschini2018}. In particular, this perspective, which centers the relationship between the two logical systems in translation, as well as the extensions of each of these systems, allows us to ask a different set of questions with a very wide scope of application. We list some of these directions that we find most promising.

The most immediate theoretical question is a better grasp of the classes of selective, strongly selective and sober translations. Namely, it would be interesting to provide more syntactic characterisations of these classes, depending solely on the structure of the translation. Additionally, we think it could be valuable to seek a logical and algebraic meaning to other natural classes of adjunctions: for instance, a consequence of sobriety, through the left adjoint preserving equalizers and being faithful, is that the adjunction is comonadic and idempotent. Can we characterise the translations which correspond to such adjunctions through syntactic or algebraic properties? Can we characterise in those terms those translations where the right adjoint has itself a right adjoint (such as is the case with the KGG translation)?

Moreover, it would be interesting to understand how the results of the present paper fare when some (or all) of the restrictions imposed here are lifted: unariness of the translation, working with quasivarieties instead of generalized quasivarieties, algebraizability of the logics when defining $\mathsf{PAt}$-variants, amongst others. Moreover, a systematic study of the kinds of properties of the logical systems which are generally preserved by the semantic operators associated to the adjunction, namely through the operator $\tau$ as defined, would be very valuable. Additionally,  understanding the properties inherited by the $\mathsf{PAt}$-variant of a logic would also fall within this general direction of preservation theorems.

In a different direction, it would be interesting to apply the tools of Polyatomic logics to concrete settings. We expect that in settings like the Goldblatt translation, relating orthologic and the KTB modal logic system, where the Blok-Esakia isomorphism was shown to fail and all forms of classic Blok-Esakia theory were shown to be out of reach, these techniques could prove useful in clarifying the relationship between these logical systems. Likewise, in the context of several logics not closed under uniform substitution -- such as Buss' logic of provability, two-dimensional modal logics, Dynamic Epistemic Logics, amongst others -- it would be useful to see if the notions of Polyatomic logics can be fruitfully applied in providing rich algebraic semantics, and developing the metalogical properties of these systems.

Finally, it would be interesting to explore to what extent the connections laid out between Blok-Esakia theory and Polyatomic logics fare in broader terms. For instance, it was shown that $f$-minimal and $f$-maximal logics correspond to least and greatest companions. Does this correspondence extend to all modal companions/$f$-variants? These and the above questions are left open for the moment.

\section{Acknowledgements}

I would like to thank Nick Bezhanishvili and Tommaso Moraschini, under whose supervision the majority of the work above was developed, for their support and advice. I am also thankful to two anonymous reviewers for their careful reading, comments and suggestions, which greatly improved the presentation and structure of this paper.

\printbibliography[
    heading=bibintoc,
    title={Bibliography}
]

@misc{stronkowskiblokesakia,
  doi = {10.48550/ARXIV.1810.09286},
  url = {https://arxiv.org/abs/1810.09286},
  author = {Stronkowski,  Michał M.},
  title = {On the Blok-Esakia theorem for universal classes},
  publisher = {arXiv},
  year = {2018},
  copyright = {arXiv.org perpetual,  non-exclusive license}
}

@BOOK{Ciardelli2018-ql,
  title     = "Inquisitive Semantics",
  author    = "Ciardelli, Ivano and Groenendijk, Jeroen and Roelofsen, Floris",
  publisher = "Oxford University Press",
  series    = "Oxford Surveys in Semantics and Pragmatics",
  month     =  nov,
  year      =  2018,
  address   = "London, England"
}

@incollection{Wolter2014,
  doi = {10.1007/978-94-017-8860-1_5},
  url = {https://doi.org/10.1007/978-94-017-8860-1_5},
  year = {2014},
  publisher = {Springer Netherlands},
  pages = {99--118},
  author = {Frank Wolter and Michael Zakharyaschev},
  title = {On the Blok-Esakia Theorem},
  booktitle = {Leo Esakia on Duality in Modal and Intuitionistic Logics}
}

@article{Banaschewski1996,
author = {Banaschewski, B. and Pultr, A.},
journal = {Cahiers de Topologie et Géométrie Différentielle Catégoriques},
language = {eng},
number = {1},
pages = {41-60},
publisher = {Dunod éditeur, publié avec le concours du CNRS},
title = {Booleanization},
url = {http://eudml.org/doc/91572},
volume = {37},
year = {1996},
}

@misc{quadrellaronakov,
  doi = {10.48550/ARXIV.2210.06047},
  url = {https://arxiv.org/abs/2210.06047},
  author = {Nakov,  Georgi and Quadrellaro,  Davide Emilio},
  title = {Algebraizable Weak Logics},
  publisher = {arXiv},
  year = {2022},
  copyright = {Creative Commons Attribution Non Commercial Share Alike 4.0 International}
}

@article{Chagrov1992,
  doi = {10.1007/bf00370331},
  url = {https://doi.org/10.1007/bf00370331},
  year = {1992},
  publisher = {Springer Science and Business Media {LLC}},
  volume = {51},
  number = {1},
  pages = {49--82},
  author = {Alexander Chagrov and Michael Zakharyashchev},
  title = {Modal companions of intermediate propositional logics},
  journal = {Studia Logica}
}

@BOOK{Awodey2010-eu,
  title     = "Category Theory",
  author    = "Awodey, Steve",
  publisher = "Oxford University Press",
  series    = "Oxford Logic Guides",
  edition   =  2,
  month     =  jun,
  year      =  2010,
  address   = "London, England"
}

@book{MacLane1978,
  doi = {10.1007/978-1-4757-4721-8},
  url = {https://doi.org/10.1007/978-1-4757-4721-8},
  year = {1978},
  publisher = {Springer New York},
  author = {Saunders Mac Lane},
  title = {Categories for the Working Mathematician}
}

@article{aqvisttwodimensional,
 ISSN = {00223611, 15730433},
 URL = {http://www.jstor.org/stable/30226969},
 author = {Lennart Åqvist},
 journal = {Journal of Philosophical Logic},
 number = {1},
 pages = {1--76},
 publisher = {Springer},
 title = {Modal Logic with Subjunctive Conditionals and Dispositional Predicates},
 urldate = {2023-03-21},
 volume = {2},
 year = {1973}
}

@article{daviesandhumberstone,
 ISSN = {00318116, 15730883},
 URL = {http://www.jstor.org/stable/4319391},
 author = {Martin Davies and Lloyd Humberstone},
 journal = {Philosophical Studies: An International Journal for Philosophy in the Analytic Tradition},
 number = {1},
 pages = {1--30},
 publisher = {Springer},
 title = {Two Notions of Necessity},
 urldate = {2023-03-21},
 volume = {38},
 year = {1980}
}

@article{Buss1990,
  doi = {10.1305/ndjfl/1093635417},
  url = {https://doi.org/10.1305/ndjfl/1093635417},
  year = {1990},
  month = mar,
  publisher = {Duke University Press},
  volume = {31},
  number = {2},
  author = {Samuel R. Buss},
  title = {The modal logic of pure provability.},
  journal = {Notre Dame Journal of Formal Logic}
}

@InCollection{sep-dynamic-epistemic,
	author       =	{Baltag, Alexandru and Renne, Bryan},
	title        =	{{Dynamic Epistemic Logic}},
	booktitle    =	{The {Stanford} Encyclopedia of Philosophy},
	editor       =	{Edward N. Zalta},
	howpublished =	{\url{https://plato.stanford.edu/archives/win2016/entries/dynamic-epistemic/}},
	year         =	{2016},
	edition      =	{{W}inter 2016},
	publisher    =	{Metaphysics Research Lab, Stanford University}
}

@article{Yang2016,
  doi = {10.1016/j.apal.2016.03.003},
  url = {https://doi.org/10.1016/j.apal.2016.03.003},
  year = {2016},
  month = jul,
  publisher = {Elsevier {BV}},
  volume = {167},
  number = {7},
  pages = {557--589},
  author = {Fan Yang and Jouko V\"{a}\"{a}n\"{a}nen},
  title = {Propositional logics of dependence},
  journal = {Annals of Pure and Applied Logic}
}

@incollection{Grilletti2022,
  doi = {10.1007/978-3-030-98479-3_15},
  url = {https://doi.org/10.1007/978-3-030-98479-3_15},
  year = {2022},
  publisher = {Springer International Publishing},
  pages = {297--322},
  author = {Gianluca Grilletti and Davide Emilio Quadrellaro},
  title = {Lattices of~Intermediate Theories via~Ruitenburg's Theorem},
  booktitle = {Lecture Notes in Computer Science}
}

@BOOK{Gorbunov1998-sh,
  title     = "Algebraic theory of quasivarieties",
  author    = "Gorbunov, Viktor A",
  publisher = "Kluwer Academic/Plenum",
  series    = "Siberian School of Algebra and Logic",
  edition   =  1998,
  month     =  sep,
  year      =  1998,
  address   = "New York, NY",
  language  = "en"
}

@article{bezhanishvili_grilletti_quadrellaro_2021, 
title={An Algebraic Approach to Inquisitive and $\mathtt {DNA}$ -Logics},
journal={The Review of Symbolic Logic}, 
publisher={Cambridge University Press}, 
author={Bezhanishvili, Nick and Grilletti, Gianluca and Quadrellaro, Davide Emilio},
year={2021},
pages={1–41}
}

@incollection{Bezhanishvili2019_grilletti_holliday,
  doi = {10.1007/978-3-662-59533-6_3},
  year = {2019},
  publisher = {Springer Berlin Heidelberg},
  pages = {35--52},
  author = {Bezhanishvili, Nick and Grilletti, Gianluca and Holliday, Wesley H. },
  title = {Algebraic and Topological Semantics for Inquisitive Logic via Choice-Free Duality},
  booktitle = {Logic,  Language,  Information,  and Computation}
}

@phdthesis{Blok1976VarietiesOI,
  title={Varieties of interior algebras},
school ={Universiteit van Amsterdam},
  author={Willem J. Blok},
  year={1976}
}

@book{BurrisSankappanavar,
    author = {Burris, Stanley and Sankappanavar,H.P. },
    year = {1981},
    title = {A Course in Universal Algebra},
    publisher = {Springer},
    address = {New York}
}

@BOOK{Chagrov1997-cr,
  title     = {Modal Logic},
  author    = {Chagrov, Alexander and Zakharyaschev, Michael},
  publisher = {Clarendon Press},
  series    = {Oxford Logic Guides},
  year      =  {1997},
  address   = {Oxford, England}
}

@article{Ciardelli2010,
  doi = {10.1007/s10992-010-9142-6},
  year = {2010},
  publisher = {Springer Science and Business Media {LLC}},
  volume = {40},
  number = {1},
  pages = {55--94},
  author = {Ciardelli, Ivano and Roelofsen, Floris},
  title = {Inquisitive Logic},
  journal = {Journal of Philosophical Logic}
}

@mastersthesis{antoniocleani,
author = {Cleani, Antonio Maria},
year = {2021},
title = {Translational Embeddings via Stable Canonical Rules},
school = {Universiteit van Amsterdam},
adress = {Amsterdam}
}

@book{Davey2002-lr,
  title = {Introduction to Lattices and Order},
  author    = {Davey, Brian A. and Priestley, Hilary A.},
  publisher = {Cambridge University Press},
  pages     = {175--200},
  year      =  {2002},
  address   = {Cambridge}
}

@article{Dukarm1988,
  doi = {10.4064/cm-55-1-11-17},
  year = {1988},
  publisher = {Institute of Mathematics,  Polish Academy of Sciences},
  volume = {55},
  number = {1},
  pages = {11--17},
  author = {J. J. Dukarm},
  title = {Morita equivalence of algebraic theories},
  journal = {Colloquium Mathematicum}
}

@book{Esakiach2019HeyAlg,
	title = {Heyting Algebras: Duality Theory},
	publisher = {Springer},
	author = {Esakia,Leo},
	editor = {Bezhanishvili, Guram and Holliday, Wesley},
	note = {English translation of the original 1985 book},
	year = {2019}
}

@inproceedings{esakiapaperatconference,
  author = {Leo Esakia},
  title  = {On Modal Companions of Superintuitionistic Logics},
  organization = {VII Soviet symposium on logic (Kiev,76)},
year = {1976},
 pages = {135-136}
}

@book{Font2016-dk,
  title     = {Abstract algebraic logic -- An introductory textbook},
  author    = {Font, Josep Maria},
  publisher = {College Publications},
  address = {London},
  year      = {2016}
}

@article{freyd,
author = {Freyd, Peter},
journal = {Colloquium Mathematicae},
language = {eng},
number = {1},
pages = {89-106},
title = {Algebra valued functors in general and tensor products in particular},
volume = {14},
year = {1966},
}

@article{Fussner2021,
  doi = {10.4204/eptcs.343.3},
  year = {2021},
  publisher = {Open Publishing Association},
  volume = {343},
  pages = {37--49},
  author = {Fussner, Wesley and St. John, Gavin},
  title = {Negative Translations of Orthomodular Lattices and Their Logic},
  journal = {Electronic Proceedings in Theoretical Computer Science}
}

@online{gehrkevangoolnewbook,
  author = {Gehrke, Mai and van Gool,  Sam},
  title = {Topological duality for distributive lattices,  and applications},
  eprinttype = {arXiv},
  eprint = {arXiv:2203.03286
},
version = {2},
  year = {2022},
}

@article{Gentzen1936,
  doi = {10.1007/bf01565428},
  year = {1936},
  publisher = {Springer Science and Business Media {LLC}},
  volume = {112},
  number = {1},
  pages = {493--565},
  author = {Gerhard Gentzen},
  title = {Die Widerspruchsfreiheit der reinen Zahlentheorie},
  journal = {Mathematische Annalen}
}

@article{GIRARD19871,
title = {Linear logic},
journal = {Theoretical Computer Science},
volume = {50},
number = {1},
pages = {1-101},
year = {1987},
issn = {0304-3975},
doi = {https://doi.org/10.1016/0304-3975(87)90045-4},
author = {Jean-Yves Girard}
}

@mastersthesis{minhatesedemestrado,
  author  = "Rodrigo N. Almeida",
  title   = "Polyatomic Logics and Generalised Blok-Esakia Theory with Applications to Orthologic and KTB",
  school  = "University of Amsterdam",
  year    = "2022"
}

@incollection{davis_1990, 
title={Eine Interpretation des intuitionistischen Aussagenkalküls},
author={Kurt Gödel},
volume={55},
DOI={10.2307/2274985}, 
number={1}, 
journal={Journal of Symbolic Logic}, publisher={Cambridge University Press}, editor={Davis, Martin}, 
year={1990},
pages={346–346}
}

@incollection{esakiapaperabouts41,
author = {Leo Esakia},
title = {On the theory of modal and superintuitionistic logics},
booktitle = {Logical Inference},
year = {1979},
editor = {V.A. Smirnov},
address = {Moscow},
pages = {147-171},
publisher = {Nauka}
}

@phdthesis{moraschiniphd,
  author       = {Tommaso Moraschini}, 
  title        = {Investigations into the role of translations in abstract algebraic logic},
  school       = {Universitat de Barcelona},
  year         = {2016}
}

@article{godelintuition,
author={Kurt Godel},
title={Eine Interpretation des intuitionischen Aussagenkalkuls},
journal={Ergebnisse eines Mathematischen Kolloquiums},
volume={4},
pages={39--40},
year={1933}
}

@article{Dummett1959,
  doi = {10.1002/malq.19590051405},
  url = {https://doi.org/10.1002/malq.19590051405},
  year = {1959},
  publisher = {Wiley},
  volume = {5},
  number = {14-24},
  pages = {250--264},
  author = {M. A. E. Dummett and E. J. Lemmon},
  title = {Modal Logics Between S 4 and S 5},
  journal = {Zeitschrift f\"{u}r Mathematische Logik und Grundlagen der Mathematik}
}

@article{Goldblatt1974,
  doi = {10.1007/bf00652069},
  year = {1974},
  publisher = {Springer Science and Business Media {LLC}},
  volume = {3},
  number = {1-2},
  pages = {19--35},
  author = {Goldblatt, Robert},
  title = {Semantic analysis of orthologic},
  journal = {Journal of Philosophical Logic}
}

@incollection{Grilletti2022-oh,
  title     = {Lattices of intermediate theories via ruitenburg's theorem},
  booktitle = {Lecture Notes in Computer Science},
  author    = {Grilletti, Gianluca and Quadrellaro, Davide Emilio},
  publisher = {Springer International Publishing},
  pages     = {297--322},
  series    = {Lecture notes in computer science},
  year      =  {2022},
  address   = {Cham}
}

@ARTICLE{Grzegorczyk1967-bb,
  title     = {Some relational systems and the associated topological spaces},
  author    = {Grzegorczyk, Andrzej},
  journal   = {Fundamenta Mathematicae},
  publisher = {Institute of Mathematics, Polish Academy of Sciences},
  volume    =  {60},
  number    =  {3},
  pages     = {223--231},
  year      =  {1967}
}

@article{Hartonas2020-kn,
  title     = {Modal translation of substructural logics},
  author    = {Hartonas, Chrysafis},
  journal   = {Journal of Applied Non-Classical Logics},
  publisher = {Informa UK Limited},
  volume    =  {30},
  number    =  {1},
  pages     = {16--49},
  year      =  {2020}
  }

@incollection{kolmogorovexcludedmidle,
  title = {On the Principle of the Excluded Middle},
  author = {Kolmogorov, Andrey},
  booktitle = {From Frege to G{\"o}del : a source book in mathematical logic,        1879-1931},
  editor    = {van Heijenoort, Jean},
  publisher = {Harvard University Press},
  series    = {Source Books in the History of the Sciences},
  year      =  {1990},
  address   = {London, England}
  }

@incollection{McKenzie2017,
  doi = {10.1201/9780203748671-10},
  year = {2017},
  month = oct,
  publisher = {Routledge},
  pages = {211--243},
  author = {Ralph McKenzie},
  title = {An algebraic version of categorical equivalence for varieties and more general algebraic categories},
  booktitle = {Logic and algebra}
}

@article{Moraschini2018,
title={A Logical and Algebraic Characterisation of Adjunctions between Generalised Quasi-Varieties},
volume={83},
number={3}, 
journal={The Journal of Symbolic Logic}, 
publisher={Cambridge University Press}, 
author={Moraschini, Tommaso}, year={2018},
pages={899–919}
}

@article{Tur2008,
  year = {2008},
  publisher = {Duke University Press},
  volume = {49},
  number = {2},
  author = {Tur, Juan Soliveres  and Vidal, Juan Climent},
  title = {Functors of Lindenbaum-Tarski,  Schematic Interpretations,  and Adjoint Cylinders between Sentential Logics},
  journal = {Notre Dame Journal of Formal Logic}
}

\end{document}